\documentclass{amsart} 

\usepackage{amsmath, amssymb, latexsym, amscd, mathabx}
\usepackage{exscale, cite, eps fig, graphics}

\usepackage[colorlinks]{hyperref}

\renewcommand{\geq}{\geqslant}
\renewcommand{\leq}{\leqslant}

\renewcommand{\l}{\langle}
\renewcommand{\r}{\rangle}

\renewcommand{\c}{c_{\textrm{\tiny ww}}}
\newcommand{\cc}{c_{\emph{\tiny ww}}}
\renewcommand{\k}{\kappa}
\renewcommand{\t}{t_{\emph{\Tiny LWP}}}

\renewcommand{\L}{\mathcal{L}}
\newcommand{\M}{\mathcal{M}}
\newcommand{\x}{\xi_0}
\newcommand{\e}{\mathbf{e}}
\newcommand{\f}{\mathbf{f}}
\newcommand{\g}{\mathbf{g}}
\newcommand{\p}{\boldsymbol{\phi}}
\newcommand{\q}{\boldsymbol{\psi}}
\renewcommand{\u}{\mathbf{u}}
\renewcommand{\v}{\mathbf{v}}

\newcommand{\hu}{\begin{pmatrix}\eta \\ u\end{pmatrix}}
\newcommand{\hv}{\begin{pmatrix}\zeta \\ v\end{pmatrix}}

\newtheorem{theorem}{Theorem}[section]
\newtheorem{lemma}[theorem]{Lemma}

\newtheorem{corollary}[theorem]{Corollary}

\newtheorem*{main-theorem}{Main Theorem}
\newtheorem*{remark*}{Remark}
\numberwithin{equation}{section}

\title[Instability in a shallow water model]{Modulational instability in \\ a full-dispersion shallow water model}

\author[Hur]{Vera~Mikyoung~Hur}
\address{Department of Mathematics, University of Illinois at Urbana-Champaign, Urbana, IL 61801 USA}
\email{verahur@math.uiuc.edu}

\author[Pandey]{Ashish~Kumar~Pandey}
\email{akpande2@illinois.edu}  

\date{\today}

\begin{document}

\maketitle

\begin{abstract}
We propose a shallow water model which combines the dispersion relation of water waves and the Boussinesq equations, and which extends the Whitham equation to permit bidirectional propagation. We establish that its sufficiently small, periodic wave train is spectrally unstable to long wavelength perturbations, provided that the wave number is greater than a critical value, like the Benjamin-Feir instability of a Stokes wave. We verify that the associated linear operator possesses infinitely many collisions of purely imaginary eigenvalues, but they do not contribute to instability away from the origin in the spectral plane to the leading order in the amplitude parameter. We discuss the effects of surface tension on the modulational instability. The results agree with those from formal asymptotic expansions and numerical computations for the physical problem. 
\end{abstract}

\section{Introduction}\label{sec:intro}

In the 1960s, Benjamin and Feir \cite{BF, BH} and Whitham \cite{Whitham1967} discovered that a Stokes wave\footnote{It is a matter of experience that waves typically seen in the ocean or a lake are approximately periodic and they propagate over a long distance practically at a a constant velocity without change of form. Stokes in his 1847 memoir (see also \cite{Stokes}) made many contributions about waves of the kind, for instance, observing that crests would be sharper and thoughts flatter as the amplitude increases, and a wave of greatest possible height would allow stagnation at the crest with a $120^\circ$ corner.} would be unstable to long wavelength perturbations --- namely, the {\em Benjamin-Feir} or {\em modulational instability} --- provided that 
\[
\kappa h>1.363\dots,
\]
where $\k$ denotes the carrier wave number, and $h$ is the undisturbed water depth. Corroborating results arrived about the same time, but independently, by Lighthill \cite{Lighthill} and Zakharov \cite{Zakharov-ww}, among others. We encourage the interested reader to \cite{ZO} for more about the early history. In the 1990s, Bridges and Mielke \cite{BM1995} addressed the corresponding spectral instability in a rigorous manner. By the way, it is difficult to justify the 1960s theory in a functional analytic setting. But the proof leaves some important issues  open, such as the stability and instability away from the origin in the spectral plane. The governing equations of the water wave problem are complicated, and they as a rule prevent a detailed account. One may resort to approximate models to gain insights.

\subsection*{Whitham's shallow water equation}

As Whitham \cite{Whitham} emphasized, ``the breaking phenomenon is one of the most intriguing long-standing problems of water wave theory." The {\em nonlinear shallow water equations},
\begin{equation}\label{E:shallow}
\begin{aligned}
&\partial_t\eta+\partial_x(u(h+\epsilon\,\eta))=0,\\
&\partial_tu+g\partial_x\eta+\epsilon\,u\partial_xu=0,
\end{aligned}
\end{equation}
approximate the physical problem when the characteristic wavelength is of a larger order than the undisturbed fluid depth, and they explain wave breaking. That is, the solution remains bounded, whereas its slope becomes unbounded in finite time. Here $t\in\mathbb{R}$ is proportional to elapsed time, and $x\in\mathbb{R}$ is the spatial variable in the primary direction of wave propagation; $\eta=\eta(x,t)$ represents the surface displacement from the depth $=h$, and $u=u(x,t)$ is the fluid particle velocity at the rigid flat bottom; $g$ denotes the constant due to gravitational acceleration, and $\epsilon$ is the dimensionless nonlinearity parameter; see \cite{Lannes}, for instance, for details. Throughout, $\partial$ means partial differentiation. Note that the phase speed for the linear part of \eqref{E:shallow} is $\sqrt{gh}$ for any wave number, whereas the speed of a $2\pi/\k$ periodic wave near the rest state of water (see \cite{Whitham}, for instance) is
\begin{equation}\label{def:c1}
\c(\k)=\sqrt{\frac{g\tanh(\k h)}{\k}}.
\end{equation}

But the shallow water theory goes too far. It predicts that all solutions carrying an increase of elevation break. Yet it is a matter of experience that some waves in water do not. Perhaps, the neglected dispersion effects inhibit wave breaking. 

But, including some\footnote{In the long wave limit as $\kappa h\to0$, one may expand the right side of \eqref{def:c1} to find 
\[
\c(\k)=\sqrt{gh}\Big(1-\frac16\k^2h^2\Big)+O(\k^4h^4).
\]}
dispersion effects, the Korteweg-de Vries (KdV) equation
\[
\partial_t\eta+\sqrt{gh}\Big(1+\frac16 h^2\partial_x^2\Big)\partial_x\eta
+\frac32\sqrt{\frac{g}{h}}\epsilon\,\eta\partial_x\eta=0
\]
goes too far and predicts that no solutions break. To recapitulate, one necessitates some dispersion effects to satisfactorily explain wave breaking, but the dispersion of the KdV equation seems too strong\footnote{This is not surprising because the phase speed $=\sqrt{gh}(1-\frac16\kappa^2h^2)$ for the KdV equation poorly approximates that for water waves when $\kappa h\gg1$.}.

Whitham noted that ``it is intriguing to know what kind of simpler mathematical equations (than the governing equations of the water wave problem) could include" the breaking effects, and he put forward (see \cite{Whitham}, for instance) 
\begin{equation}\label{E:Whitham}
\partial_t\eta+\c(|\partial_x|)\partial_x\eta
+\frac32\sqrt{\frac{g}{h}}\epsilon\,\eta\partial_x\eta=0, 
\end{equation}
where $\c(|\partial_x|)$ is a Fourier multiplier operator, defined as 
\[
\widehat{\c(|\partial_x|)v}(\k)=\c(|\k|)\widehat{v}(\k),
\]
and $\c$ is in \eqref{def:c1}. Here and elsewhere, the circumflex means the Fourier transform. It combines the dispersion relation of the unidirectional propagation of water waves and a nonlinearity of the shallow water equations. In a small amplitude and long wavelength regime, such that $\epsilon=\kappa^2h^2\ll1$, the {\em Whitham equation} agrees with the KdV equation up to the order of $\epsilon$ over a relevant time interval; see \cite{Lannes}, for instance, for details. But it may offer an improvement over the KdV equation for short waves. Whitham conjectured the wave breaking for \eqref{E:Whitham} and \eqref{def:c1}. One of the authors \cite{Hur-breaking} settled it.

\subsection*{The full-dispersion shallow water equations}

In recent years, the Whitham equation gathered renewed attention because of its ability to explain high frequency phenomena of water waves. In particular, one of the authors \cite{HJ2} demonstrated that a sufficiently small, $2\pi/\k$-periodic wave train of \eqref{E:Whitham} and \eqref{def:c1} be spectrally unstable to long wavelength perturbations, provided that $\k h>1.145\dots$, like the Benjamin-Feir instability of a Stokes wave, whereas it be stable to square integrable perturbations otherwise. 

But the Whitham equation does not include collisions of eigenvalues away from the origin in the spectral plane, which numerical computations \cite{McLean1981, McLean1982, MS1986}, for instance, suggest to develop instability in the water wave problem. 
This motivates us to propose the {\em full-dispersion shallow water equations}, 
\begin{equation}\label{E:BW}
\begin{aligned}
&\partial_t\eta+\partial_x(u(1+\epsilon\eta))= 0,\\
&\partial_tu+\c^2(|\partial_x|)\partial_x\eta+\epsilon\,u\partial_xu=0,
\end{aligned}
\end{equation}
where $\c$ is in \eqref{def:c1}. They combine the dispersion relation of water waves and the nonlinear shallow water equations, and they extend the Whitham equation to permit bidirectional propagation. Moreover, \eqref{E:BW} and \eqref{def:c1} exhibit the spectral behavior of the physical problem (see \cite{Whitham}, for instance, for details). 

When $\epsilon=\kappa^2h^2\ll 1$, the full-dispersion shallow water equations agree with a variant\footnote{They do not appear explicitly in the work of Boussinesq. But (280) in \cite{Bnesq1877}, for instance, after several ``higher order terms" drop out,  becomes equivalent to \eqref{E:Bnesq}.} of the Boussinesq equations,
\begin{equation}\label{E:Bnesq}
\begin{aligned}
&\partial_t\eta+\partial_x(u(h+\epsilon\,\eta))=0,\\
&\Big(1-\frac13h^2\partial_x^2\Big)\partial_tu+g\partial_x\eta+\epsilon\,u\partial_xu=0,
\end{aligned}
\end{equation}
up to the order of $\epsilon$. Hence they as well go by the name of the {\em Boussinesq-Whitham equations}. Indeed, one may modify the argument in \cite[Section~7.4.5]{Lannes} to verify that the solutions of \eqref{E:BW}  and \eqref{E:Bnesq} exist, where $\c$ is in \eqref{def:c1}, and they converge to the solutions of the water wave problem up to terms of order $\epsilon$ over a relevant time interval. The global-in-time well-posedness for \eqref{E:Bnesq} was established in \cite{Schonbek} and \cite{Amick}, for instance, whereas the wave breaking for \eqref{E:BW} and \eqref{def:c1} was in \cite{HT2}, under an assumption that $\eta$ is considerably smaller than $h$. 

\subsection*{Modulational instability}

In the past decades, much research effort aimed at translating Whitham's formal modulation theory (see \cite{Whitham}, for instance) into analytical stability results. It would be impossible to do justice to all advances in the vein. We encourage the interested reader to \cite{BHJ} and references therein. But the arguments as a rule make strong use of Evans function techniques and other ODE methods. Hence they are not directly applicable to \eqref{E:Whitham} or \eqref{E:BW}, which involve a nonlocal\footnote{
Indeed, $\c(\k)$ and, hence, $\c^2(\k)$ are not polynomials of $i\k$.} operator. The authors and collaborators \cite{BHV, HJ2, HJ3, HP1} (see also \cite{BHJ}) instead worked out the corresponding long wavelength perturbation in a rigorous manner, whereby they successfully determined modulational stability and instability, for a broad class of nonlinear dispersive equations permitting nonlocal operators. In Section~\ref{sec:existence} and Section~\ref{sec:MI}, we take matters further for the full-dispersion shallow water equations. Corollary~\ref{cor:kc} states that a sufficiently small, $2\pi/\k$-periodic wave train of \eqref{E:BW} and \eqref{def:c1} is modulationally unstable, provided that 
\[
\k h>1.610\dots,
\] 
and it is stable to square integrable perturbations in the vicinity of the origin in the spectral plane otherwise. Note that the critical wave number compares reasonably well with what is in \cite{BH, Whitham1967} and \cite{BM1995}. 

The proof follows along the same line as those in \cite{HJ2, HP1}, making a lengthy and complicated, but explicit, spectral perturbation for the associated linearized operator. For the zero Floquet exponent, we distinguish four eigenvalues at the origin and calculate the small amplitude expansion of the associated eigenfunctions. It seems impossible to find the eigenfunctions explicitly without recourse to the small amplitude theory. On the other hand, \eqref{E:BW} and \eqref{def:c1} lose relevances for large amplitude waves. For small values of the Floquet exponent, we then determine four eigenvalues near the origin in the spectral plane, up to the quadratic order in the Floquet exponent and linear order in the amplitude parameter, whereby we derive a modulational instability index as a function of the carrier wave number.

We do not expect to uniquely determine higher order terms in the eigenfunction expansion. To compare, one may find the Whitham eigenfunctions to any order. Fortuitously, we detect modulational instability at the linear order in the amplitude parameter. We are able to calculate the quadratic order terms in the eigenfunction expansion. But they are bulky, so that the index formulae become unwieldy.

\subsection*{Comparison to other Boussinesq-Whitham models}

Perhaps, the best known among Boussinesq's equations of the shallow water theory is
\begin{equation}\label{E:Bnesq1}
\partial_t^2\eta=gh\Big(1+\frac13h^2\partial_x^2\Big)\partial_x^2\eta+\frac32\epsilon\partial_x^2(\eta^2).
\end{equation}
Including the full dispersion in water waves, one may follow Whitham's heuristics and replace the square of the phase speed~$=gh(1-\frac13\k^2h^2)$ by $\c^2(\k)$ in \eqref{def:c1}. The result becomes
\begin{equation}\label{E:BW1}
\partial_t^2\eta=\c^2(|\partial_x|)\partial_x^2\eta+\frac32\epsilon\partial_x^2(\eta^2).
\end{equation}
It is one of many which stake the claim to be the Boussinesq-Whitham equation. Unfortunately, \eqref{E:BW1} is not suitable to describing wave packet propagation. Indeed, the Cauchy problem for \eqref{E:BW1} is ill-posed in the periodic setting, implying that a negative constant solution is spectrally unstable however small it is; see \cite{DT} or Appendix~\ref{sec:BW1} for details.

Under the assumption $\partial_t\eta+\partial_x\eta=O(\epsilon)$, which by the way enforces unidirectional propagation, \eqref{E:Bnesq1} is formally equivalent to
\[
\partial_t^2\eta=gh\partial_x^2\eta+\frac13h^2\partial_x^2\partial_t^2\eta+\frac32\epsilon\partial_x^2(\eta^2)
\]
up to the order of $\epsilon$. Including the full dispersion in water waves, likewise, we arrive at
\begin{equation}\label{E:BW2}
\c^{-2}(|\partial_x|)\partial_t^2\eta=\partial_x^2\eta+\frac32\epsilon\partial_x^2(\eta^2).
\end{equation} 
The Cauchy problem for \eqref{E:BW2} is well-posed at least for short times. But it fails to predict modulational instability. Indeed, a sufficiently small, periodic wave train of \eqref{E:BW2} is stable to square integrable perturbations in the vicinity of the origin in the spectral plane for any wave number; see \cite{HP1} for details. 
In contrast, in Section~\ref{sec:MI} we establish the modulational instability for \eqref{E:BW}. 

Moreover, proposed in \cite{MKD} are
\begin{equation}\label{E:BW3}
\begin{aligned}
&\partial_t\eta+\c^2(|\partial_x|)\partial_xu+\epsilon\partial_x(u\eta)= 0,\\
&\partial_tu+\partial_x\eta+\epsilon\,u\partial_xu=0,
\end{aligned}
\end{equation}
as a Boussinesq-Whitham model. Note that \eqref{E:BW3} is formally equivalent to \eqref{E:BW} up to the order of $\epsilon$. But they are not suitable to explaining wave breaking. Indeed, to the best of the authors' knowledge, the well-posedness for \eqref{E:BW3} is not understood. In contrast, in Appendix~\ref{sec:LWP}, we establish the well-posedness for \eqref{E:BW} for short times.

To recapitulate, \eqref{E:BW} is preferred over other Boussinesq-Whitham models for the purpose of studying the breaking and stability of water waves.  

\subsection*{Stability and instability away from the origin}

The spectrum of the linear operator associated with \eqref{E:BW} and \eqref{def:c1} aligns with that for the water wave problem (see \cite{Whitham}, for instance, for details). In particular, it contains infinitely many collisions of purely imaginary eigenvalues. To compare, no Whitham eigenvalues collide other than at the origin. In the 1980s, McLean and collaborators \cite{McLean1981, McLean1982} (see also \cite{MS1986}) numerically approximated the spectrum for the physical problem, but in the infinite depth, whereby they argued that all colliding eigenvalues for the zero amplitude parameter would contribute to instability as the amplitude increases. Numerical computations in \cite{DO} and \cite{AN2014}, for instance, in the finite depth bear this out.

In Section~\ref{sec:HF}, we make an explicit spectral perturbation to demonstrate that all nonzero eigenvalues of the linear operator for \eqref{E:BW} and \eqref{def:c1} remain on the imaginary axis to the linear order in the small amplitude parameter. Consequently, the modulational instability dominates the spectral instability away from the origin. The result agrees with the numerical findings in \cite{MS1986, DO, AN2014}, for instance, for the physical problem. Indeed, some unstable eigenvalue away from the origin for the water wave problem grows like a quartic in the amplitude parameter; see \cite{MS1986}, for instance, for details. 

It is interesting to analytically calculate higher order terms in the eigenvalue expansion for \eqref{E:BW} and \eqref{def:c1}, and their contribution to stability; see Appendix~\ref{sec:2,0} for some details. It is interesting to numerically study the stability and instability away from the origin in the spectral plane, and the growth rates of unstable eigenvalues. 

\subsection*{Effects of surface tension}

In the presence of the effects of {\em surface tension} (see \cite{Whitham}, for instance),
\[
\sqrt{(g+T\k^2)\frac{\tanh(\k h)}{\k}}
\]
replaces \eqref{def:c1}, where $T$ is the coefficient of surface tension. 
In Section~\ref{sec:ST}, we adapt the argument in Section~\ref{sec:existence} and Section~\ref{sec:MI} to demonstrate that the capillary effects considerably alter the modulational instability in \eqref{E:BW} and \eqref{def:c1}. Specifically, in the $\k h$ and $\k\sqrt{T/g}$ plane, we determine the regions of stability and instability, whose boundaries are associated with an extremum of the group velocity, the resonance of short and long waves, the resonance of the fundamental mode and the second harmonic, and the resonance of the dispersion and nonlinear effects; see Figure~\ref{fig:ST} for details. The result agrees with those in \cite{Kawahara} and \cite{DR}, for instance, from formal asymptotic expansions for the physical problem. To compare, the Whitham equation fails to predict the limit of ``large surface tension;" see \cite{HJ3}, for instance, for details. Therefore, \eqref{E:BW} offers an improvement over the Whitham equation for gravity capillary waves.

\subsection*{Notation} 

For $p$ in the range $[1,\infty]$, let $L^p(\mathbb{R})$ consist of real or complex valued, Lebesgue measurable functions over $\mathbb{R}$ such that
\[
\|v\|_{L^p(\mathbb{R})}:=\Big(\frac{1}{\sqrt{2\pi}}\int_{-\infty}^\infty|v(x)|^p~dx\Big)^{1/p}<\infty
\qquad\text{if\quad$p<\infty$},
\]
and $\|v\|_{L^\infty(\mathbb{R})}:=\text{ess\,sup}_{x\in\mathbb{R}}|v(x)|<\infty$ if $p=\infty$. 

For $v\in L^1(\mathbb{R})$, the Fourier transform of $v$ is written $\widehat{v}$ and defined by
\[
\widehat{v}(\k)=\frac{1}{\sqrt{2\pi}}\int_{-\infty}^\infty v(x)e^{-i\k x}~dx.
\]
If $v\in L^2(\mathbb{R})$ then the Parseval theorem asserts that $\|\widehat{v}\|_{L^2(\mathbb{R})}=\|v\|_{L^2(\mathbb{R})}$. For $s\in\mathbb{R}$, let $H^s(\mathbb{R})$ consist of tempered distributions such that 
\[
\|v\|_{H^s(\mathbb{R})}:=\Big(\int_{-\infty}^\infty(1+|\k|^2)^s|\widehat{v}(\k)|^2~d\k\Big)^{1/2}<\infty.
\]
Let $H^\infty(\mathbb{R})=\bigcap_{s\in\mathbb{R}} H^s(\mathbb{R})$. 

Let $\mathbb{T}$ denote the unit circle in $\mathbb{C}$. We identify functions over $\mathbb{T}$ with $2\pi$ periodic functions over $\mathbb{R}$ via $v(e^{iz})=V(z)$ and, for simplicity of notation, we write $v(z)$ rather than $v(e^{iz})$. For $p\in[1,\infty]$, let $L^p(\mathbb{T})$ consist of real or complex valued, Lebesgue measurable, and $2\pi$ periodic functions over $\mathbb{R}$ such that 
\[
\|v\|_{L^p(\mathbb{T})}:=\Big(\frac{1}{2\pi}\int^\pi_{-\pi}|v(z)|^p~dz\Big)^{1/p}<\infty
\qquad\text{if\quad$p<\infty$},
\]
and $\|v\|_{L^\infty(\mathbb{T})}:=\text{ess\,sup}_{-\pi<z\leq\pi}|v(z)|<\infty$ if $p=\infty$. Let $H^1(\mathbb{T})$ consist of $L^2(\mathbb{T})$ functions whose derivatives are in $L^2(\mathbb{T})$. Let $H^\infty(\mathbb{T})=\bigcap_{k=0}^\infty H^k(\mathbb{T})$. 

For $v \in L^1(\mathbb{T})$, the Fourier series of $v$ is defined by
\[
\sum_{n\in \mathbb{Z}}\widehat{v}(n)e^{inz}, \qquad\text{where}\quad
\widehat{v}(n)=\frac{1}{2\pi}\int^\pi_{-\pi}v(z)e^{-inz}~dz.
\]
If $v\in L^2(\mathbb{T})$ then its Fourier series converges to $v$ pointwise almost everywhere. We define the $L^2(\mathbb{T})$ inner product as 
\[
\langle v_1,v_2\rangle_{L^2(\mathbb{T})}=\frac{1}{2\pi}\int^\pi_{-\pi} v_1(z)v_2^*(z)~dz
=\sum_{n\in\mathbb{Z}}\widehat{v_1}(n)\widehat{v_2}^*(n).
\]
Here and elsewhere, the asterisk means complex conjugation. The latter equality follows from the Parseval theorem.

We extend the above to product spaces in the usual manner. In particular, we define the Fourier series as $\hv \sim {\displaystyle \sum_{n\in\mathbb{Z}}\begin{pmatrix} \widehat{\zeta}(n) \\ \widehat{v}(n)\end{pmatrix}e^{inz}}$ and the $L^2(\mathbb{T})\times L^2(\mathbb{T})$ inner product as 
\begin{align*}
\Big\langle \begin{pmatrix}\zeta_1\\v_1\end{pmatrix},
\begin{pmatrix}\zeta_2\\v_2\end{pmatrix} \Big\rangle_{L^2(\mathbb{T})\times L^2(\mathbb{T})}
= &\langle \zeta_1,\zeta_2\rangle_{L^2(\mathbb{T})}+\langle v_1,v_2\rangle_{L^2(\mathbb{T})}\\
=&\frac{1}{2\pi}\int^\pi_{-\pi} (\zeta_1(z)\zeta_2^*(z)+v_1(z)v_2^*(z))~dz. \notag
\end{align*}

\section{Sufficiently small, periodic wave trains}\label{sec:existence}

We determine periodic wave trains of the full-dispersion shallow water equations, after normalization of parameters,
\begin{equation}\label{E:main}
\begin{aligned}
&\partial_t\eta+\partial_xu+\partial_x(u\eta)=0,\\
&\partial_tu+\c^2(|\partial_x|)\partial_x\eta+u\partial_xu=0.
\end{aligned}
\end{equation}
Here and in the sequel, (by abuse of notation)
\begin{equation}\label{def:c}
\c(\k)=\sqrt{\frac{\tanh\k}{\k}}.
\end{equation}
Indeed, $\epsilon=1$ turns \eqref{E:BW} to \eqref{E:main}; $\k h\mapsto \k$ and $gh=1$ turn \eqref{def:c1} to \eqref{def:c}. We then calculate their small amplitude expansion. 

\subsection{Properties of $\c$}\label{ss:c}

Note that $\c^2$ is even and real analytic, $\c^2(0)=1$, and it decreases to zero monotonically away from the origin. Indeed, 
\[
\tanh\k=\sum_{n=1}^\infty\frac{2^{2n}(2^{2n}-1)}{(2n)!}B_{2n}\k^{2n-1}
=\k-\frac13\k^3+\frac{2}{15}\k^5-\cdots\qquad\text{for $|\k|\ll1$},
\]
where $B_{2n}$ is the Bernoulli number. Since
\[
\frac{1}{1+|\k|}\leq \c^2(\k)\leq \frac{2}{1+|\k|}\qquad\text{pointwise in $\mathbb{R}$}
\]
by brutal force, $\c^2(|\partial_x|)$ may be regarded equivalent to $(1+|\partial_x|)^{-1}$ in the $L^2$-Sobolev space setting. In particular, $\c^2(|\partial_x|):H^s(\mathbb{R})\to H^{s+1}(\mathbb{R})$ for any $s\in\mathbb{R}$.  

\subsection{Periodic traveling waves}\label{ss:periodic}

By a traveling wave of \eqref{E:main}-\eqref{def:c}, we mean a solution which propagates at a constant velocity without change of form. That is, $\eta$ and $u$ are functions of $x-ct$ for some $c>0$, the wave speed. Under the assumption, we will go to a moving coordinate frame, changing $x-ct$ to $x$, whereby $t$ will disappear. The result becomes, by quadrature, 
\begin{align*}
&-c\eta+u+u\eta=(1-c^2)b_1,\\
&-cu+\c^2(|\partial_x|)\eta+\frac12u^2=(1-c^2)b_2
\end{align*}
for some $b_1$, $b_2\in\mathbb{R}$; $1-c^2$ is for convenience. We seek a {\em periodic traveling wave} of \eqref{E:main}-\eqref{def:c}. That is, $\eta$ and $u$ are $2\pi$ periodic functions of 
\[
z:=\k x\qquad\text{for some $\k>0$, \quad the wave number,} 
\]
and they solve
\begin{equation}\label{E:periodic}
\begin{aligned}
&-c\eta+u+u\eta=(1-c^2)b_1,\\
&-cu+\c^2(\k|\partial_z|)\eta+\frac12u^2=(1-c^2)b_2.
\end{aligned}
\end{equation}
Note that 
\begin{equation}\label{E:c-bound}
\c^2(\k|\partial_z|):H^k(\mathbb{T})\to H^{k+1}(\mathbb{T})\qquad
\text{for any $\k>0$}\quad\text{for any integer $k\geq 0$}.
\end{equation}
Note that 
\begin{equation}\label{def:ck}
\c^2(\k|\partial_z|)e^{inz}=\c^2(n\k)e^{inz}\qquad\text{for $n\in\mathbb{Z}$},
\end{equation}
or, equivalently, $\c^2(\k|\partial_z|)(1)=1$, 
\begin{equation*}\label{def:ck'}
\c^2(\k|\partial_z|)(\cos nz)=\c^2(n\k)\cos nz\quad\text{and}\quad
\c^2(\k|\partial_z|)(\sin nz)=\c^2(n\k)\sin nz. 
\end{equation*}

Note that \eqref{E:periodic} remains invariant under 
\begin{equation}\label{E:invariance}
z\mapsto z+z_0\quad\text{and}\quad z\mapsto -z
\end{equation}
for any $z_0\in \mathbb{R}$. Hence, in particular, we may assume that $\eta$ and $u$ are even. But \eqref{E:periodic} does not possess scaling invariance. Hence we may not a priori assume that $\k=1$. Rather, the (in)stability results reported herein depend on the carrier wave number; see Theorem~\ref{thm:index} and Corollary~\ref{cor:kc}, for instance, for details. Moreover, \eqref{E:periodic} does not possess Galilean invariance. Hence we may not a priori assume that $b_1$ or $b_2=0$. Rather, we exploit variations of \eqref{E:periodic} in the $b_1$ and $b_2$ variables in the course of the stability proof; see Lemma~\ref{lem:L} for details. To compare, the Whitham equation (see \eqref{E:Whitham}) for periodic traveling waves, after normalization of parameters,
\[
-c\eta+\c(\k|\partial_z|)\eta+\eta^2=(1-c)^2b\qquad \text{for some $b\in \mathbb{R}$}
\]
remains invariant under
\[
\eta\mapsto\eta+\eta_0,\quad c\mapsto c-2\eta_0,\quad\text{and}\quad 
(1-c)^2b\mapsto (1-c)^2b+(1-c)\eta_0+\eta_0^2
\]
for any $\eta_0\in\mathbb{R}$; see \cite{HJ2}, for instance.

We state an existence result for periodic traveling waves of \eqref{E:main}-\eqref{def:c} and their small amplitude expansion.

\begin{theorem}[Existence of sufficiently small, periodic wave trains]\label{thm:existence} 
For any $\k>0$, $b_1$, $b_2 \in \mathbb{R}$ and $|b_1|$, $|b_2|$ sufficiently small, a one parameter family of solutions of \eqref{E:periodic} exists, denoted $\eta(a; \k, b_1, b_2)(z)$, $u(a; \k, b_1, b_2)(z)$, and $c(a; \k, b_1, b_2)$, for $a \in \mathbb{R}$ and $|a|$ sufficiently small; $\eta$ and $u$ are $2\pi$ periodic, even, and smooth in $z$, and $c$ is even in $a$; $\eta$, $u$, and $c$ depend analytically on $a$, and $\k$, $b_1$, $b_2$. Moreover,
\begin{subequations}\label{E:hu-small}
\begin{align}
\eta(a;\k,b_1,b_2)(z)=&\eta_0(\k,b_1,b_2)+a\cos z +a(b_1\cc(\k)+b_2)\cos z \label{E:h-small}\\
&+a^2(h_0+h_2\cos2z)+O(a(a+b_1+b_2)^2), \notag \\
u(a;\k,b_1,b_2)(z)=&u_0(\k,b_1,b_2)+a\cc(\k)\cos z
+\frac12a\cc(\k)(b_1\cc(\k)+b_2)\cos z\hspace*{-25pt} \label{E:u-small}\\
&+a^2\cc(\k)\Big(h_0-\frac12+\Big(h_2-\frac12\Big)\cos2z\Big)+O(a(a+b_1+b_2)^2),\hspace*{-25pt}\notag
\intertext{and}
c(a;\k,b_1,b_2)=&c_0(\k,b_1,b_2)
+\frac34a^2\cc(\k)(2h_0+h_2-1)+O(a(a+b_1+b_2)^2)\hspace*{-25pt}\label{E:c-small}
\end{align}
\end{subequations}
as $a$, $b_1$, $b_2 \to 0$; 
\begin{subequations}\label{E:hu0}
\begin{align}
\eta_0(\k,b_1,b_2)&=b_1\cc(\k)+b_2+O((b_1+b_2)^2),\label{E:h0}\\
u_0(\k,b_1,b_2)&=b_1+b_2\cc(\k)+O((b_1+b_2)^2),\label{E:u0}
\intertext{and}
c_0(\k,b_1,b_2)&=\cc(\k)+b_1\Big(\frac12\cc^2(\k)+1\Big)+\frac32b_2\cc(\k)+O((b_1+b_2)^2)\label{E:c0}
\end{align}
\end{subequations}
as $b_1$, $b_2\to 0$, where
\begin{equation}\label{def:h02}
h_0=\frac34\frac{\cc^2(\k)}{\cc^2(\k)-1}\quad\text{and}\quad 
h_2=\frac34\frac{\cc^2(\k)}{\cc^2(\k)-\cc^2(2\k)}.
\end{equation}
\end{theorem}

\subsection{Regularity}\label{ss:regularity}

As a preliminary, we establish the smoothness of solutions of \eqref{E:periodic}. 

\begin{lemma}[Regularity]\label{lem:regularity}
If $\eta$, $u \in H^1(\mathbb{T})$ solve \eqref{E:periodic} for some $c>0$, and $\k>0$, $b_1$, $b_2\in\mathbb{R}$ 
and if $c-\|u\|_{L^\infty(\mathbb{T})}\geq\epsilon>0$ for some $\epsilon$ then $\eta$, $u\in H^\infty(\mathbb{T})$.
\end{lemma}

\begin{proof}
We differentiate \eqref{E:periodic} to arrive at
\begin{align}
-c\eta'+u'+u\eta'+u'\eta=0\quad&\text{and}\quad -cu'+\c^2(\k|\partial_z|)\eta'+uu'=0,\notag
\intertext{whence}
\eta'=\frac{1+\eta}{c-u}u'\quad&\text{and}\quad u'=\frac{1}{c-u}\c^2(\k|\partial_z|)\eta'.\label{E:hu'}
\end{align}
Here and elsewhere, the prime means ordinary differentiation. 

Note that $\frac{1}{c-u}:H^1(\mathbb{T})\to H^1(\mathbb{T})$ by hypothesis. Since $\eta'\in L^2(\mathbb{T})$ by hypothesis, it follows from the latter equation of \eqref{E:hu'} and \eqref{E:c-bound} that $u'\in H^1(\mathbb{T})$. It then follows from the former equation of \eqref{E:hu'} and a Sobolev inequality that $\eta'\in H^1(\mathbb{T})$. In other words, $\eta$, $u\in H^2(\mathbb{T})$. A bootstrap argument completes the proof. 
\end{proof}

\subsection*{Notation}
Throughout, we use 
\begin{equation}\label{def:uv}
\u=\hu \quad\text{and}\quad \v=\hv
\end{equation}
whenever it is convenient to do so.

\subsection{An operator equation}\label{ss:operator}

Let $\f: H^1(\mathbb{T}) \times H^1(\mathbb{T}) \times \mathbb{R}^+ \times \mathbb{R}^+\times \mathbb{R}\times \mathbb{R} \to H^1(\mathbb{T}) \times H^1(\mathbb{T})$ such that
\begin{equation} \label{def:f}
\f(\u,c;\k,b_1,b_2)=
\begin{pmatrix}-c\eta+u+u\eta-(1-c^2)b_1 \\ -cu+\c^2(\k|\partial_z|)\eta+\frac12u^2-(1-c^2)b_2\end{pmatrix}.
\end{equation}
It is well defined by \eqref{E:c-bound} and a Sobolev inequality. We seek a solution $\u\in H^1(\mathbb{T})\times H^1(\mathbb{T})$, $c>0$, and $\k>0$, $b_1$, $b_2\in \mathbb{R}$ of 
\begin{equation}\label{E:f=0}
\f(\u,c;\k,b_1,b_2)=\mathbf{0}
\end{equation} 
satisfying $c-\|u\|_{L^\infty(\mathbb{T})}\geq \epsilon>0$ for some $\epsilon$ and, by virtue of Lemma~\ref{lem:regularity}, a solution $\u\in H^\infty(\mathbb{T})\times H^\infty(\mathbb{T})$ of \eqref{E:periodic}. Note that $\f$ is invariant under \eqref{E:invariance} for any $z_0\in\mathbb{R}$. Hence we may assume that $\u$ is even.

For any $c>0$, and $\k>0$, $b_1$, $b_2\in\mathbb{R}$, note that 
\[
\partial_{\u}\f(\u,c;\k,b_1,b_2)=\begin{pmatrix}u-c & 1+\eta\\ \c^2(\k|\partial_z|) & u-c\end{pmatrix}:
H^1(\mathbb{T}) \times H^1(\mathbb{T}) \to H^1(\mathbb{T}) \times H^1(\mathbb{T})
\]
is continuous by \eqref{E:c-bound} and a Sobolev inequality. For any $\u\in  H^1(\mathbb{T}) \times H^1(\mathbb{T})$, and $\k>0$, $b_1$, $b_2\in\mathbb{R}$, moreover,
\[
\partial_c\f(\u,c;\k,b_1,b_2)=\begin{pmatrix} -\eta+2cb_1\\ -u+2cb_2\end{pmatrix}:
\mathbb{R} \to H^1(\mathbb{T}) \times H^1(\mathbb{T})
\]
is continuous. Here (by abuse of notation) $\partial$ means Fr\'echet differentiation. Since
\[
\partial_{\k}\f(\u,c;\k,b_1,b_2)=
\begin{pmatrix} 0\\ \frac{1}{\k}(\text{sech}^2(\k|\partial_z|)-\c^2(\k|\partial_z|)) \end{pmatrix},
\]
and
\[
\partial_{b_1}\f(\u,c;\k,b_1,b_2)=\begin{pmatrix} c^2-1\\0\end{pmatrix}, \qquad 
\partial_{b_2}\f(\u,c;\k,b_1,b_2)=\begin{pmatrix} 0\\ c^2-1\end{pmatrix}
\]
are continuous, likewise, $\f$ depends continuously differentiably on its arguments. Furthermore, since the Fr\'echet derivatives of $\f$ with respect to $\u$, and $c$, $b_1$, $b_2$ of all orders $\geq 3$ are zero everywhere by brutal force, and since $\c^2$ is a real analytic function, $\f$ is a real analytic operator. 

\subsection{Bifurcation condition}\label{ss:bifurcation}

For any $c>0$ for any $\k>0$, $b_1$, $b_2\in \mathbb{R}$ and $|b_1|$, $|b_2|$ sufficiently small, note that 
\begin{equation}\label{E:hu0'}
\begin{aligned}
\eta_0(c;\k,b_1,b_2)&=b_1c+b_2+O((b_1+b_2)^2),\\
u_0(c;\k,b_1,b_2)&=b_1+b_2c+O((b_1+b_2)^2)
\end{aligned}
\end{equation}
make a constant solution of \eqref{def:f}-\eqref{E:f=0} and, hence, \eqref{E:periodic}. Let $\u_0=\begin{pmatrix} \eta_0 \\ u_0\end{pmatrix}(c;\k,b_1,b_2)$. It follows from the implicit function theorem that if non-constant solutions of \eqref{def:f}-\eqref{E:f=0} and, hence, \eqref{E:periodic} bifurcate from $\u=\u_0$ for some $c=c_0$ then, necessarily,
\[
\mathbf{L}_0:=\partial_\u\f(\u_0,c_0;\k,b_1,b_2): 
H^1(\mathbb{T}) \times H^1(\mathbb{T}) \to H^1(\mathbb{T}) \times H^1(\mathbb{T})
\]
is {\em not} an isomorphism. Here $\u_0$ depends on $c_0$. But we suppress it for simplicity of notation. Note that
\begin{equation}\label{E:hu1'}
\mathbf{L}_0\u_1e^{\pm inz}
=\begin{pmatrix}u_0-c_0&1+\eta_0\\ \c^2(\k|\partial_z|)&u_0-c_0\end{pmatrix}\u_1e^{\pm inz}=\mathbf{0} 
\qquad \text{for $n\in\mathbb{Z}$}
\end{equation}
for some nonzero $\u_1$ if and only if
\begin{equation}\label{E:bifurcation}
(c_0-u_0)^2=\c^2(n\k)(1+\eta_0).
\end{equation}
For $b_1=b_2=0$ and, hence, $\eta_0=u_0=0$ by \eqref{E:hu0'}, it simplifies to $c_0=\pm\c(n\k)$ --- the phase velocity of a $2\pi/n\k$ periodic wave in the linear theory; $\pm$ indicate right and left propagating waves, respectively. Without loss of generality, here we restrict the attention to $n=1$ and we assume the $+$ sign. For $|b_1|$ and $|b_2|$ sufficiently small, \eqref{E:bifurcation} becomes
\[
c_0=\c(\k)+b_1\Big(\frac12\c^2(\k)+1\Big)+\frac32b_2\c(\k)+O((b_1+b_2)^2).
\]
Substituting it into \eqref{E:hu0'}, we find
\begin{align*}
\eta_0&(\k,b_1,b_2)=b_1\c(\k)+b_2+O((b_1+b_2)^2),\\
u_0&(\k,b_1,b_2)=b_1+b_2\c(\k)+O((b_1+b_2)^2).
\end{align*} 
They agree with \eqref{E:hu0}. In the sequel, $\u_0=\begin{pmatrix}\eta_0\\u_0\end{pmatrix}(\k, b_1,b_2)$ and $c_0=c_0(\k,b_1,b_2)$.

\begin{figure}[h]
\includegraphics[scale=0.5]{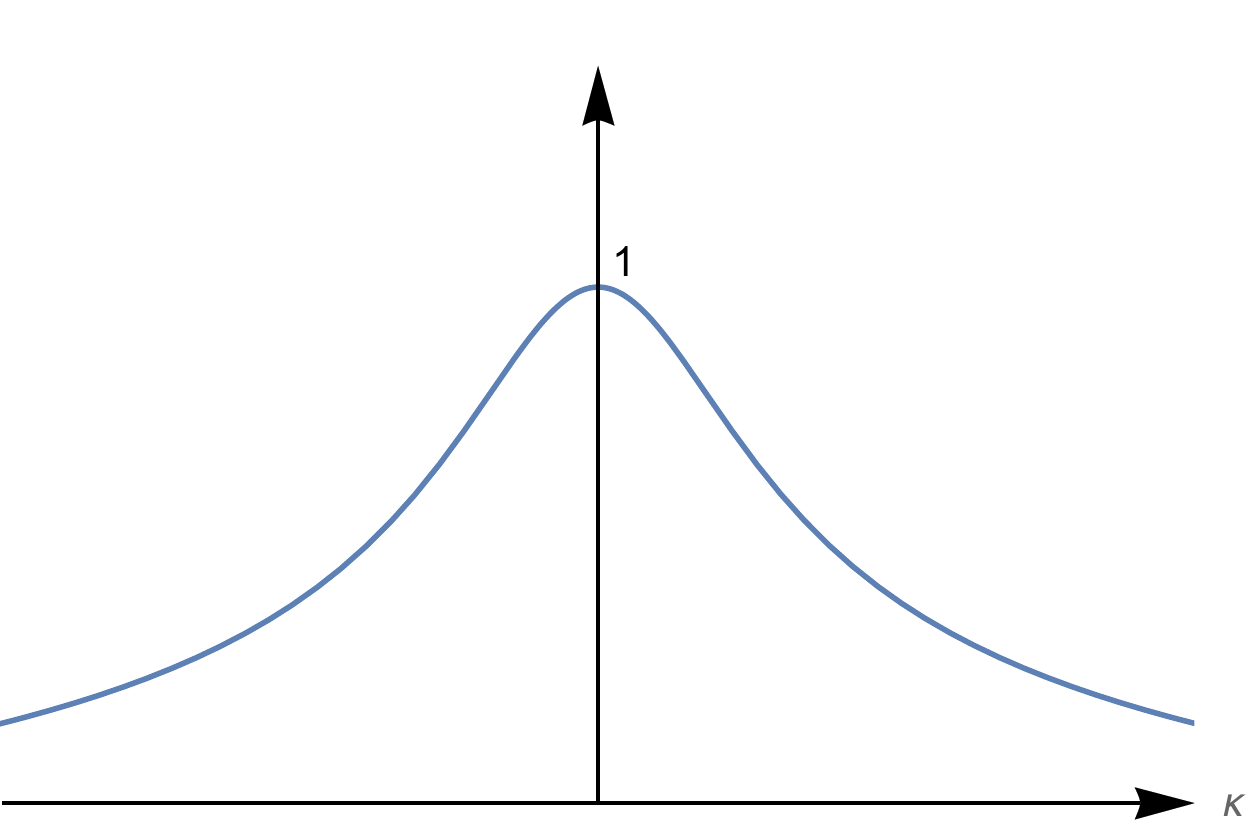}\caption{Schematic plot of $\c$.}\label{fig:c}
\end{figure}

Since $\c(\k)>\c(n\k)$ for $n=2,3,\dots$ pointwise in $\mathbb{R}$ (see Figure~\ref{fig:c}), a straightforward calculation reveals that for any $\k>0$, $b_1$, $b_2\in\mathbb{R}$ and $|b_1|$, $|b_2|$ sufficiently small, the $H^1(\mathbb{T}) \times H^1(\mathbb{T})$ kernel of $\mathbf{L}_0=\partial_\u\f(\u_0,c_0;\k,b_1,b_2)$ is two dimensional and spanned by $\u_1e^{\pm iz}$ for some nonzero $\u_1$ satisfying \eqref{E:hu1'}. Note from \eqref{E:hu1'} and \eqref{E:hu0} that 
\begin{equation}\label{E:hu1}
\u_1=\begin{pmatrix} 1+\eta_0\\c_0-u_0\end{pmatrix}=
\begin{pmatrix}1+b_1\c(\k)+b_2\\ \c(\k)+\frac12b_1\c^2(\k)+\frac12b_2\c(\k)\end{pmatrix}+O((b_1+b_2)^2)
\end{equation}
as $b_1$, $b_2\to0$ up to the multiplication by a constant. This agrees with \eqref{E:h-small} and \eqref{E:u-small} at the order of $a$. By the way, in the presence of the effects of surface tension, if $\c(\k)=\c(n\k)$ for some integer $n\geq 2$ for some $\k>0$, resulting in the resonance of the fundamental mode and the $n$-th harmonic, then the kernel would be four dimensional; see Section~\ref{sec:ST} for details. 

Moreover, a straightforward calculation reveals that for any $\k>0$, $b_1$, $b_2 \in \mathbb{R}$ and $|b_1|$, $|b_2|$ sufficiently small,  the $H^1(\mathbb{T}) \times H^1(\mathbb{T})$ co-kernel of $\mathbf{L}_0$ is two dimensional and  spanned by $\u_1^\perp e^{\pm iz}$ for some $\u_1^\perp$ orthogonal to $\u_1$. In particular, $\mathbf{L}_0$ is a Fredholm operator of index zero. 

\subsection{Lyapunov-Schmidt procedure}\label{ss:LS}

For any $\k>0$, $b_1$, $b_2 \in \mathbb{R}$ and $|b_1|$, $|b_2|$ sufficiently small, we turn the attention to non-constant solutions of \eqref{def:f}-\eqref{E:f=0} and, hence, \eqref{E:periodic} bifurcating from $\u=\u_0$ and $c=c_0$, where $\eta_0$, $u_0$, and $c_0$ are in \eqref{E:hu0}. A Lyapunov-Schmidt procedure (see \cite[Section~2.7.6]{Nir}, for instance) is instrumental for the purpose. Here the proof follows along the same line as the arguments in \cite{HJ2, HP1}, but with suitable modifications to accommodate product spaces. Throughout the subsection, $\k$, and $b_1$, $b_2$ are fixed and suppressed for simplicity of notation.  

Recall $\f(\u_0,c_0)=\mathbf{0}$, where $\f$ is in \eqref{def:f}. Recall $\mathbf{L}_0\u_1e^{\pm iz}=0$, where $\mathbf{L}_0$ is in \eqref{E:hu1'} and $\u_1$ is in \eqref{E:hu1}. We write that
\begin{equation}\label{def:soln2}
\u(z)=\u_0+\frac12\u_1(ae^{iz}+a^*e^{-iz})+\u_r(z)
\quad\text{and}\quad c=c_0+c_r,
\end{equation}
and we require that $a\in \mathbb{C}$, $\u_r=\begin{pmatrix}\eta_r\\u_r\end{pmatrix}\in H^1(\mathbb{T})\times H^1(\mathbb{T})$ be even and
\begin{align}\label{E:ortho}
\l\u_r, \u_1(ae^{iz}+a^*e^{-iz})\r= 
\frac{1}{2\pi}\int^\pi_{-\pi}(&(1+\eta_0)\eta_r(z)(ae^{iz}+a^*e^{-iz})\\ 
&+(c_0-u_0)u_r(z)(ae^{iz}+a^*e^{-iz}))~dz=0,\hspace*{-10pt}\notag
\end{align}
and $c_r\in\mathbb{R}$. Here and elsewhere, the asterisk means complex conjugation; $\l\cdot\,,\cdot\r$ is the $L^2(\mathbb{T})\times L^2(\mathbb{T})$ inner product.

Substituting \eqref{def:soln2} into \eqref{def:f}-\eqref{E:f=0}, we use $\f(\u_0;c_0)=\mathbf{0}$, and \eqref{E:hu1'}, \eqref{E:hu1}, and we make an explicit calculation to arrive at
\begin{align}\label{E:Lv=g}
\mathbf{L}_0\u_r=&
-\begin{pmatrix} \left(\frac12(c_0-u_0)(ae^{iz}+a^*e^{-iz})+u_r\right)
\left(\frac12(1+\eta_0)(ae^{iz}+a^*e^{-iz})+\eta_r\right)\\
\frac12\left(\frac12(c_0-u_0)(ae^{iz}+a^*e^{-iz})+u_r\right)^2\end{pmatrix}\\
&+c_r \begin{pmatrix}\frac12(1+\eta_0)(ae^{iz}+a^*e^{-iz})+\eta_r \\
\frac12(c_0-u_0)(ae^{iz}+a^*e^{-iz})+u_r\end{pmatrix} \notag \\
=:&\g(\u_r;a,a^*,c_r) \notag
\end{align}
up to terms of order $c_r$ as $c_r\to 0$. Recall that $\f$ is a real analytic operator. Hence $\g$ depends analytically on its arguments. Clearly, $\g(\mathbf{0};0,0,c_r)=0$ for any $c_r\in\mathbb{R}$. 

Recall that $\mathbf{L}_0$ is a Fredholm operator of index zero,
\[
\ker\mathbf{L}_0=\text{span}\{\u_1e^{\pm iz}\}\quad\text{and}\quad 
\text{co-ker\,}\mathbf{L}_0=\text{span}\{\u_1^\perp e^{\pm iz}\},
\]
where $\u_1$ and $\u_1^\perp$ are orthogonal to each other. Let $\Pi$ denote the spectral projection of $L^2(\mathbb{T})\times L^2(\mathbb{T})$ onto the kernel of $\mathbf{L}_0$. Specifically, if $\v={\displaystyle \sum_{n\in\mathbb{Z}}\begin{pmatrix} \widehat{\zeta}(n) \\ \widehat{v}(n)\end{pmatrix}e^{inz}}$ in the Fourier series then
\begin{align*}
\Pi\v=&\l\v,\u_1e^{iz}\r\u_1e^{iz}+\l\v,\u_1e^{-iz}\r\u_1e^{-iz}  \\
=&((1+\eta_0)(\widehat{\zeta}(1)e^{iz}+\widehat{\zeta}(-1)e^{-iz})
+(c_0-u_0)(\widehat{v}(1)e^{iz}+\widehat{v}(-1)e^{-iz}))\u_1.
\end{align*}
Since $\Pi\u_r=0$ by \eqref{E:ortho}, we may recast \eqref{E:Lv=g} as
\begin{equation}\label{E:LS}
\mathbf{L}_0\u_r=(1-\Pi)\g(\u_r;a,a^*, c_r)\quad\text{and}\quad \mathbf{0}=\Pi\g(\u_r;a,a^*, c_r).
\end{equation}
Moreover, $\mathbf{L}_0: (1-\Pi)(H^1(\mathbb{T})\times H^1(\mathbb{T})) \to \text{range\,}\mathbf{L}_0$ is invertible. Specifically, if 
\[
\v={\displaystyle \begin{pmatrix} 1+\eta_0 \\ u_0-c_0 \\ \end{pmatrix}(v_{+1}e^{iz}+v_{-1}e^{-iz})
+\sum_{n\neq\pm1}\begin{pmatrix} \widehat{\zeta}(n) \\ \widehat{v}(n)\end{pmatrix}e^{inz}},
\] 
for some constants $v_{\pm1}$, belongs to the range of $\mathbf{L}_0$ by \eqref{E:hu1'} then
\begin{align*}
\mathbf{L}_0^{-1}\v(z)=&\begin{pmatrix} 0 \\1 \end{pmatrix}(v_{+1}e^{iz}+v_{-1}e^{-iz}) \\
&+\sum_{n\neq\pm1}\frac{1}{(u_0-c_0)^2-\c^2(n\k)(1+\eta_0)}
\begin{pmatrix}u_0-c_0 & -1-\eta_0 \\ -\c^2(\k n) & u_0-c_0\end{pmatrix}
\begin{pmatrix} \widehat{\zeta}(n) \\ \widehat{v}(n)\end{pmatrix} e^{inz}.
\end{align*}
It is well defined since \eqref{E:bifurcation} holds true if and only if $n=\pm1$. Hence we may recast \eqref{E:LS} as
\begin{equation}\label{E:LS'}
\u_r=\mathbf{L}_0^{-1}(1-\Pi)\g(\u_r;a,a^*, c_r)\quad\text{and}\quad \mathbf{0}=\Pi\g(\u_r;a,a^*,c_r).
\end{equation}
Clearly, $\mathbf{L}_0^{-1}(1-\Pi)\g$ depends analytically on its arguments. Since $\g(\mathbf{0};0,0,c_r)=0$ for any $c_r\in\mathbb{R}$, it follows from the implicit function theorem that a unique solution $\u_r=\u_2(a,a^*,c_r)$ exists to the former equation of \eqref{E:LS'} near $\u_r=\mathbf{0}$ for $a\in\mathbb{C}$ and $|a|$ sufficiently small for any $c_r\in\mathbb{R}$. Note that $\u_2$ depends analytically on its arguments and it satisfies \eqref{E:ortho} for $|a|$ sufficiently small for any $c_r\in\mathbb{R}$. The uniqueness implies
\begin{equation}\label{E:u2=0}
\u_2(0,0,c_r)=\mathbf{0}\qquad \text{for any $c_r\in\mathbb{R}$.}
\end{equation}
Moreover, since \eqref{def:f}-\eqref{E:f=0} and, hence, \eqref{E:LS'} are invariant under \eqref{E:invariance} for any $z_0\in\mathbb{R}$, it follows that
\begin{equation}\label{E:u2-symm}
\u_2(a,a^*,c_r)(z+z_0)=\u_2(ae^{iz_0},a^*e^{-iz_0},c_r)(z) \quad\text{and}\quad
\u_2(a,a^*,c_r)(-z)=\u_2(a,a^*,c_r)(z) 
\end{equation}
for any $z_0\in\mathbb{R}$ for any $a\in\mathbb{C}$, $|a|$ sufficiently small, and $c_r\in\mathbb{R}$.

To proceed, we write the latter equation of \eqref{E:LS'} as 
\[
\Pi \g(\u_2(a,a^*,c_r);a, a^*,c_r)=\mathbf{0}
\]
for $a\in\mathbb{C}$ and $|a|$ sufficiently small for any $c_r\in\mathbb{R}$. This is solvable, provided that
\begin{equation}\label{E:q=0}
\pi_{\pm}(a,a^*,c_r):=\l\g(\u_2(a,a^*, c_r);a,a^*,c_r),\u_1(ae^{iz}\pm a^*e^{-iz})\r=0;
\end{equation}
$\l\cdot\,,\cdot\r$ is the $L^2(\mathbb{T})\times L^2(\mathbb{T})$ inner product. We use \eqref{E:u2-symm}, where $z_0=-2\arg(a)$, and \eqref{E:q=0} to show that
\[
\pi_- (a^*,a,c_r)=\pi_- (a,a^*,c_r)=-\pi_- (a^*,a,c_r).
\]
Hence $\pi_-(a,a^*,c_r)=0$ holds true for any $a\in\mathbb{C}$ and $|a|$ sufficiently small for any $c_r\in\mathbb{R}$. Moreover, we use \eqref{E:u2-symm}, where $z_0=-\arg (a)$, and \eqref{E:q=0} to show that 
\[
\pi_+ (a,a^*,c_r)=\pi_+ (|a|,|a|,c_r).
\] 
Hence it suffices to solve $\pi_+(a,a,c_r)=0$ for $a$, $c_r \in\mathbb{R}$ and $|a|$ sufficiently small. 

Substituting \eqref{E:Lv=g} into \eqref{E:q=0}, where $\u_r=\u_2(a,a,c_r)=:\begin{pmatrix} \eta_2 \\ u_2\end{pmatrix}(a,c_r)$, we make an explicit calculation to arrive at 
\[
\pi_+(a,a,c_r)=:a^2(c_r((1+\eta_0)^2+(c_0-u_0)^2)+\pi_r(a,c_r))
\]
for $a$, $c_r\in\mathbb{R}$ and $|a|$ sufficiently small, where 
\begin{align*}
\pi_r(a,c_r) =&-a^2(1+\eta_0)((c_0-u_0)\l\cos z\,\eta_2(a,c_r),\cos z\r\\
&\hspace*{55pt}+(1+\eta_0)\l\cos z\,u_2(a,c_r),\cos z\r+a^{-1}\l(\eta_2 u_2)(a,c_r),\cos z\r) \\
&-\tfrac12a^2(c_0-u_0)(2(c_0-u_0)\l\cos z\,u_2(a,c_r),\cos z\r+a^{-1}\l u_2^2(a,c_r),\cos z\r);
\end{align*}
$\l\cdot\,,\cdot\r$ means the $L^2(\mathbb{T})$ inner product. We merely pause to remark that $\pi_r$ is well defined. Indeed, $a^{-1}\eta_2$ and $a^{-1}u_2$ are not singular for $a\in\mathbb{R}$ and $|a|$ sufficiently small by \eqref{E:u2=0}. Clearly, $\pi_r$ and, hence, $\pi_+$ depend analytically on their arguments. Since $\pi_r(0,0)=0$ and $(\partial_{c_r}\pi_r)(0,0)=0$ by \eqref{E:u2=0}, it follows from the implicit function theorem that a unique solution $c_r=c_2(a)$ exists to $\pi_+(a,a,c_r)=0$ and, hence, the latter equation of \eqref{E:LS'} near $c_r=0$ for $a\in\mathbb{R}$ and $|a|$ sufficiently small. Clearly, $c_2$ depends analytically on $a$.

To summarize, 
\[
\u_r=\u_2(a,a,c_2(a))\quad\text{and}\quad c_r=c_2(a)
\] 
uniquely solve \eqref{E:Lv=g} for $a\in\mathbb{R}$ and $|a|$ sufficiently small, and by virtue of \eqref{def:soln2}, 
\begin{equation}\label{E:soln2}
\u(a)(z)=\u_0+a\u_1\cos z+\u_2(a,a,c_2(a))(z)\quad\text{and}\quad  c(a)=c_0+c_2(a)
\end{equation}
uniquely solve \eqref{def:f}-\eqref{E:f=0} and, hence, \eqref{E:periodic} for $a\in\mathbb{R}$ and $|a|$ sufficiently small. Note that $\u_2$ and, hence, $\u$ are $2\pi$ periodic and even in $z$. Since $\u_2$ and $c_2$ are near $\mathbf{0}$ and $0$, Lemma~\ref{lem:regularity} implies that $\u$ is smooth in $z$. We claim that $c$ is even in $a$. Indeed, note that \eqref{E:periodic} and, hence, \eqref{def:f}-\eqref{E:f=0} remain invariant under $z\mapsto z+\pi$ by \eqref{E:invariance}. Since $(\partial_a\u)(0)(z)=\u_1\cos z$, however, $\u(z)\neq \u(z+\pi)$ must hold true. Thus $(\partial_ac)(0)=0$. This proves the claim. 
Clearly, $\u$ and $c$ depend analytically on $a\in\mathbb{R}$ and $|a|$ sufficiently small. This completes the existence proof.

\subsection{Small amplitude expansion}\label{ss:small}

It remains to verify \eqref{E:hu-small}. Throughout the subsection, $\k>0$ is fixed and suppressed for simplicity of notation; $b_1$, $b_2\in\mathbb{R}$ and $|b_1|$, $|b_2|$ sufficiently small are fixed. 

Recall from the existence proof that \eqref{E:soln2} depends analytically on $a$, $b_1$, $b_2\in\mathbb{R}$ and $|a|$, $|b_1|$, $|b_2|$ sufficiently small. We write that
\begin{align*}
\eta(a;b_1,b_2)(z)=&\eta_0(b_1,b_2)+a(1+\eta_0(b_1,b_2))\cos z \\
&+a^2\eta_2(z)+a^3\eta_3(z)+O(a^4+a^2(b_1+b_2)+a(b_1+b_2)^2), \\
u(a;b_1,b_2)(z)=&u_0(b_1,b_2)+a(c_0-u_0)(b_1,b_2)\cos z \\
&+a^2u_2(z)+a^3u_3(z)+O(a^4+a^2(b_1+b_2)+a(b_1+b_2)^2),
\intertext{and}
c(a;b_1,b_2)=&c_0(b_1,b_2)+a^2c_2+O(a^4+a^2(b_1+b_2)+a(b_1+b_2)^2)
\end{align*}
as $a,b_1,b_2\to0$, where $\eta_0$, $u_0$, and $c_0$ are in \eqref{E:hu0}, and we require that $\eta_2$, $u_2$, and $\eta_3$, $u_3$ be $2\pi$ periodic, even, and smooth functions of $z$, and $c_2\in\mathbb{R}$. We merely pause to remark that $\eta_2$, $u_2$, $\eta_3$, $u_3$, and $c_2$ do {\em not} depend on $b_1$ and $b_2$, whereas $\eta_0$, $u_0$, and $c_0$ do. In the following sections, we restrict the attention to periodic traveling waves of \eqref{E:main}-\eqref{def:c} for $a\in\mathbb{R}$ and $|a|$ sufficiently small for $b_1=b_2=0$, and we calculate the spectrum of the associated linearized operator up to the order of $a$. (The index formulae would become unwieldy when terms of order $a^2$ were to be added.) For the purpose, it suffices to calculate solutions explicitly up to terms of orders $a^2$, and $ab_1$, $ab_2$. 

Substituting the above into \eqref{E:periodic}, we recall that $\eta_0$, $u_0$, and $c_0$ solve \eqref{E:periodic}, and we make an explicit calculation to arrive, at the order of $a$, at
\begin{align*}
&-c_0(1+\eta_0)\cos z+(c_0-u_0)\cos z=0,\\
&-c_0(c_0-u_0)\cos z+\c^2(\k|\partial_z|)(1+\eta_0)\cos z=0.
\end{align*}
This holds true up to terms of orders $b_1$ and $b_2$ by \eqref{def:ck'}, \eqref{E:c0}, and \eqref{E:hu1'}, \eqref{E:hu1}. 

To proceed, we assume $b_1=b_2=0$ and, hence, $\eta_0=u_0=0$ and $c_0=\c(\k)$ by \eqref{E:hu0}. At the order of $a^2$, we gather
\begin{align*}
&-c_0\eta_2+u_2+c_0\cos^2 z=0,\\
&-c_0u_2+\c^2(\k|\partial_z|)\eta_2+\tfrac12c_0^2\cos^2 z=0.
\end{align*}
We then use \eqref{def:ck'}, \eqref{E:c0} and we make an explicit calculation to find
\begin{equation}\label{E:hu2}
\eta_2(z)=h_0+h_2\cos 2z\quad\text{and}\quad u_2(z)=h_0-\frac12+\Big(h_2-\frac12\Big)\cos 2z,
\end{equation}
where $h_0$ and $h_2$ are in \eqref{def:h02}. 
Continuing, at the order of $a^3$, we gather
\begin{align*}
&-c_0\eta_3-c_2\cos z+u_3+u_2\cos z+c_0\eta_2\cos z=0,\\
&-c_0u_3-c_2c_0\cos z+\c^2(\k|\partial_z|)\eta_3+c_0u_2\cos z=0.
\end{align*}
Taking $L^2(\mathbb{T})$ inner products, we use \eqref{def:ck'} and \eqref{E:hu2}, so that
\begin{align*}
&-c_0\l\eta_3,\cos z\r-c_2+\l u_3,\cos z\r+h_0-\frac12+\frac12\Big(h_2-\frac12\Big)+c_0\Big(h_0+\frac12h_2\Big)=0,\\
&-c_0\l u_3,\cos z\r-c_2c_0+\c^2(\k)\l\eta_3,\cos z\r+c_0\Big(h_0-\frac12+\frac12\Big(h_2-\frac12\Big)\Big)=0.
\end{align*}
We then use \eqref{E:c0} and we make an explicit calculation to find 
\[
c_2=\frac34\c(\k)(2h_0+h_2-1).
\] 
This completes the proof.

\section{Modulational instability}\label{sec:MI}

Let $\eta=\eta(a;\k,b_1,b_2)$, $u=u(a;\k,b_1,b_2)$, and $c=c(a;\k,b_1,b_2)$, for some $a\in\mathbb{R}$ and $|a|$ sufficiently small, $\k>0$, $b_1$, $b_2\in\mathbb{R}$ and $|b_1|$, $|b_2|$ sufficiently small, denote a $2\pi/\k$-periodic wave train of \eqref{E:main}-\eqref{def:c}, whose existence follows from Theorem~\ref{thm:existence}. We address its stability and instability to ``slow modulations." Throughout the section, we employ the notation of \eqref{def:uv} whenever it is convenient to do so.

\subsection*{Well-posedness}

The solution of the linear part of \eqref{E:main}-\eqref{def:c} does not possess smoothing effects. Hence it is difficult to work out the well-posedness in spaces of low regularities. But, for the present purpose, it suffices to solve the Cauchy problem in some functional analytic setting. In Appendix~\ref{sec:LWP}, we establish the local-in-time well-posedness for \eqref{E:main}-\eqref{def:c} in $H^s(\mathbb{R})\times H^{s+1/2}(\mathbb{R})$ for any $s>2$.

\subsection{Spectral stability and instability}\label{ss:stability}

Intuitively, the stability of $\u$ means that if we perturb $\u$ at time $t=0$ then the solution at later times remains near (a spatial translate of) it. In a leading approximation, we will linearize \eqref{E:main}-\eqref{def:c} about $\u$ in the coordinate frame moving at the speed $c$. Recall that $\u$ and $c$ solve \eqref{E:periodic} and $z=\k x$. The result becomes 
\[
\partial_t\v=\k\partial_z \begin{pmatrix} c-u &-1-\eta \\ -\c^2(\k|\partial_z|) & c-u \end{pmatrix} \v.
\]
We seek a solution of the form $\v(z,t)=e^{\lambda \k t}\v(z)$, $\lambda\in\mathbb{C}$, to arrive at
\begin{equation}\label{def:L}
\lambda\v=\partial_z\begin{pmatrix} c-u &-1-\eta \\ -\c^2(\k|\partial_z|) & c-u \end{pmatrix}\v
=:\mathcal{L}(a;\k,b_1,b_2)\v,
\end{equation}
where
\[
\L:H^1(\mathbb{R})\times H^1(\mathbb{R}) \subset 
L^2(\mathbb{R})\times L^2(\mathbb{R}) \to L^2(\mathbb{R})\times L^2(\mathbb{R}).
\] 
We say that $\u$ is {\em spectrally unstable} to square integrable perturbations if the $L^2(\mathbb{R}) \times L^2(\mathbb{R})$ spectrum of $\L$ intersects the open right-half plane of $\mathbb{C}$, and it is spectrally stable otherwise. Note that $\eta$ and $u$ are $2\pi$ periodic in $z$, but $\v$ needs not. 

Note that \eqref{def:L} remains invariant under 
\[
\lambda\mapsto \lambda^*\quad\text{and}\quad \v\mapsto \v^*,
\] 
where $*$ means complex conjugation, and under
\[
\lambda\mapsto -\lambda\quad\text{and}\quad z\mapsto -z.
\]
Together, the spectrum of $\L$ is symmetric with respect to the reflections in the real and imaginary axes. Hence $\u$ is spectrally unstable if and only if the spectrum of $\L$ is not contained in the imaginary axis.

Spectral instability reported herein promotes spatially localized and temporally exponentially growing solutions of \eqref{E:main}-\eqref{def:c}. We will investigate this in a future publication. In stark contrast, spectral stability does {\em not} in general imply nonlinear stability.  

\subsection{Floquet characterization of the spectrum}\label{ss:Floquet}

It is well known (see \cite[Section~8.16]{RSIV} and \cite[Section~2.4]{Chicone}, for instance, for details; see also \cite{BHJ}) that the $L^2(\mathbb{R}) \times L^2(\mathbb{R})$ spectrum of $\L$, which by the way involves periodic coefficients, contains no eigenvalues. Rather, it consists of the essential spectrum. Moreover, a nontrivial solution of \eqref{def:L} does not belong to $L^p(\mathbb{R}) \times L^p(\mathbb{R})$ for any $p\in[1,\infty)$. Rather, if $\v\in L^\infty(\mathbb{R})\times L^\infty(\mathbb{R})$ solves \eqref{def:L} then, necessarily,
\[
\v(z)=e^{i\xi z}\p(z),\qquad\text{where\quad $\p(z+2\pi)=\p(z)$},
\]
for some $\xi$ in the range $(-1/2,1/2]$, the {\em Floquet exponent}. We take a Floquet theory approach to characterize the $L^2(\mathbb{R}) \times L^2(\mathbb{R})$ spectrum of $\L$ in a convenient form. By the way, $\L$ involves a nonlocal operator. Hence classical Floquet theory is not directly applicable. Details are found in \cite{BHJ}, for instance, and references therein. Hence we merely hit the main points. 

We begin by writing $v\in L^2(\mathbb{R})$ as
\[
v(z)=\frac{1}{\sqrt{2\pi}}\int^{1/2}_{-1/2} \Big(\sum_{n\in\mathbb{Z}} \widehat{v}(n+\xi)e^{inz}\Big)e^{i\xi z}~d\xi
=:\int^{1/2}_{-1/2}v(\xi)(z)e^{i\xi z}~d\xi,
\]
where $\widehat{v}$ means the Fourier transform of $v$. It is well defined in the Schwartz class by the Fubini theorem and the dominated convergence theorem, and it is extended to $L^2(\mathbb{R})$ by a density argument. Note that $v(\xi) \in L^2(\mathbb{T})$ for any $\xi\in(-1/2,1/2]$. The Parseval theorem asserts that 
\[
\|v\|_{L^2(\mathbb{R})}^2=\|\widehat{v}\|_{L^2(\mathbb{R})}^2=\int^{1/2}_{-1/2}\|v(\xi)\|_{L^2(\mathbb{T})}^2~d\xi.
\]
Hence $v\mapsto v(\xi)$ is an isomorphism between $L^2(\mathbb{R})$ and $L^2((-1/2,1/2]; L^2(\mathbb{T}))$. Let $\M:L^2(\mathbb{R}) \to L^2(\mathbb{R})$ denote a Fourier multiplier operator, defined as
\[
\widehat{\M v}(\k)=m(\k)\widehat{v}(\k)
\]
for a suitable function $m$, real valued and Lebesgue measurable. It is straightforward to verity that 
\[
(\M v)(\xi)=e^{-i\xi z}\M e^{i\xi z}v(\xi)=:\M(\xi)v(\xi)
\]
for any $v\in L^2(\mathbb{R})$ and $\xi\in(-1/2,1/2]$. Note that $\M(\xi): L^2(\mathbb{T}) \to L^2(\mathbb{T})$ for any~$\xi$. Moreover, for a suitable function $f$, 
\[
(fv)(\xi)=fv(\xi)\qquad\text{for any $v\in L^2(\mathbb{R})$ and $\xi\in(-1/2,1/2]$}.
\]
We extend this to product spaces in the usual manner. It is then straightforward to verify that
\[
(\L\v)(\xi)=e^{-i\xi z}\L e^{i\xi z}\v(\xi)=:\L(\xi)\v(\xi)
\]
for any $\v\in L^2(\mathbb{R})\times L^2(\mathbb{R})$ and $\xi\in(-1/2,1/2]$. Note that
\[
\L(\xi):H^1(\mathbb{T})\times H^1(\mathbb{T})\subset 
L^2(\mathbb{T})\times L^2(\mathbb{T}) \to L^2(\mathbb{T})\times L^2(\mathbb{T})
\]
for any $\xi\in(-1/2,1/2]$. 

Furthermore (see \cite[Section~8.16]{RSIV}, for instance, for details; see also \cite{BHJ}), $\lambda$ belongs to the $L^2(\mathbb{R}) \times L^2(\mathbb{R})$ spectrum of $\L$ if and only if it belongs to the $L^2(\mathbb{T}) \times L^2(\mathbb{T})$ spectrum of $\L(\xi)$ for some $\xi\in(-1/2,1/2]$. That is, 
\begin{equation}\label{def:Lx}
\lambda\p=e^{-i\xi z}\partial_z\begin{pmatrix} c-u &-1-\eta \\ -\c^2(\k|\partial_z|) & c-u \end{pmatrix}e^{i\xi z}\p
\end{equation}
for some nontrivial $\p\in L^2(\mathbb{T})\times L^2(\mathbb{T})$ and $\xi\in(-1/2,1/2]$. Hence
\[
\text{spec}_{L^2(\mathbb{R})\times L^2(\mathbb{R})}(\L)
=\bigcup_{\xi\in[-1/2,1/2)}\text{spec}_{L^2(\mathbb{T})\times L^2(\mathbb{T})}(\L(\xi)).
\]
Note that for any $\xi \in(-1/2,1/2]$, the $L^2(\mathbb{T})\times L^2(\mathbb{T})$ spectrum of $\L(\xi)$ consists of eigenvalues with finite multiplicities. Thus we characterize the essential spectrum of $\L$ as a one parameter family of point spectra of $\L(\xi)$ for $\xi\in (-1/2,1/2]$. 

Note that \eqref{def:Lx}, when $\pm\xi$ are taken in pair, remains invariant under
\[
\lambda\mapsto \lambda^*\quad\text{and}\quad \p\mapsto \p^*,
\] 
and under
\[
\lambda\mapsto -\lambda\quad\text{and}\quad z\mapsto -z.
\]
Hence we may assume $\xi\in[0,1/2]$.

\subsection{Definition of modulational instability}\label{ss:MI}

Note that $\xi=0$ corresponds to the same period perturbations as $\u$. Moreover, $\xi>0$ and small corresponds to long wavelength perturbations, whose effects are to slowly vary the period and other wave characteristics, such as the amplitude. They supply the spectral information of $\L$ in the vicinity of the origin in $\mathbb{C}$; see \cite{BHJ}, for instance, for details. We then say that $\u$ is {\em modulationally unstable} if the $L^2(\mathbb{T})\times L^2(\mathbb{T})$ spectra of $\L(\xi)$ are not contained in the imaginary axis near the origin for $\xi>0$ and small, and it is modulationally stable otherwise.

For an arbitrary $\xi$, one must study \eqref{def:Lx} numerically except for few cases --- for instance, completely integrable systems (see \cite{BHJ}, for instance, for references). But, for $\xi>0$ and small for $\lambda$ in the vicinity of the origin in $\mathbb{C}$, we may take a spectral perturbation approach in \cite{HJ2, HP1}, for instance, to address it analytically. This is the subject of investigation here. 

\subsection*{Notation} 

In the remaining of the section, $\k>0$ is suppressed for simplicity of notation, unless specified otherwise. We assume $b_1=b_2=0$. For nonzero $b_1$ and $b_2$, one may explore in like manner. But the calculation becomes lengthy and tedious. Hence we do not discuss the details. We use 
\begin{equation}\label{def:Lxa}
\L(\xi,a)=\L(\xi)(a;\k,0,0)
\end{equation}
for simplicity of notation. 

\subsection{Spectra of $\L(\xi,0)$}\label{ss:a=0}
We begin by discussing the $L^2(\mathbb{T})\times L^2(\mathbb{T})$ spectra of $\L(\xi,0)$ for $\xi\in[0,1/2]$. This is the linearization of \eqref{E:main}-\eqref{def:c} about $\eta=u=0$ and $c=\c(\k)$ --- namely, the rest state --- in the moving coordinate frame. 

Note from \eqref{def:Lx} and \eqref{E:hu0} that
\[
\L(\xi,0)=e^{-i\xi z}\partial_z \begin{pmatrix} \c(\k) &-1\\ -\c^2(\k|\partial_z|) & \c(\k)\end{pmatrix}e^{i\xi z}.
\]
We use \eqref{def:ck} and make an explicit calculation to show that 
\begin{equation}\label{E:Lx0=0}
\mathcal{L}(\xi,0)\e(n+\xi,\pm)=i\omega(n+\xi,\pm)\e(n+\xi,\pm)
\qquad\text{for $n\in\mathbb{Z}$ and $\xi\in[0,1/2]$,}
\end{equation}
where 
\begin{equation}\label{def:eigen}
\omega(n+\xi,\pm)=(n+\xi)(\c(\k)\pm\c(\k(n+\xi)))\quad\text{and}\quad
\e(n+\xi,\pm)(z)=\begin{pmatrix}1\\ \mp\c(\k(n+\xi))\end{pmatrix}e^{inz}.
\end{equation}
Hence for any $\xi\in[0,1/2]$, the spectrum of $\L(\xi,0)$ consists of {\em two} families of infinitely many and purely imaginary eigenvalues, each with finite multiplicity. In particular, the rest state of \eqref{E:main}-\eqref{def:c} is spectrally stable to square integrable perturbations. 

The spectrum of the linear operator associated with the water wave problem consists of $i\omega(n+\xi,\pm)$ for $n\in\mathbb{Z}$ and $\xi\in[-1/2,1/2)$; see \cite{Whitham}, for instance, for details. To compare, the spectrum of the linear operator for the Whitham equation (see \eqref{E:Whitham}) consists of $i\omega(n+\xi,-)$ for $n\in\mathbb{Z}$ and $\xi\in[-1/2,1/2)$; see \cite{HJ2}, for instance, for details. 
Perhaps, this is because the Whitham equation merely includes unidirectional propagation. In the following section, we discuss the effects of bidirectional propagation in \eqref{E:main}-\eqref{def:c}.

As $|a|$ increases, the eigenvalues in \eqref{E:Lx0=0} move around and they may leave the imaginary axis to lose the spectral stability. Recall that the spectrum of $\L(\pm\xi,a)$ is symmetric with respect to the reflections in the real and imaginary axes for any $\xi\in[0,1/2]$ for any $a\in\mathbb{R}$ and admissible. Hence a necessary condition of the spectral instability is that a pair of eigenvalues on the imaginary axis collide. 

Note that the eigenfunctions in \eqref{E:Lx0=0} vary, analytically, with $\xi\in[0,1/2]$. To compare, the eigenfunctions of the linear operator for the Whitham equation do not depend on $\xi$; see \cite{HJ2}, for instance, for details.

To proceed, for $\xi=0$, note from \eqref{def:eigen} that
\[
\omega(0,+)=\omega(0,-)=\omega(1,-)=\omega(-1,-)=0.
\]
Since
\begin{align*}
\cdots<\omega(-3,-)<\omega(-2,-)<&0<\omega(2,-)<\omega(3,-)<\cdots
\intertext{and}
\cdots<\omega(-2,+)<\omega(-1,+)<&0<\omega(1,+)<\omega(2,+)<\dots
\end{align*}
by brutal force, zero is an $L^2(\mathbb{T})\times L^2(\mathbb{T})$ eigenvalue of $\L(0,0)$ with multiplicity four. Note that
\begin{equation}\label{def:p00}
\begin{aligned}
&\p_{1}(z):=\frac12(\e(1,-)+\e(-1,-))(z)=\begin{pmatrix}1\\ \c(\k) \end{pmatrix}\cos z, \\
&\p_{2}(z):=\frac{1}{2i}(\e(1,-)-\e(-1,-))(z)=\begin{pmatrix} 1\\ \c(\k) \end{pmatrix}\sin z,\\
&\p_{3}(z):=\frac12((\c(\k)+2)\e(0,+)-(\c(\k)-2)\e(0,-))(z)=\begin{pmatrix} 2\\ -\c(\k)\end{pmatrix},\\
&\p_{4}(z):=\frac12((\c(\k)-2)\e(0,+)+(\c(\k)+2)\e(0,-))(z)=\begin{pmatrix} \c(\k)\\ 2\end{pmatrix}
\end{aligned}
\end{equation}
are the associated eigenfunctions, real valued and orthogonal to each other. 

For $\xi\neq 0$, since $\omega(n+\xi,+)$ increases in $n+\xi$ for any $n\in\mathbb{Z}$ and $\xi\in(0,1/2]$, and since $\omega(n+\xi,-)$ decreases in $n+\xi$ if $-1/2<n+\xi<1/2$ and increases if $n+\xi<-1$ or $n+\xi>1$ by brutal force, it follows that
\[
\omega(1/2,-)\leq\omega(0+\xi,\pm), \omega(\pm1+\xi,-)\leq \omega(1/2,+),
\]
and
\begin{gather*}
\cdots<\omega(-2+\xi,-)<\omega(\xi,-)<0<\omega(-1+\xi,-)<\omega(1+\xi,-)<\omega(2+\xi,-)<\cdots,\\
\cdots<\omega(-2+\xi,+)<\omega(-1+\xi,+)<0<\omega(\xi,+)<\omega(1+\xi,+)<\omega(2+\xi,+)<\cdots.
\end{gather*}
Hence $\omega(n+\xi,\pm)\neq 0$ for any $n\in\mathbb{Z}$ and $\xi\in(0,1/2]$. 
But in Section~\ref{ss:collision}, we observe infinitely many collisions of purely imaginary eigenvalues of $\L(\xi,0)$ away from the origin. To compare, no eigenvalues of the linear operator for the Whitham equation (see \eqref{E:Whitham}) collide other than at the origin; see \cite{HJ2}, for instance.
 
Continuing, for $\xi>0$ and sufficiently small, $i\omega(\xi,\pm)$ and $i\omega(\pm1+\xi,-)$ are $L^2(\mathbb{T})\times L^2(\mathbb{T})$ eigenvalues of $\L(\xi,0)$ in the vicinity of the origin in $\mathbb{C}$. Moreover, (by abuse of notation)
\begin{equation}\label{def:px0}
\begin{aligned}
\p_1(z):=&\frac12\sqrt{\c^2(\k)+1}\left(\frac{\e(1+\xi,-)}{\|\e(1+\xi,-)\|}+\frac{\e(-1+\xi,-)}{\|\e(-1+\xi,-)\|}\right)(z)\\
=&\begin{pmatrix}1\\ \c(\k) \end{pmatrix}\cos z
+i\xi \frac{\k\c'(\k)}{\c^2(\k)+1}\begin{pmatrix} -\c(\k)\\1\end{pmatrix}\sin z
+\xi^2 \mathbf{p}_2\cos z+O(\xi^3),\hspace*{-20pt}\\
\p_2(z):=&\frac{1}{2i}\sqrt{\c^2(\k)+1}\left(\frac{\e(1+\xi,-)}{\|\e(1+\xi,-)\|}-\frac{\e(-1+\xi,-)}{\|\e(-1+\xi,-)\|}\right)(z)\\
=&\begin{pmatrix}1\\ \c(\k) \end{pmatrix}\sin z
-i\xi\frac{ \k\c'(\k)}{\c^2(\k)+1}\begin{pmatrix}-\c(\k)\\ 1\end{pmatrix}\cos z 
+\xi^2 \mathbf{p}_2\sin z+O(\xi^3),\hspace*{-20pt} \\
\p_3(z):=&\frac12((\c(\k)+2)\e(\xi,+)-(\c(\k)-2)\e(\xi,-))(z) \\
=&\begin{pmatrix}2\\-\c(\k)\end{pmatrix}+\frac16\xi^2\k^2\c(\k)\begin{pmatrix} 0\\1\end{pmatrix}+O(\xi^3),\\
\p_4(z):=&\frac12((\c(\k)-2)\e(\xi,+)+(\c(\k)+2)\e(\xi,-))(z) \\
=&\begin{pmatrix}\c(\k)\\ 2 \end{pmatrix}-\frac13\xi^2\k^2\begin{pmatrix} 0\\1\end{pmatrix}+O(\xi^3)
\end{aligned}
\end{equation}
span the associated eigenspace, orthogonal to each other, where $\|\cdot\|=\sqrt{\l\cdot\,,\cdot\r_{L^2(\mathbb{T})\times L^2(\mathbb{T})}}$ and 
\begin{equation}\label{def:p2}
\mathbf{p}_2=\frac12\frac{\k^2}{\c^2(\k)+1} 
\begin{pmatrix}{\displaystyle \c'(\k)^2\frac{2\c^2(\k)-1}{\c^2(\k)+1}-(\c\c'')(\k)}\\ 
{\displaystyle -3\frac{(\c(\c')^2)(\k)}{\c^2(\k)+1}+\c''(\k)} \end{pmatrix}.
\end{equation}
Here and elsewhere, the prime means ordinary differentiation.
For $\xi=0$, note that $\p_1$, $\p_2$, $\p_3$, $\p_4$ become \eqref{def:p00}. Recall that $\c$ is a real analytic function. Hence they depend analytically on $\xi\in[0,1/2]$. 

Note that $\p_1$ and $\p_2$ vary with $\xi>0$ and sufficiently small to the linear order. In the following subsection, we take this into account and construct an eigenspace for $\xi$, $a\neq0$ and sufficiently small, which varies analytically with $\xi$ and $a$; see \eqref{def:p} for details. Consequently, the spectral perturbation calculation in Section~\ref{ss:perturbation} becomes lengthy and complicated. To compare, the eigenfunctions of the linear operator for the Whitham equation (see \eqref{E:Whitham}) do not depend on $\xi$ for any $a\in\mathbb{R}$ and admissible; see \cite{HJ2}, for instance, for details.

Note that $\p_{1}$ and $\p_{2}$ are complex valued. For real valued functions, one must take $\pm\xi$ in pair and deal with six functions. But the spectral perturbation calculation in Section~\ref{ss:perturbation} involves complex valued operators anyway. Hence this is not worth the effort. 

\subsection{Spectra of $\L(\xi,a)$}\label{ss:a}

We turn the attention to the $L^2(\mathbb{T})\times L^2(\mathbb{T})$ spectra of $\L(\xi,a)$ in the vicinity of the origin in $\mathbb{C}$, for $\xi\in[0,1/2]$ for $a\in\mathbb{R}$ and $|a|$ sufficiently small. 

Note from \eqref{def:Lx} and \eqref{E:hu-small} that 
\begin{align*}
\L(\xi,a)=&e^{-i\xi z}\partial_z \begin{pmatrix} c-u & -1-\eta \\ -\c^2(\k|\partial_z|) & c-u \end{pmatrix} e^{i\xi z}\\
=&e^{-i\xi z}\partial_z \left(\begin{pmatrix} \c(\k) & -1 \\ -\c^2(\k|\partial_z|) & \c(\k) \end{pmatrix} 
+a\begin{pmatrix} -\c(\k) & -1 \\ 0 & -\c(\k) \end{pmatrix}\cos z\right)e^{i\xi z}+O(a^2)
\end{align*}
as $a\to 0$, whence 
\[
\|\L(\xi,a)-\L(\xi,0)\|_{L^2(\mathbb{T})\times L^2(\mathbb{T}) \to L^2(\mathbb{T})\times L^2(\mathbb{T})}=O(a)
\]
as $a\to 0$ uniformly for $\xi\in[0,1/2]$. Recall from the previous subsection that the $L^2(\mathbb{T})\times L^2(\mathbb{T})$ spectrum of $\L(\xi,0)$ contains four purely imaginary eigenvalues $i\omega(\xi,\pm)$, $i\omega(\pm1+\xi,-)$ in the vicinity of the origin in $\mathbb{C}$ for $\xi>0$ and sufficiently small. Since $\L(\xi,a)$ depends analytically on $\xi\in[0,1/2]$ and $a\in\mathbb{R}$ admissible, it follows from perturbation theory (see \cite[Section~4.3.1]{K}, for instance, for details) that the $L^2(\mathbb{T})\times L^2(\mathbb{T})$ spectrum of $\L(\xi,a)$ contains four eigenvalues, denoted 
\[
\lambda_1(\xi,a), \lambda_2(\xi,a), \lambda_3(\xi,a), \lambda_4(\xi,a),
\] 
near the origin for $\xi>0$, $a\in\mathbb{R}$ and $\xi$, $|a|$ sufficiently small. 

Moreover, a straightforward calculation reveals that
\[
|\lambda_{k}(\xi,0)-\lambda_{\ell}(\xi,0)|\geq\omega_0>0\qquad \text{for $k,\ell=1,2,3,4$ and $k\neq \ell$}
\]
for any $\xi\geq \xi_0>0$ for any $\xi_0$ for some $\omega_0$. Hence it follows from perturbation theory that $\lambda_1$, $\lambda_2$, $\lambda_3$, $\lambda_4$ remain purely imaginary for any $\xi\geq \xi_0>0$ for any $\xi_0$, for any $a\in\mathbb{R}$ and $|a|$ sufficiently small. In particular, a sufficiently small, periodic wave train of \eqref{E:main}-\eqref{def:c} is spectrally stable to ``short wavelength perturbations" in the vicinity of the origin in $\mathbb{C}$. For $\xi=0$, on the other hand, we demonstrate that four eigenvalues collide at the origin. 

\begin{lemma}[Spectrum of $\L(0,a)$]\label{lem:L}
For $a\in\mathbb{R}$ and $|a|$ sufficiently small, zero is an $L^2(\mathbb{T})\times L^2(\mathbb{T})$ eigenvalue of $\L(0,a)$ with algebraic multiplicity four and geometric multiplicity three. Moreover, (by abuse of notation)
\begin{subequations}\label{def:p0a}
\begin{align} 
\p_{1}(z):=&\frac{2}{\cc^2(\k)+2}((\partial_{b_1}c)(\partial_a\u)-(\partial_{a}c)(\partial_{b_1}\u))(z) \label{def:p1}\\
=&\begin{pmatrix} 1\\ \cc(\k)\end{pmatrix}\cos z 
+a\begin{pmatrix}-3h_2\dfrac{\cc^2(\k)}{\cc^2(\k)+2}\\ 
\cc(\k)\Big(\frac12-3h_2\dfrac{1}{\cc^2(\k)+2}\Big)\end{pmatrix}\notag \\
&\hspace*{68pt}+2a\begin{pmatrix}h_2\\ \cc(\k)\left(h_2-\frac12\right)\end{pmatrix}\cos 2z+O(a^2),\notag\\
\p_{2}(z):=&-\frac{1}{a}(\partial_z\u)(z) \label{def:p2} \\
=&\begin{pmatrix} 1\\ \cc(\k)\end{pmatrix}\sin z
+2a\begin{pmatrix}h_2\\ \cc(\k)\left(h_2-\frac12\right)\end{pmatrix}\sin 2z+O(a^2),\notag
\intertext{where $h_2$ is in \eqref{def:h02}, and}
\p_{3}(z):=&\frac{2}{\cc^2(\k)-1}
((\partial_{b_1}c)(\partial_{b_2}\u)-(\partial_{b_2}c)(\partial_{b_1}\u))(z)\label{def:p3}\\
=&\begin{pmatrix} 2\\-\cc(\k)\end{pmatrix}+a\begin{pmatrix}2\\ \cc(\k) \end{pmatrix}\cos z+O(a^2),\notag\\
\p_{4}(z):=&\frac13\frac{\cc^2(\k)+4}{\cc(\k)}
\Big((\partial_{b_2}\u)+\frac{\cc^2(\k)-2}{\cc^2(\k)+4}\p_{3,a}\Big)(z) \label{def:p4} \\
=&\begin{pmatrix}\cc(\k)\\2 \end{pmatrix}+a\begin{pmatrix}\cc(\k)\\ \frac12\cc^2(\k)\end{pmatrix}\cos z+O(a^2), \notag
\end{align}
\end{subequations}
are the associated eigenfunctions. Specifically,
\[
\L(0,a)\p_{k}=0 \quad \text{for $k=1,2,3,$}\quad\text{and}\quad \L(0,a)\p_{4}=\frac14a(\cc^2(\k)+1)\p_{2}.
\]
\end{lemma}

For $a=0$, note that $\p_1$, $\p_2$, $\p_3$, $\p_4$ becomes \eqref{def:p00}. Theorem~\ref{thm:existence} implies that they depend analytically on $a\in\mathbb{R}$ and $|a|$ sufficiently small. 

\begin{proof}
Exploiting variations of \eqref{E:periodic} in the $z$, and $a$, $b_1$, $b_2$ variables, the proof is similar to that of \cite[Lemma~3.1]{HJ2}, for instance. Here we include the details for the sake of completeness.

Differentiating \eqref{E:periodic} with respect to $z$ and evaluating the result at $b_1=b_2=0$, we infer from \eqref{def:L} that
\[
\mathcal{L}(0,a)(\partial_z \u)=0. 
\]
Hence zero is an eigenvalue of $\mathcal{L}(0,a)$ and $\partial_z\u$ is an associated eigenfunction. We then use \eqref{E:h-small} and \eqref{E:u-small} to find \eqref{def:p2}. By the way, this is reminiscent of that \eqref{E:periodic} remains invariant under spatial translations. 

Differentiating \eqref{E:periodic} with respect to $a$, and $b_1$, $b_2$, and evaluating at $b_1=b_2=0$, we infer from \eqref{def:L} that
\[
\mathcal{L}(0,a)(\partial_a\u)=-(\partial_a c)(\partial_z \u),
\]
and 
\[
\mathcal{L}(0,a)(\partial_{b_1}\u)=-(\partial_{b_1} c)(\partial_z \u),\qquad
\mathcal{L}(0,a)(\partial_{b_2}\u)=-(\partial_{b_2} c)(\partial_z \u).
\]
Hence
\begin{align*}
\L(0,a)((\partial_{b_1} c)(\partial_a\u)-(\partial_a c)(\partial_{b_1}\u))=0\quad\text{and}\quad
\L(0,a)((\partial_{b_1} c)(\partial_{b_2}\u)-(\partial_{b_2} c)(\partial_{b_1}\u))=0.
\end{align*}
We then use \eqref{E:hu-small} and \eqref{E:hu0} to find \eqref{def:p1} and \eqref{def:p3}. Note that $\partial_{b_2}\u$ is a generalized eigenfunction. We use \eqref{E:hu-small} and \eqref{E:hu0} to find \eqref{def:p4}. This completes the proof.
\end{proof}

To recapitulate, for $\xi>0$ and sufficiently small for $a=0$, the $L^2(\mathbb{T})\times L^2(\mathbb{T})$ spectrum of $\L(\xi,0)$ contains four purely imaginary eigenvalues $i\omega(\xi,\pm)$, $i\omega(\pm1+\xi,-)$ in the vicinity of the origin in $\mathbb{C}$, and \eqref{def:px0} spans the associated eigenspace, which depends analytically on $\xi$. For $\xi=0$ for $a\in\mathbb{R}$ and $|a|$ sufficiently small, the spectrum of $\L(0,a)$ contains four eigenvalues at the origin, and \eqref{def:p0a} makes the associated eigenfunctions, which depends analytically on $a$. 

For $\xi>0$, $a\in\mathbb{R}$ and $\xi$, $|a|$ sufficiently small, the $L^2(\mathbb{T})\times L^2(\mathbb{T})$ spectrum of $\L(\xi,a)$ contains four eigenvalues $\lambda_1(\xi,a)$, $\lambda_2(\xi,a)$, $\lambda_3(\xi,a)$, $\lambda_4(\xi,a)$ near the origin, and the associated eigenfunctions vary analytically from \eqref{def:px0} and \eqref{def:p0a}. Let (by abuse of notation)
\begin{equation}\label{def:p}
\begin{aligned}
\p_1(\xi,a)(z)=&\begin{pmatrix}1\\ \c(\k) \end{pmatrix}\cos z 
+i\xi \frac{\k\c'(\k)}{\c^2(\k)+1}\begin{pmatrix} -\c(\k)\\ 1\end{pmatrix}\sin z \\
&+a\begin{pmatrix}-3h_2\dfrac{\c^2(\k)}{\c^2(\k)+2}\\ \c(\k)\Big(\frac12-3h_2\dfrac{1}{\c^2(\k)+2}\Big)\end{pmatrix}
+2a\begin{pmatrix}h_2\\ \c(\k)\left(h_2-\frac12\right)\end{pmatrix}\cos 2z \\
&+\xi^2 \mathbf{p}_2\cos z+O(\xi^3+\xi^2a+a^2), \\
\p_2(\xi,a)(z)=&\begin{pmatrix}1\\ \c(\k) \end{pmatrix}\sin z 
-i\xi \frac{\k\c'(\k)}{\c^2(\k)+1}\begin{pmatrix} -\c(\k)\\ 1\end{pmatrix}\cos z\\
&+2a\begin{pmatrix}h_2\\ \c(\k)\left(h_2-\frac12\right)\end{pmatrix}\sin 2z
+\xi^2 \mathbf{p}_2\sin z+O(\xi^3+\xi^2a+a^2),\\
\p_3(\xi,a)(z)=&\begin{pmatrix} 2\\-\c(\k)\end{pmatrix}+a\begin{pmatrix}2\\ \c(\k)\end{pmatrix}\cos z
+\frac16\xi^2\k^2\c(\k)\begin{pmatrix} 0\\1\end{pmatrix}+O(\xi^3+\xi^2a+a^2), \\
\p_4(\xi,a)(z)=&\begin{pmatrix} \c(\k)\\2\end{pmatrix}+a\begin{pmatrix}\c(\k)\\ \frac12\c^2(\k) \end{pmatrix}\cos z
-\frac13\xi^2 \k^2\begin{pmatrix} 0\\1\end{pmatrix}+O(\xi^3+\xi^2a+a^2)
\end{aligned}
\end{equation}
as $\xi$, $a\to0$, where $h_2$ is in \eqref{def:h02} and $\mathbf{p}_2$ is in \eqref{def:p2}. For $a=0$, note that $\p_1$, $\p_2$, $\p_3$, $\p_4$ become \eqref{def:px0}. For $\xi=0$, they become \eqref{def:p0a}. Hence $\p_1$, $\p_2$, $\p_3$, $\p_4$ span the eigenspace associated with $\lambda_1$, $\lambda_2$, $\lambda_3$, $\lambda_4$ up to terms of orders $\xi^2$ and $a$ as $\xi$, $a\to 0$. 

It seems impossible to uniquely determine terms of orders $\xi a$ and higher in the eigenfunction expansion without  ad hoc orthogonality conditions. Fortuitously, it turns out that they do not contribute to the modulational instability. Hence we may neglect them in \eqref{def:p}. To compare, the eigenfunctions of the linear operator for the Whitham equation (see \eqref{E:Whitham}), which do not depend on $\xi$, extend to $a\neq0$; see \cite{HJ2}, for instance, for details. We are able to calculate terms of orders $a^2$ and higher in the eigenfunction expansion. But the index formulae become unwieldy. Hence we do not use them in the calculation in the following subsection.

\subsection{Spectral perturbation calculation}\label{ss:perturbation}

Recall that for $\xi>0$, $a\in\mathbb{R}$ and $\xi$, $|a|$ sufficiently small, the $L^2(\mathbb{T}) \times L^2(\mathbb{T})$ spectrum of $\L(\xi,a)$ contains four eigenvalues $\lambda_1(\xi,a)$, $\lambda_2(\xi,a)$, $\lambda_3(\xi,a)$, $\lambda_4(\xi,a)$ in the vicinity of the origin in $\mathbb{C}$, and \eqref{def:p} spans the associated eigenspace up to terms of orders $\xi^2$ and $a$. 
Let 
\begin{equation}\label{def:L3}
\mathbf{L}(\xi,a)=\left(\frac{\l\L(\xi,a)\p_k(\xi,a),\p_\ell(\xi,a)\r}{\l\p_k(\xi,a),\p_k(\xi,a)\r}\right)_{k,\ell=1,2,3,4}
\end{equation}
and 
\begin{equation}\label{def:I3}
\mathbf{I}(\xi,a)=\left(\frac{\l\p_k(\xi,a),\p_\ell(\xi,a)\r}{\l\p_k(\xi,a),\p_k(\xi,a)\r}\right)_{k,\ell=1,2,3,4},
\end{equation}
where $\p_1$, $\p_2$, $\p_3$, $\p_4$ are in \eqref{def:p}. Throughout the subsection, $\langle\,,\rangle$ means the $L^2(\mathbb{T})\times L^2(\mathbb{T})$ inner product. Note that $\mathbf{L}$ represents the action of $\L(\xi,a)$ on the eigenspace associated with $\lambda_1$, $\lambda_2$, $\lambda_3$, $\lambda_4$, up to the orders of $\xi^2$ and $a$ as $\xi$, $a\to0$, after normalization, and $\mathbf{I}$ is the projection of the identity onto the eigenspace. 
It follows from perturbation theory (see \cite[Section~4.3.5]{K}, for instance, for details) that for $\xi>0$, $a\in\mathbb{R}$ and $\xi$, $|a|$ sufficiently small, the roots of $\det(\mathbf{L}-\lambda\mathbf{I})$ coincide with the eigenvalues of $\L(\xi,a)$ up to terms of orders $\xi^2$ and $a$. 

For any $a\in\mathbb{R}$ and $|a|$ sufficiently small, we make a Baker-Campbell-Hausdorff expansion to write 
\[
\L(\xi,a)=\L_0+i\xi\L_1-\frac12\xi^2\L_2+O(\xi^3)
\]
as $\xi\to0$, where
\begin{align*}
\L_0=&\L(0,a)=\partial_z\begin{pmatrix} \c(\k) & -1\\ -\c^2(\k|\partial_z|) & \c(\k)\end{pmatrix}
-a\partial_z\begin{pmatrix} \c(\k) & 1\\ 0 & \c(\k)\end{pmatrix}\cos z+O(a^2),\\
\L_1=&[\L_0,z]=\begin{pmatrix} \c(\k) & -1\\ -[\partial_z\c^2(\k|\partial_z|),z] & \c(\k)\end{pmatrix}
-a\begin{pmatrix} \c(\k) & 1\\ 0 & \c(\k)\end{pmatrix}\cos z+O(a^2), \\ 
\L_2=&[\L_1,z]=\begin{pmatrix} 0 & 0\\ -[[\partial_z\c^2(\k|\partial_z|),z],z] & 0\end{pmatrix}+O(a^2)
\end{align*}
as $a\to0$, and $[\cdot,\,\cdot]$ means the commutator. The latter equalities follow from \eqref{def:L}, \eqref{def:Lx}, \eqref{E:hu-small} and that $\L(\xi,a)$ depends analytically on $\xi$ near $\xi=0$. We merely pause to remark that $\L_1$ and $\L_2$ are well defined in the periodic setting even though $z$ is not. Indeed, $[\partial_z\c^2(\k|\partial_z|),z]=\c^2(\k|\partial_z|)+[\c^2(\k|\partial_z|),z]\partial_z$ and
\[
[\c^2(\k|\partial_z|),z]e^{inz}=i e^{inz}\sum_{m\neq 0} \frac{(-1)^{|m|+1}}{m}
(-\c^2(\k n)-\c^2(\k (n+m)))e^{imz} \quad \text{for $n\in \mathbb{Z}$}
\]
by brutal force.
One may likewise represent $[[\partial_z\c^2(\k|\partial_z|),z],z]$ in the Fourier series. Unfortunately, this is not convenient for an explicit calculation. We instead rearrange the above as 
\begin{align}\label{E:Lexp}
\mathcal{L}(\xi,a)
=&\mathcal{L}(0,0)+i\xi [\mathcal{L}(0,0),z]-\frac{\xi^2}{2}[[\mathcal{L}(0,0),z],z]\hspace*{-20pt}\notag\\
&-a\partial_z\begin{pmatrix} \c(\k) & 1 \\ 0 & \c(\k)\end{pmatrix}\cos z
-i\xi a\begin{pmatrix} \c(\k) & 1 \\ 0 & \c(\k) \end{pmatrix}\cos z+O(\xi^3+\xi^2 a+a^2)\hspace*{-20pt}\notag\\
=&:\mathcal{M}-a\partial_z\begin{pmatrix} c(\k)&1\\0&c(\k)\end{pmatrix}\cos z
-i\xi a\begin{pmatrix} \c(\k) & 1 \\ 0 & \c(\k)\end{pmatrix}\cos z+O(\xi^3+\xi^2 a+a^2)\hspace*{-20pt}
\end{align}
as $\xi$, $a\to 0$, and note that $\mathcal{M}$ agrees with $\L(\xi,0)$ up to terms of order $\xi^2$ for $\xi>0$ and sufficiently small. We then resort to \eqref{def:ck} and make an explicit calculation to find
\begin{align*}
\L(\xi,0)\begin{pmatrix} \zeta \\ v \end{pmatrix}e^{inz} 
=&in\begin{pmatrix} \c(\k)\zeta -v \\ -\c^2(\k n)\zeta+\c(\k)v \end{pmatrix}e^{inz}\\
&+i\xi\begin{pmatrix} \c(\k)\zeta-v \\-\left(\c^2(\k n)+\k n(\c^2)'(\k n)\right)\zeta+ \c(\k)v \end{pmatrix}e^{inz}\\
&-i\xi^2\,\k n\Big((\c^2)'(\k n)+\frac12\k n(\c^2)''(\k n)\Big)\zeta\begin{pmatrix}0\\1 \end{pmatrix}e^{inz}+O(\xi^3) \\
=&\mathcal{M}\begin{pmatrix} \zeta \\ v \end{pmatrix}e^{inz} +O(\xi^3)
\end{align*}
as $\xi\to 0$, for any constants $\zeta$, $v$ and $n\in\mathbb{Z}$. For instance, since $\c^2(0)=1$ and $(\c^2)'(0)=0$, it follows that 
\[
\M\begin{pmatrix} \zeta \\ v \end{pmatrix}
=i\xi\begin{pmatrix}\c(\k)\zeta-v \\ \c(\k)v-\zeta \end{pmatrix}.
\]
One may likewise calculate $\M\begin{pmatrix}\zeta \\ v\end{pmatrix}\left\{\begin{matrix} \cos nz\\ \sin nz\end{matrix}\right\}$ explicitly up to the order of $\xi^2$. We omit the details. 

We use \eqref{E:Lexp}, \eqref{def:p}, and the above formula for $\M$, and we make a lengthy and complicated, but explicit, calculation to show that
\begin{align*}
\L\p_1=&-2i\xi\,\k(\c\c')(\k)\begin{pmatrix}0\\1 \end{pmatrix}\cos z
+i\xi\,\k\c'(\k)\begin{pmatrix} -1\\ \c(\k)\end{pmatrix}\cos z \\
&-\frac12i\xi a\,\c(\k)\begin{pmatrix} 2\\ \c(\k)\end{pmatrix} (\cos 2z+1)
+i\xi a\,\frac{\k\c'(\k)}{\c^2(\k)+1}\begin{pmatrix}\c^2(\k)-1\\ -\c(\k) \end{pmatrix} \cos 2z \\
&-\frac12i\xi a\,\c(\k)\begin{pmatrix} 6h_2\dfrac{\c^2(\k)-1}{\c^2(\k)+2}+1\\ -\c(\k) \end{pmatrix} \\
&+\frac12 i\xi\,a\c(\k)\begin{pmatrix}2\\ \c(\k)\Big(1-12\k\dfrac{(\c\c')(2\k)}{\c^2(\k)-\c^2(2\k)}\Big)\end{pmatrix}\cos 2z\\
&+\xi^2\,\k(2(\c\c')(\k)+\k((\c')^2+\c\c'')(\k))\begin{pmatrix}0\\1 \end{pmatrix} \sin z\\
&-\xi^2\,\k\c(\k)\begin{pmatrix}-1\\ \c(\k)\Big(1+2\k\dfrac{(\c\c')(\k)}{\c^2(\k)+1}\Big) \end{pmatrix}\sin z\\
&+\frac12\xi^2\,\k^2\Big(2\frac{(\c(\c')^2)(\k)}{\c^2(\k)+1}-\c''(\k)\Big)\begin{pmatrix}-1\\ \c(\k) \end{pmatrix}\sin z
+O(\xi^3+\xi^2 a+a^2)
\end{align*}
as $\xi$, $a\to 0$, where $h_2$ is in \eqref{def:h02}. Moreover, 
\begin{align*}
\L\p_2=&-2i\xi\,\k (\c\c')(\k)\begin{pmatrix}0\\1 \end{pmatrix}\sin z
+i\xi\,\k\c'(\k)\begin{pmatrix} -1\\ \c(\k)\end{pmatrix}\sin z \\
&-\frac12i\xi a\,\c(\k)\begin{pmatrix} 2\\ \c(\k)\end{pmatrix} \sin 2z
+i\xi a\,\frac{ \k\c'(\k)}{\c^2(\k)+1}\begin{pmatrix}\c^2(\k)-1\\-\c(\k) \end{pmatrix} \sin 2z \\
&+\frac12 i\xi a\,\c(\k)\begin{pmatrix}2\\ \c(\k)\Big(1-12\k\dfrac{(\c\c')(2\k)}{\c^2(\k)-\c^2(2\k)}\Big)\end{pmatrix}\sin 2z\\
&-\xi^2\,\k(2\c\c'+\k((\c')^2+\c\c''))(\k)\begin{pmatrix}0\\1 \end{pmatrix} \cos z\\
&+\xi^2\,\k\c(\k)\begin{pmatrix}-1\\ \c(\k)\Big(1+2\k\dfrac{(\c\c')(\k)}{\c^2(\k)+1}\Big) \end{pmatrix}\cos z\\
&-\frac12\xi^2\,\k^2\Big(2\frac{(\c(\c')^2)(\k)}{\c^2(\k)+1}-\c''(\k)\Big)\begin{pmatrix}-1\\ \c(\k) \end{pmatrix}\cos z
+O(\xi^3+\xi^2 a+a^2),
\intertext{and}
\L\p_3=&i\xi\begin{pmatrix} 3\c(\k)\\-\c^2(\k)-2\end{pmatrix}
-2i\xi a\,\k(\c\c')(\k)\begin{pmatrix} 0\\1  \end{pmatrix}\cos z+O(\xi^3+\xi^2 a+a^2), \\
\L\p_4=&i\xi\begin{pmatrix} \c^2(\k)-2\\ \c(\k) \end{pmatrix}
+\frac12 a(\c^2(\k)+4) \begin{pmatrix} 1\\ \c(\k) \end{pmatrix} \sin z\\
&-\frac12 i\xi a\begin{pmatrix} \c^2(\k)+4\\ \c(\k)(\c^2(\k)+4+2\k(\c\c')(\k))\end{pmatrix}\cos z +O(\xi^3+\xi^2 a+a^2) 
\end{align*}
as $\xi$, $a\to 0$. 

To proceed, we take the $L^2(\mathbb{T})\times L^2(\mathbb{T})$ inner products of the above and \eqref{def:p}, and we make a lengthy and complicated, but explicit, calculation to show that 
\begin{align*}
\langle \mathcal{L}\p_1,\p_1\rangle=&\langle \mathcal{L}\p_2,\p_2\rangle \\
=&-\frac{1}{2} i\xi\,\k\c'(\k)(\c^2(\k)+1)+O(\xi^3+\xi^2 a+a^2),\\
\langle \mathcal{L}\p_1,\p_2\rangle =&-\langle \mathcal{L}\p_2,\p_1\rangle \\
=&\frac12\xi^2\Big(\k\c'(\k)+\frac12\k^2\c''(\k)\Big)(\c^2(\k)+1) +O(\xi^3+\xi^2 a+a^2),\\
\langle \mathcal{L}\p_1,\p_3\rangle=&\frac{2}{\c(\k)}\langle \mathcal{L}\p_1,\p_4\rangle\\
=&-3i\xi a\Big(2h_2c(\k)\frac{\c^2(\k)-1}{\c^2(\k)+2}+\c(\k)+\frac16\k\c'(\k)(\c^2(\k)+2) \Big)\\ 
&+O(\xi^3+\xi^2 a+a^2)
\intertext{as $\xi$, $a\to 0$, where $h_2$ is in \eqref{def:h02}. Moreover,}
\langle \mathcal{L}\p_2,\p_3\rangle=&\langle \mathcal{L}\p_2,\p_4\rangle
=0 +O(\xi^3+\xi^2 a+a^2),\\
\langle \mathcal{L}\p_3,\p_1\rangle
=&-i\xi a\,\c(\k)\Big(6h_2\frac{\c^2(\k)+1}{\c^2(\k)+2}+\frac12(\c^2(\k)+2)+2\k(\c\c')(\k) \Big)\\
&+O(\xi^3+\xi^2 a+a^2),  \\
\langle \mathcal{L}\p_3,\p_2\rangle=& 0+O(\xi^3+\xi^2 a+a^2),  \\
\langle \mathcal{L}\p_3,\p_3\rangle=& i\xi\c(\k)(\c^2(\k)+8)+O(\xi^3+\xi^2 a+a^2),\\
\langle \mathcal{L}\p_3,\p_4\rangle=&\langle \mathcal{L}\p_4,\p_3\rangle
=i\xi(\c^2(\k)-4)+O(\xi^3+\xi^2 a+a^2) ,
\intertext{and}
\langle \mathcal{L}\p_4,\p_1\rangle
=&-i\xi a\Big(\frac14(\c^4+3\c^2)(\k)+1+3h_2\c^2(\k)\frac{\c^2(\k)-1}{\c^2(\k)+2}+\frac12\k(\c^3\c')(\k)\Big)\\
& +O(\xi^3+\xi^2 a+a^2),  \\
\langle \mathcal{L}\p_4,\p_2\rangle=&\frac14a(\c^2(\k)+4)(\c^2(\k)+1) +O(\xi^3+\xi^2 a+a^2),\\   
\langle \mathcal{L}\p_4,\p_4\rangle=& i\xi \c^3(\k)+O(\xi^3+\xi^2 a+a^2)
\end{align*}
as $\xi$, $a\to 0$, where $h_2$ is in \eqref{def:h02}.

Continuing, we take the $L^2(\mathbb{T})\times L^2(\mathbb{T})$ inner products of \eqref{def:p} and we make an explicit calculation to show that
\begin{align*}
\langle\p_1,\p_1\rangle=&\langle \p_2,\p_2 \rangle
=\frac12(\c^2(\k)+1)-\frac34\xi^2\k^2\frac{\c'(\k)^2}{\c^2(\k)+1}+O(\xi^3+\xi^2 a+a^2),\\
\langle\p_1,\p_2\rangle=&\langle\p_2,\p_1\rangle= 0+O(\xi^3+\xi^2 a+a^2),\\
\langle\p_1,\p_3\rangle=&\langle\p_3,\p_1\rangle
=a\Big(1-3h_2\frac{\c^2(\k)}{\c^2(\k)+2}\Big)+O(\xi^3+\xi^2 a+a^2),\\
\langle\p_1,\p_4\rangle=&\langle\p_4,\p_1\rangle
=\frac14a\,\c(\k)(\c^2(\k)+6-12h_2)+O(\xi^3+\xi^2 a+a^2)
\intertext{as $\xi$, $a\to 0$, where $h_2$ is in \eqref{def:h02}. Moreover,}
\langle \p_2,\p_3 \rangle=&\langle \p_3,\p_2 \rangle= 2\langle \p_2,\p_4 \rangle=2\langle \p_4,\p_2 \rangle
=\frac12 i\xi a\,\frac{\k(\c\c')(\k)}{\c^2(\k)+1}+O(\xi^3+\xi^2 a+a^2),  \\
\langle \p_3,\p_3 \rangle=&\langle \p_4,\p_4 \rangle= \c^2(\k)+4+O(\xi^3+\xi^2 a+a^2),\\
\langle \p_3,\p_4 \rangle=&\langle \p_4,\p_3 \rangle=0 +O(\xi^3+\xi^2 a+a^2) 
\end{align*}
as $\xi$, $a\to 0$. 

Together, \eqref{def:L3} becomes
\begin{equation}\label{E:L3}
\begin{split}
\mathbf{L}(\xi,a)=
&\frac14 a(\c^2(\k)+1)\begin{pmatrix} 0&0&0&0\\0&0&0&0\\0&0&0&0\\0&1&0&0\end{pmatrix}\\
&+i\xi\begin{pmatrix} -\k\c'(\k)&0&0&0\\ 0&-\k\c'(\k)&0&0\\ 
0&0&\c(\k)\dfrac{\c^2(\k)+8}{\c^2(\k)+4}&\dfrac{\c^2(\k)-4}{\c^2(\k)+4}\\
0&0&\dfrac{\c^2(\k)-4}{\c^2(\k)+4}&\dfrac{\c^3(\k)}{\c^2(\k)+4}\end{pmatrix} \\
&+i\xi a\,L\begin{pmatrix} 0&0&2&\c(\k)\\ 0&0&0&0\\0&0&0&0\\0&0&0&0 \end{pmatrix} 
-i\xi a\frac{1}{\c^2(\k)+4} \begin{pmatrix}0&0&0&0\\0&0&0&0\\
L_{31}&0&0&0\\
L_{41}&0&0&0\end{pmatrix}\\
&+\frac12\xi^2\k(2\c'(\k)+\k\c''(\k))\begin{pmatrix}0&1&0&0\\ -1&0&0&0\\ 0&0&0&0\\0&0&0&0 \end{pmatrix}
+O(\xi^3+\xi^2 a+a^2)
\end{split}
\end{equation}
as $\xi$, $a\to0$, where 
\begin{align*}
L=&-\frac{3\c(\k)}{\c^2(\k)+1}\left(2h_2\frac{\c^2(\k)-1}{\c^2(\k)+2}+1\right)
-\frac12\k\c'(\k)\frac{\c^2(\k)+2}{\c^2(\k)+1}, \\
L_{31}=&\c(\k)\left(6h_2\frac{\c^2(\k)+1}{\c^2(\k)+2}+\frac12(\c^2(\k)+2)+2\k(\c\c')(\k)\right),\\
L_{41}=&\frac14(\c^4(\k)+3\c^2(\k)+4)+3h_2\c^2(\k)\frac{\c^2(\k)-1}{\c^2(\k)+2}+\frac12\k(\c^3\c')(\k),
\end{align*}
and $h_2$ is in \eqref{def:h02}. Moreover, \eqref{def:I3} becomes
\begin{equation}\label{E:I3}
\begin{split}
\mathbf{I}(\xi,a) =&\mathbf{I}+a\frac{2}{\c^2(\k)+1}
\begin{pmatrix}0&0& 1-3h_2\dfrac{\c^2(\k)}{\c^2(\k)+2} & \c(\k)\left(\frac14\c^2(\k)+\frac32-3h_2\right)\\
0&0&0&0\\0&0&0&0\\0&0&0&0 \end{pmatrix} \\
&+a\frac{1}{\c^2(\k)+4}\begin{pmatrix}0&0&0&0\\0&0&0&0\\
1-3h_2\dfrac{\c^2(\k)}{\c^2(\k)+2}&0&0&0\\
\c(\k)\left(\frac14\c^2(\k)+\frac32-3h_2\right)&0&0&0 
\end{pmatrix} \\
&-\frac12i\xi a\frac{\k(\c\c')(\k)}{(\c^2(\k)+1)^2(\c^2(\k)+4)}\begin{pmatrix}0&0&0&0\\
0&0&2(\c^2(\k)+4)&\c^2(\k)+4\\
0&\c^2(\k)+1&0&0\\
0&\c^2(\k)+1&0&0\end{pmatrix}\\
&+O(\xi^3+\xi^2 a+a^2) 
\end{split}
\end{equation}
as $\xi$, $a\to 0$, where $\mathbf{I}$ means the $4\times4$ identity matrix. Note that the coefficient matrices are explicit functions of $\k$.

For $a=0$, \eqref{E:L3} and \eqref{E:I3} become
\begin{align*}
\mathbf{L}(\xi,0)=&i\xi\begin{pmatrix} -\k\c'(\k)&0&0&0\\ 0&-\k\c'(\k)&0&0\\ 
0&0&\c(\k)\dfrac{\c^2(\k)+8}{\c^2(\k)+4}&\dfrac{\c^2(\k)-4}{\c^2(\k)+4}\\
0&0&\dfrac{\c^2(\k)-4}{\c^2(\k)+4}&\c(\k)\dfrac{\c^2(\k)}{\c^2(\k)+4}\end{pmatrix} \\
&+\frac12\xi^2\k(2\c'(\k)+\k\c''(\k))\begin{pmatrix}0&1&0&0\\ -1&0&0&0\\ 0&0&0&0\\0&0&0&0 \end{pmatrix}+O(\xi^3)
\end{align*}
and $\mathbf{I}(\xi,0)=\mathbf{I}$ as $\xi\to0$. It is then easy to verify that the roots of $\det(\mathbf{L}-\lambda\mathbf{I})(\xi,0)$ coincide with the eigenvalues $i\omega(\pm1+\xi,-)$ and $i\omega(\xi,\pm)$ of $\L(\xi,0)$ up to terms of order $\xi^2$ for $\xi>0$ and sufficiently small. For $\xi=0$, \eqref{E:L3} and \eqref{E:I3} become
\[
\mathbf{L}(0,a)=\frac14 a(\c^2(\k)+1)\begin{pmatrix} 0&0&0&0\\0&0&0&0\\0&0&0&0\\0&1&0&0\end{pmatrix}+O(a^2)
\]
and $\mathbf{I}(0,a)=\mathbf{I}+O(a)$ for $a\to0$. This is reminiscent of the Jordan block structure of $\L(0,a)$; see Lemma~\ref{lem:L}.

\subsection{Modulational instability index}\label{ss:index}

We turn the attention to the roots of
\begin{equation}\label{def:P}
\begin{split}
\det(\mathbf{L}-&\lambda\mathbf{I})(\xi)(a;\k,0,0) \\
=&p_4(\xi,a;\k)\lambda^4+ip_3(\xi,a;\k)\lambda^3+p_2(\xi,a;\k)\lambda^2+ip_1(\xi,a;\k)\lambda+p_0(\xi,a;\k) \\
=:&p(\lambda)(\xi,a;\k)
\end{split}
\end{equation}
for $\xi>0$, $a\in\mathbb{R}$ and $\xi$, $|a|$ sufficiently small for $\k>0$, where $\mathbf{L}$ and $\mathbf{I}$ are in \eqref{E:L3} and \eqref{E:I3}. Recall that they coincide with the $L^2(\mathbb{T})\times L^2(\mathbb{T})$ eigenvalues of $\L(\xi)(a;\k,0,0)$ in the vicinity of the origin in $\mathbb{C}$ up to terms of orders $\xi^2$ and $a$ as $\xi$, $a\to0$.  

Note that $p_0$, $p_1$, $\dots$, $p_4$ depend analytically on $\xi$, $a$, and $\k$ for any $\xi>0$ and $|a|$ sufficiently small for any $\k>0$. Recall that the spectrum of $\L(\xi,a)$ is symmetric with respect to the reflection in the imaginary axis for any $\xi\in[0,1/2]$ and $a\in\mathbb{R}$ admissible for any $\k>0$. Hence $p_0$, $p_1$, $\dots$, $p_4$ are real valued. Recall that 
\[
\text{spec}\,\L(\xi,a)^*=\text{spec}\,\L(-\xi,a).
\] 
Hence $p_1$ and $p_3$ are even in $\xi$, whereas $p_0$, $p_2$, $p_4$ are odd. Moreover, the spectrum of $\L(\xi,a)$ remains invariant under $a\mapsto -a$ by \eqref{E:invariance} for any $\xi\in[0,1/2]$ and $a\in\mathbb{R}$ admissible for any $\k>0$. Hence $p_0$, $p_1$, $\dots$, $p_4$ are even in $a$. 

For $\xi=0$, Lemma~\ref{lem:L} implies that $\lambda=0$ is a root of $p(0,a;\k)$ with multiplicity four for any $a\in\mathbb{R}$ and $|a|$ sufficiently small for any $\k>0$. Likewise, $\xi=0$ is a root of $p(\cdot\,,a;\k)(0)$ with multiplicity four. 
Thus we may define
\[
q(-i\xi \lambda)(\xi,a;\k)=
\xi^4 (q_4(\xi,a;\k)\lambda^4-q_3(\xi,a;\k)\lambda^3-q_2(\xi,a;\k)\lambda^2+q_1(\xi,a;\k)\lambda+q_0(\xi,a;\k)),
\]
where 
\begin{equation}\label{def:qk}
p_k(\xi,a;\k):=\xi^{4-k}q_k(\xi,a;\k)\qquad \text{for $k=0$, $1$, $\dots$, $4$}.
\end{equation} 
Note that $q_0$, $q_1$, $\dots$, $q_4$ are real valued and depend analytically on $\xi$, $a$, and $\k$ for any $\xi>0$ and $|a|$ sufficiently small for any $\k>0$. Moreover, they are odd in $\xi$ and even in $a$. 
For $a\in\mathbb{R}$ and $|a|$ sufficiently small for $\k>0$, by virtue of Section~\ref{ss:MI}, a sufficiently small, periodic wave train $\eta(a;\k,0,0)$, $u(a;\k,0,0)$ and $c(a;\k,0,0)$ of \eqref{E:main}-\eqref{def:c} is modulationally unstable, provided that $q$ possesses a pair of complex roots for $\xi>0$ and small. 

Let 
\begin{align*}
\Delta_0=
&256q_4^3q_0^3-192q_4^2q_3q_1q_0^2-128q_4^2q_2^2q_0^2+144q_4^2q_2q_1^2q_0 \\
&-27q_4^2q_1^4+144q_4q_3^2q_2q_0^2-6q_4q_3^2q_1^2q_0-80q_4q_3q_2^2q_1q_0\\
&+18q_4q_3q_2q_1^3+16q_4q_2^4q_0-4q_4q_2^3q_1^2-27q_3^4q_0^2+18q_3^3q_2q_1q_0\\
&-4q_3^3q_1^3-4q_3^2q_2^3q_0+q_3^2q_2^2q_1^2,
\intertext{and}
\Delta_1=&-8q_4q_2-3q_3^2,\\
\Delta_2=&64q_4^3q_0-16q_4^2q_2^2-16q_4q_3^2q_2+16q_4^2q_3q_1-3q_3^4.
\end{align*}
They classify the nature of the roots of the quartic polynomial $q$. Specifically, if $\Delta_0<0$ then the roots of $q$ are distinct, two real and two complex. If $\Delta_0>0$ and $\Delta_1\geq 0$ then the roots are distinct and complex. If $\Delta_0>0$ and if $\Delta_1<0$, $\Delta_2>0$ then the roots of $q$ are distinct and complex. If $\Delta_0>0$ and if $\Delta_1<0$, $\Delta_2<0$, on the other hand, then the roots are distinct and real. If $\Delta_0=0$ then at least two roots are equal; see \cite{HP1}, for instance, for a complete proof. Note that $\Delta_0$ is the discriminant of $q$. 

Note that $\Delta_0$, $\Delta_1$, $\Delta_2$ are even in $\xi$ and $a$. We may write
\begin{align*}
\Delta_0(\xi,a;\k)=:&\Delta_0(\xi,0;\k)+a^2\Delta(\k)+O(a^2(\xi^2+a^2)),
\intertext{and}
\Delta_1(\xi,a;\k)=&\Delta_1(\xi,0;\k)+O(a^2), \\
\Delta_2(\xi,a;\k)=&\Delta_2(\xi,0;\k)+O(a^2)
\end{align*}
as $a\to0$ for any $\xi>0$ and sufficiently small for any $\k>0$. We then use \eqref{E:L3}, \eqref{E:I3}, \eqref{def:P}, \eqref{def:qk}, and we make a Mathematica calculation to show that
\begin{align*}
\Delta_0(\xi,0;\k)=&4\xi^2\k^2(((\k\c(\k))')^2-1)^4((\k\c(\k))'')^2+O(\xi^4)>0,
\intertext{and}
\Delta_1(\xi,0;\k)=&-4(2+(\c(\k)+\k\c'(\k))^2)+O(\xi^2)<0, \\
\Delta_2(\xi,0;\k)=&-16(1+2(\c(\k)+\k\c'(\k))^2)+O(\xi^2)<0 
\end{align*} 
as $\xi\to 0$ for any $\k>0$. Therefore, for $a\in\mathbb{R}$, $|a|$ sufficiently small and fixed, if $\Delta(\k)<0$ for some $\k>0$ then it is possible to find a sufficiently small $\xi_0>0$ such that $\Delta_0(\xi,a;\k)<0$ and $\Delta_1$, $\Delta_2<0$ for $\xi\in(0,\xi_0)$. Hence $q$ possesses two real and two complex roots for $\xi\in(0,\xi_0)$, implying the modulational instability. We pause to remark that one must take $\xi$ small enough so that $a^2\Delta(\k)$ dominates $\Delta_0(\xi,0;\k)=O(\xi^2)$. That means, the modulational instability is a nonlinear phenomenon. If $\Delta\geq 0$, on the other hand, then $\Delta_0>0$ and $\Delta_1$, $\Delta_2<0$ for $\xi>0$ sufficiently small. Hence the roots of $q$ are real for $\xi>0$ sufficiently small. Recall from Section~\ref{ss:a} the spectral stability in the vicinity of the origin in $\mathbb{C}$ for $\xi$ away from zero. Hence this implies the spectral stability in the vicinity of the origin in $\mathbb{C}$.

We use \eqref{E:L3}, \eqref{E:I3}, \eqref{def:P}, \eqref{def:qk}, and we make a Mathematica calculation to find $\Delta$ explicitly, whereby we derive a modulational instability index for \eqref{E:main}-\eqref{def:c}. We summarize the conclusion.

\begin{theorem}[Modulational instability index] \label{thm:index}
A sufficiently small, $2\pi/\k$-periodic wave train of \eqref{E:main}-\eqref{def:c} is modulationally unstable, provided that
\begin{equation} \label{def:ind}
\Delta(\k):=\frac{i_1(\k)i_2(\k)}{i_3(\k)}i_4(\k)<0,
\end{equation}
where
\begin{subequations}
\begin{align}
i_1(\k)=&(\k\cc(\k))'', \label{def:i1} \\
i_2(\k)=&((\k\cc(\k))')^2-1, \label{def:i2} \\
i_3(\k)=&\cc^2(\k)-\cc^2(2\k), \label{def:i3}
\intertext{and}
i_4(\k)=&3\cc^2(\k)+5\cc^4(\k)-2\cc^2(2\k)(\cc^2(\k)+2) \label{def:i4} \\
&+18\k\cc^3(\k)\cc'(\k)+\k^2(\cc')^2(\k)(5\cc^2(\k)+4\cc^2(2\k)).\hspace*{-20pt}\notag
\end{align}
\end{subequations}
It is spectrally stable to square integrable perturbations in the vicinity of the origin in $\mathbb{C}$ otherwise.
\end{theorem}

Theorem~\ref{thm:index} elucidates four resonance mechanisms which contribute to the sign change in $\Delta$ and, ultimately, the change in the modulational stability and instability for \eqref{E:main}-\eqref{def:c}. Note that 
\[
\text{$\pm\c(\k)=$the phase velocity\quad and\quad $\pm(\k\c(\k))'=$the group velocity}
\]
in the linear theory, where $\pm$ mean right and left propagating waves, respectively. Specifically,  
\begin{itemize}
\item[(R1)] $i_1(\k)=0$ at some $\k$; that is, the group velocity achieves an extremum at the wave number $\k$; 
\item[(R2)] $i_2(\k)=0$ at some $\k$; that is, the group velocity at the wave number $\k$ coincides with the phase velocity in the long wave limit as $\k\to0$, resulting in the ``resonance of short and long waves;"
\item[(R3)] $i_3(\k)=0$ at some $\k$; that is, the phase velocities of the fundamental mode and the second harmonic coincide at the wave number $\k$, resulting in the ``second harmonic resonance;"
\item[(R4)] $i_4(\k)=0$ at some $\k$.
\end{itemize}
Resonances (R1), (R2), (R3) are determined by the dispersion relation in the linear theory. For instance, $i_1$, $i_2$, $i_3$ appear in an index formula for \eqref{E:BW2} and \eqref{def:c1} (or \eqref{def:c}), which shares the dispersion relation in common with \eqref{E:main}-\eqref{def:c}; see \cite{HP1} for details. Moreover, $i_2$ appears in \cite{BM1995}, albeit implicitly. Resonance (R4), on the other hand, results from a rather complicated balance of the dispersion and nonlinear effects. For \eqref{E:BW2}, for instance, $i_4$ is replaced by $2i_3+\c^2(2\k)i_2$; see \cite{HP1} for details. To compare, a modulational instability index for the Whitham equation (see \cite{HJ2,HJ3}, for instance)
\[
\frac{(\k\c(\k))''((\k\c(\k))'-1)}{\c(\k)-\c(2\k)}(\k\c(\k))'-1+2(\c(\k)-\c(2\k))
\]
elucidates the same resonance mechanisms which contribute to the change in the modulational stability and instability, but in unidirectional propagation.

\subsection{Critical wave number}

Since $(\k\c(\k))'<1$ for any $\k>0$ and decreases monotonically over the interval $(0,\infty)$ by brutal force, $i_1(\k)<0$ and $i_2(\k)<0$ for any $\k>0$. Since $\c(\k)>0$ for any $\k>0$ and decreases monotonically over the interval $(0,\infty)$ (see Figure~\ref{fig:c}), $i_3(\k)>0$ for any $\k>0$. Hence the sign of $\Delta$ coincides with that of $i_4$. By the way, $i_1$, $i_2$, $i_3$ may change their signs in the presence of the effects of surface tension; see Section~\ref{sec:ST} for details.

We use \eqref{def:i4} and make an explicit calculation to show that 
\[
\lim_{\k\to 0+}\frac{i_4(\k)}{\sqrt{\k}^{5}}=9\quad\text{and}\quad\lim_{\k\to\infty}\k i_4(\k)=-3.
\]
Hence $\Delta(\k)>0$ for $\k>0$ sufficiently small, implying the modulational stability, and it is negative for $\k>0$ sufficiently large, implying the spectral stability in the vicinity of the origin in $\mathbb{C}$. Moreover, the intermediate value theorem asserts a root of $i_4$, which changes the modulational stability and instability. 

\begin{figure}[h] 
~~\includegraphics[scale=0.7]{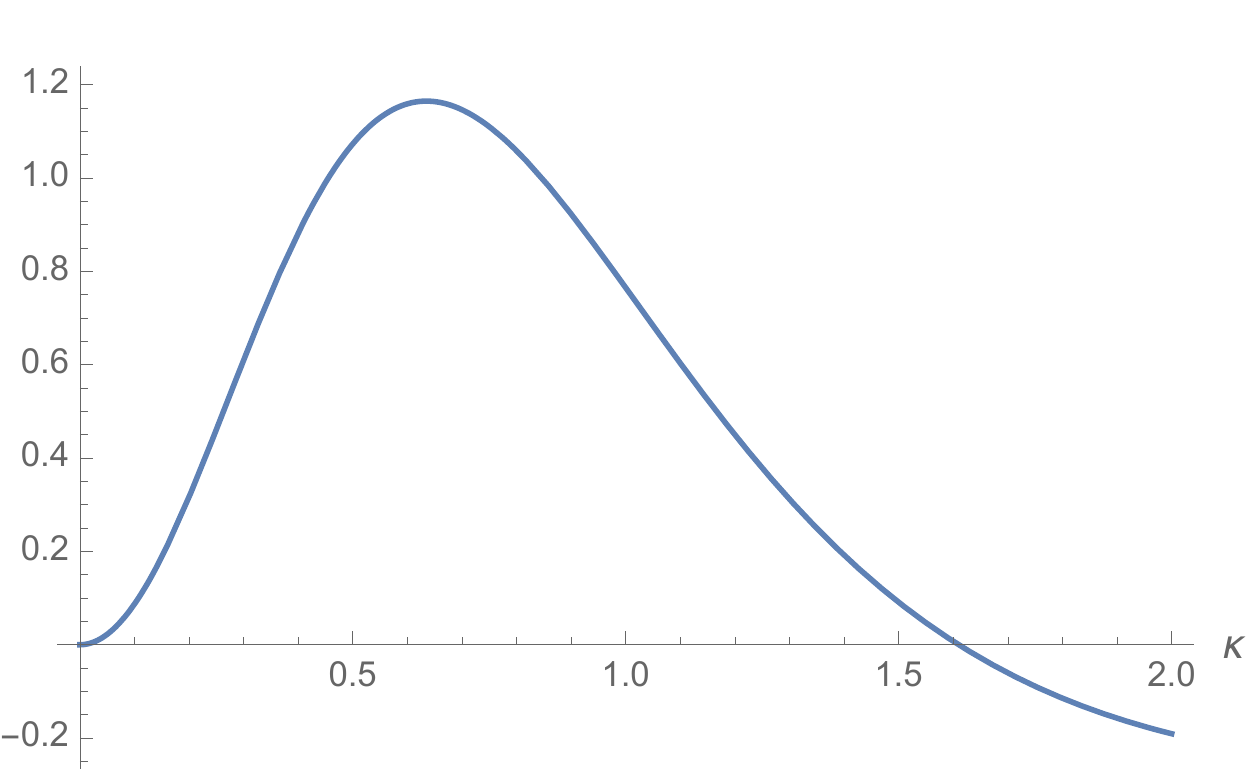}\quad
\caption{The graph of $i_4(\k)$ for $\k\in(0,2)$.}\label{fig:i4}
\end{figure}

\begin{figure}[h] 
~~\includegraphics[scale=0.7]{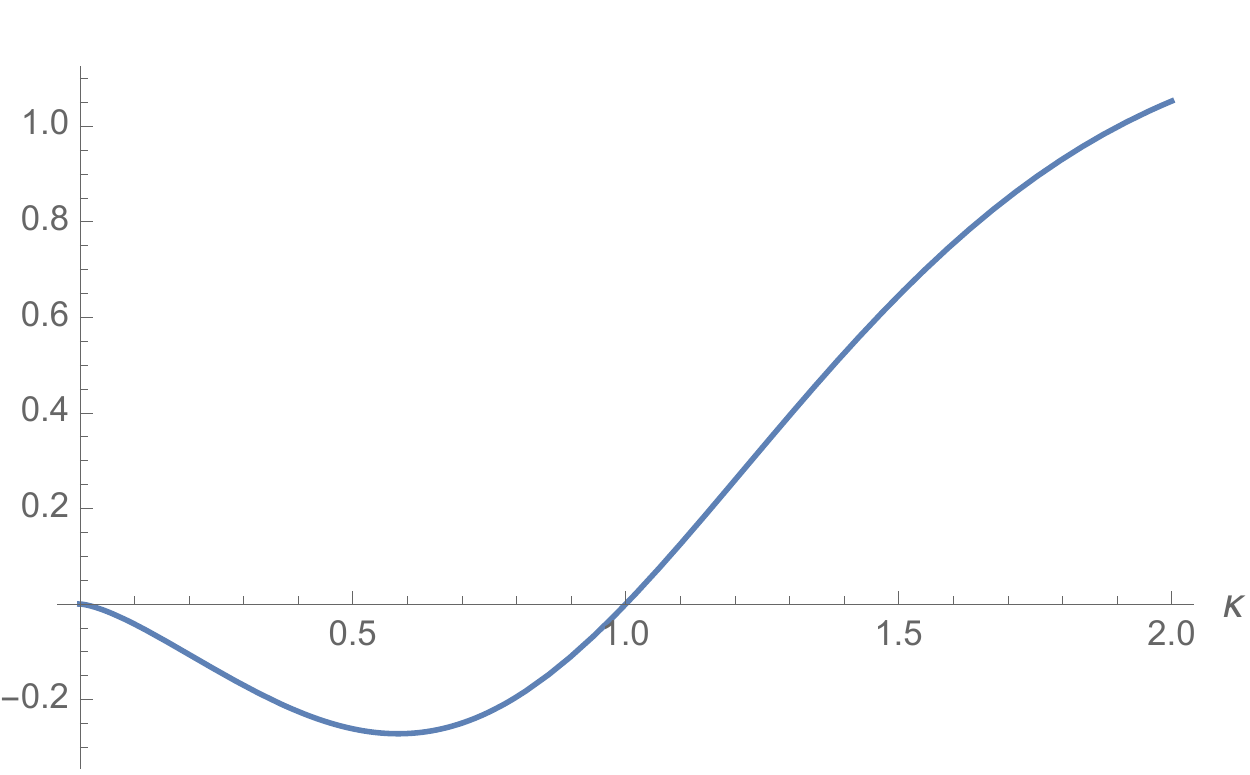}\quad
\caption{The graph of $i_4(1.61\k^{-1})$ for $\k\in(0,2)$.}\label{fig:i4k}
\end{figure}

It is difficult to analytically study the sign of $i_4$ further. On the other hand, a numerical evaluation of \eqref{def:i4} reveals a unique root $\k_c$, say, of $i_4$ over the interval $(0,\infty)$ (see Figure~\ref{fig:i4}) such that $i_4(\k)>0$ if $0<\k<\k_c$ and it is negative if $\k_c<\k<\infty$. Upon close inspection (see Figure~\ref{fig:i4k}), moreover, $\k_c=1.610\dots$. We summarize the conclusion.

\begin{corollary}[Critical wave number]\label{cor:kc} 
A sufficiently small, $2\pi/\k$-periodic wave train of \eqref{E:main}-\eqref{def:c} is modulationally unstable if $\k>\k_c$, where $\k_c=1.610\dots$ is a unique root of $i_4$ in \eqref{def:i4} over the interval $(0,\infty)$. It is spectrally stable to square integrable perturbations in the vicinity of the origin in $\mathbb{C}$ if $0<\k<\k_c$.
\end{corollary}

Corollary~\ref{cor:kc} qualitatively states the Benjamin-Feir instability of a Stokes wave. Fortuitously, the critical wave number compares reasonably well with that in \cite{BH, Whitham1967} and \cite{BM1995}. The critical wave number for the Whitham equation (see \eqref{E:Whitham}) is $1.146\dots$; see \cite{HJ2}, for instance. 

We point out that the critical wave number in \cite{BH, Whitham1967} and \cite{BM1995} was determined by an approximation of the numerical value of some explicit function of $\k$, which seems difficult to calculate analytically. Therefore, it is not surprising that the proof of Corollary~\ref{cor:kc} ultimately relies on a numerical evaluation of the modulational instability index \eqref{def:ind}.

\section{Stability and instability away from the origin}\label{sec:HF}

Let $\eta=\eta(a;\k,0,0)$, $u=u(a;\k,0,0)$, and $c=c(a;\k,0,0)$, for some $a\in\mathbb{R}$ and $|a|$ sufficiently small for some $\k>0$, denote a $2\pi/\k$-periodic wave train of \eqref{E:main}-\eqref{def:c} near the rest state, whose existence follows from Theorem~\ref{thm:existence}. In the previous section, we studied the spectrum of the associated linearized operator in the vicinity of the origin in $\mathbb{C}$, whereby we determined its modulational stability and instability. We turn the attention to the spectral stability and instability away from the origin. Throughout the section, we employ the notation in the previous section.

\subsection{Collision condition}\label{ss:collision}

Recall from Section~\ref{ss:a=0} that
\[
\mathcal{L}(\xi,0)\e(n+\xi,\pm)=i\omega(n+\xi,\pm)\e(n+\xi,\pm)\qquad \text{for $n\in \mathbb{Z}$ and $\xi\in[0,1/2]$},
\]
where 
\[
\omega(n+\xi,\pm)=(n+\xi)(\c(\k)\pm\c(\k(n+\xi)))\quad\text{and}\quad
\e(n+\xi,\pm)(z)=\begin{pmatrix}1\\ \mp \c(\k(n+\xi) \end{pmatrix}e^{inz}
\]
for $\k>0$. Recall that $i\omega(0,\pm)=i\omega(\pm1,-)=0$ and $i\omega(n+\xi,\pm)\neq 0$ otherwise. Moreover, no collisions take place among $i\omega(n+\xi,+)$'s or among $i\omega(n+\xi,-)$'s except at the origin. But
\[
i\omega(n_1+\xi,-)=i\omega(n_2+\xi,+)\neq 0\qquad \text{for some $n_1,n_2\in\mathbb{Z}$ and $\xi\in(0,1/2]$}
\]
for some $\k>0$ if and only if
\begin{equation}\label{E:collision}
(n_1+\xi)\c(\k(n_1+\xi))+(n_2+\xi)c(\k(n_2+\xi))=(n_1-n_2)\c(\k).
\end{equation}

We claim that $i\omega(n_1+\xi,-)$ and $i\omega(n_2+\xi,+)$ do not collide for any $n_1$, $n_2\in\mathbb{Z}$ such that $|n_1-n_2|=1$ and $\xi\in(0,1/2]$ for any $\k>0$. Suppose on the contrary that $i\omega(n_1+\xi,-)=i\omega(n_2+\xi,+)$ for some $n_1$, $n_2\in\mathbb{Z}$ such that $|n_1-n_2|=1$ and $\xi\in(0,1/2]$ for some $\k>0$. Assume for now $n_1=n_2+1$ and $n_2\geq 0$. Since $n_2+1+\xi\geq1$ and since $z\c(z)>0$ and increases monotonically over the interval $(0,\infty)$, it follows that
\[
(n_2+1+\xi)\c(\k(n_2+1+\xi))+(n_2+\xi)\c(\k(n_2+\xi))>\c(\k)
\]
for any integer $n_2\geq 0$ and $\xi\in(0,1/2]$ for any $\k>0$. This contradicts \eqref{E:collision}. One may repeat the argument for $n_1$, $n_2<0$. This proves the claim.

\begin{figure}[h]
~~\includegraphics[scale=0.5]{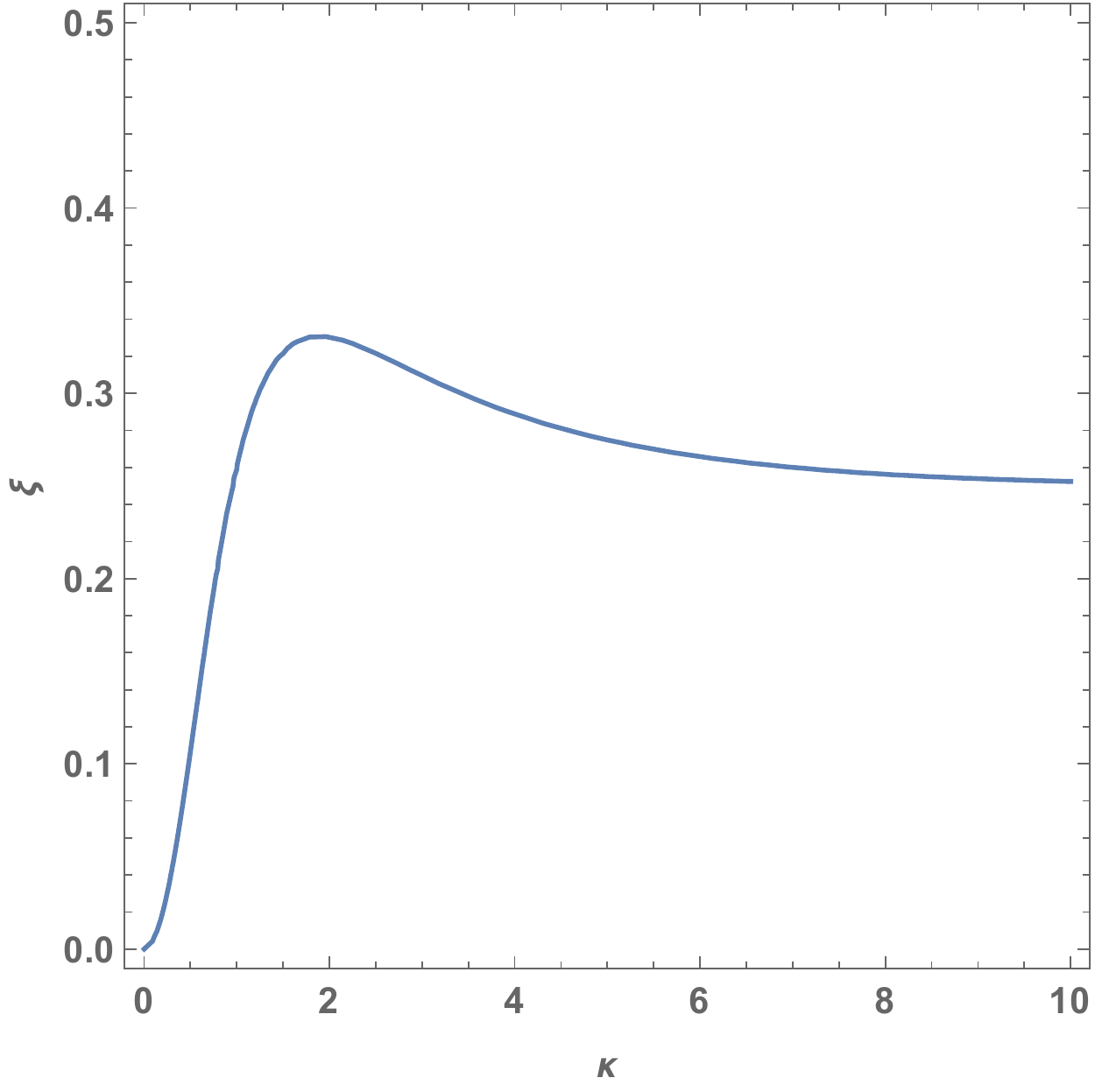}\quad
\caption{$\k$ vs. $\xi$ when $i\omega(2+\xi,-)$ and $i\omega(0+\xi,+)$ collide.}
 \label{fig:xi-k}
\end{figure}

To proceed, a numerical evaluation of \eqref{E:collision} reveals that $i\omega(n_1+\xi,-)=i\omega(n_2+\xi,+)$ for some $n_1$, $n_2\in\mathbb{Z}$ such that $|n_1-n_2|=2$ and $\xi\in(0,1/2]$ and for some $\k>0$ if and only if $n_1=2$ and $n_2=0$. Such a collision takes place for any $\k>0$ at some $\xi\in(0,1/2]$ depending on $\k$; see Figure~\ref{fig:xi-k}.

\begin{figure}[h] 
~~\includegraphics[scale=0.5]{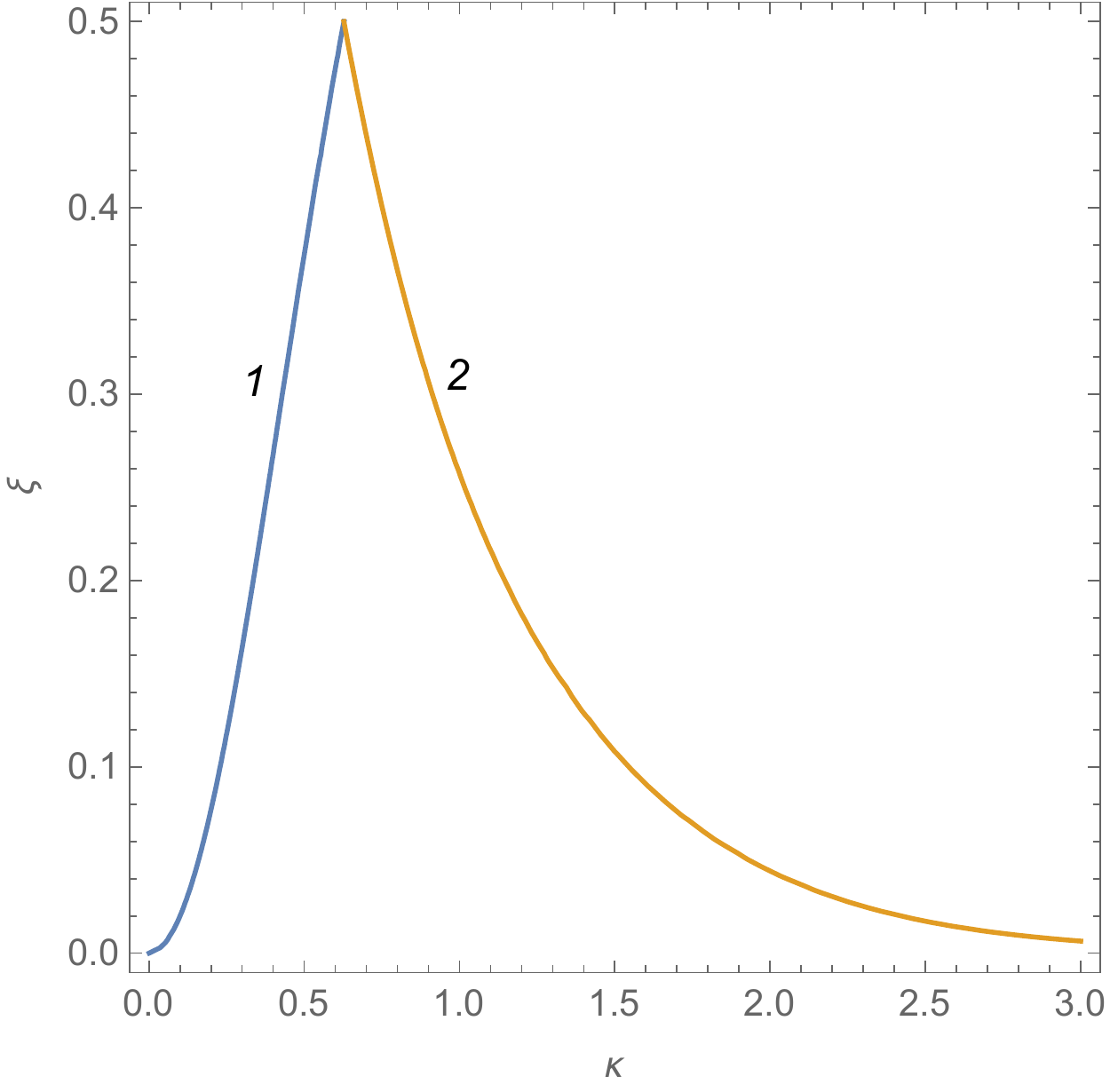}\quad
\caption{$\k$ vs. $\xi$ when $i\omega(n_1+\xi,-)$ and $i\omega(n_2+\xi,+)$ collide for $|n_1-n_2|=3$. Along Curve~1, $(n_1,-)=(0,-)$ and $(n_2,+)=(3,+)$ collide. Curve~2 represents the collision for $(-1,-)$ and $(-4,+)$.}
\label{fig:case3}
\end{figure}

Moreover, $i\omega(n_1+\xi,-)=i\omega(n_2+\xi,+)$ for some $n_1$, $n_2\in\mathbb{Z}$ such that $|n_1-n_2|=3$ and $\xi\in(0,1/2]$ for some $\k>0$ if and only if $n_1=0$ and $n_2=3$, or else $n_1=-1$ and $n_2=-4$. Together, such a collision takes place for any $\k>0$ at some $\xi\in(0,1/2]$ depending on $\k$; see Figure~\ref{fig:case3}.

\begin{figure}[h]
~~\includegraphics[scale=0.5]{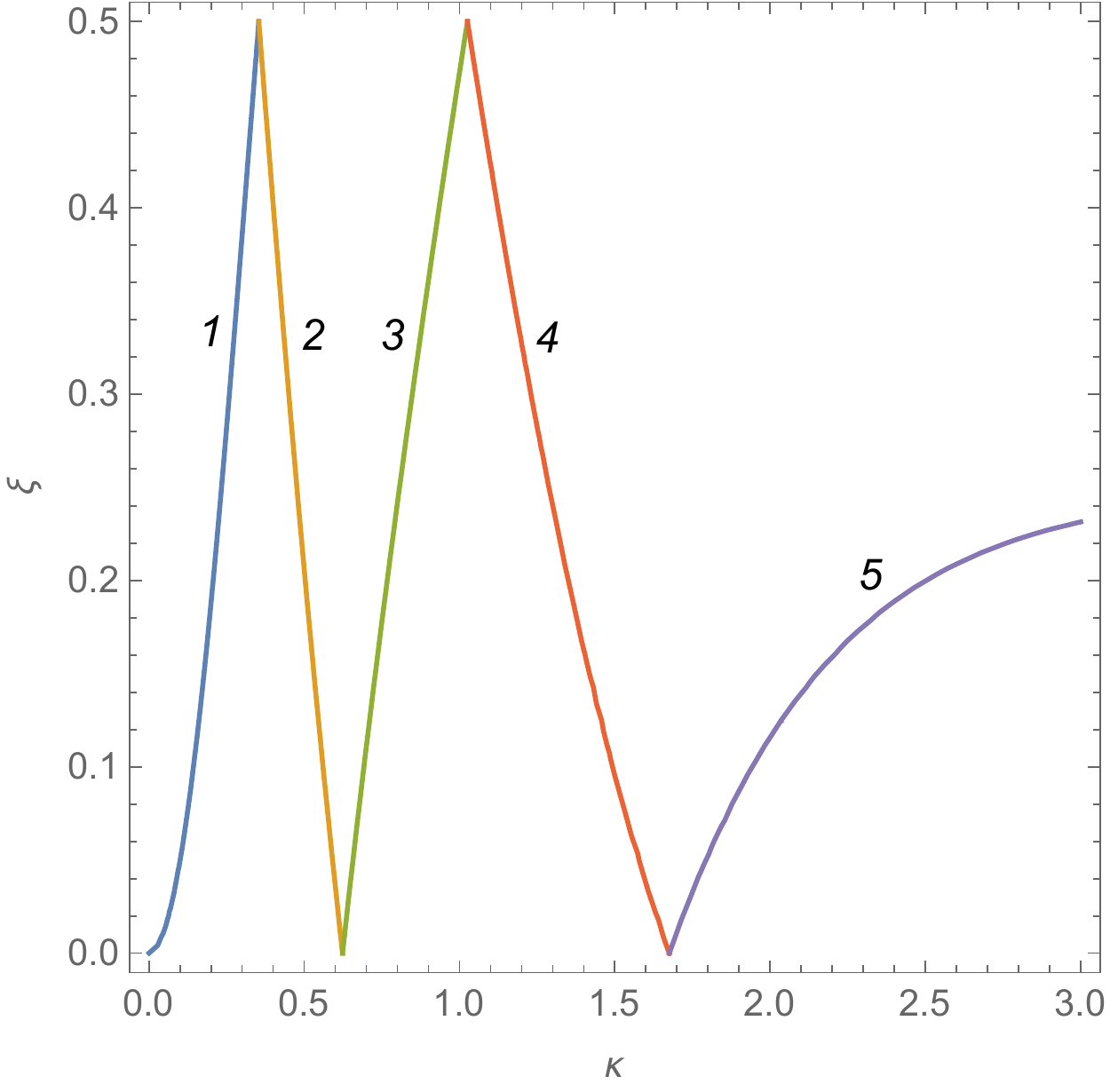}\quad
\caption{$\k$ vs. $\xi$ when $i\omega(n_1+\xi,-)$ and $i\omega(n_2+\xi,+)$ collide for $|n_1-n_2|=4$. Curves 1 through 5 represent the collisions for $(n_1,-)=(0,-)$ and $(n_2,+)=(4,+)$, $(-1,-)$ and $(-5,+)$, $(1,-)$ and $(5,+)$, $(-2,-)$ and $(-6,+)$, $(2,-)$ and $(6,+)$, respectively.}
 \label{fig:case4}
\end{figure}

Continuing, a numerical evaluation of \eqref{E:collision} reveals five collisions for $|n_1-n_2|=4$ when $(n_1,-)=(0,-)$ and $(n_2,+)=(4,+)$, $(-1,-)$ and $(-5,+)$, $(1,-)$ and $(5,+)$, $(-2,-)$ and $(-6,+)$, $(2,-)$ and $(6,+)$. Together, such a collision takes place for any $\k>0$ at some $\xi\in(0,1/2]$ depending on $\k$; see Figure~\ref{fig:case4}. 

\begin{table}
\centering
\begin{tabular}{c|c|c}
 $n_1$ & $n_2$ & $\xi$  \\ \hline
  $2$ & $0$  & 0.261\dots\\ \hline
   $1$ & $5$  & 0.473\dots \\ \hline
    $4$ & $10$  & 0.184\dots\\ \hline
     $6$ & $13$  & 0.158\dots\\ \hline
      $11$ & $20$  & 0.250\dots\\ \hline
       $14$ & $24$  & 0.368\dots\\ \hline
        $26$ & $39$  & 0.006\dots\\ 
\end{tabular}
\caption{Some $n_1$, $n_2$, and $\xi$ when $i\omega(n_1+\xi,-)$ and $i\omega(n_2+\xi,+)$ collide for $\k=1$ .}
\label{table1}
\end{table}

Indeed, for any $\k>0$ for any integer $n\geq2$, it is possible to find $n_1$, $n_2\in\mathbb{Z}$ such that $i\omega(n_1+\xi,-)=i\omega(n_2+\xi,+)$ for some $\xi\in(0,1/2]$ and $|n_1-n_2|=n$. The number of collisions increases as $|n_1-n_2|$ increases. Therefore, $i\omega(n_1+\xi,-)$ and $i\omega(n+\xi,+)$ collide for infinitely many $n_1$, $n_2\in\mathbb{Z}$. In Table~\ref{table1}, we record some $n_1$, $n_2\in\mathbb{Z}$ ($n_1$, $n_2\geq 0$) and $\xi\in(0,1/2]$ for which $i\omega(n_1+\xi,-)$ and $i\omega(n_2+\xi,+)$ collide for $\k=1$. 

The spectrum of the linear operator associated with the water wave problem in the finite depth (see \cite{DO, AN2014}, for instance, for details) contains infinitely many collisions of purely imaginary eigenvalues, which align with the collisions of $i\omega(n_1+\xi,-)$ and $i\omega(n_2+\xi,+)$ for \eqref{E:main}-\eqref{def:c}. To compare, no eigenvalues of the linear operator for the Whitham equation (see \eqref{E:Whitham}) collide other than at the origin; see \cite{HJ2}, for instance, for details. 

\subsection{Signature calculation}\label{ss:signature}

The linear equation associated with \eqref{E:main}-\eqref{def:c} in the moving coordinate frame may be written in Hamiltonian form as
\[
\partial_t \u=\begin{pmatrix}0 &- \partial_z\\ -\partial_z & 0 \end{pmatrix}\delta H(\u),
\]
where 
\[
H(\u)=\int \Big(\frac12 \eta\c^2(\k|\partial_z|)\eta+\frac12u^2-c\eta u\Big)~dz 
\]
and $\delta$ means variational differentiation. By the way, to the best of the authors' knowledge, the Hamiltonian structure for \eqref{E:main}-\eqref{def:c} itself is not understood. In contrast, \eqref{E:BW3} and \eqref{def:c1} (or \eqref{def:c}) are a Hamiltonian system. But the well-posedness for \eqref{E:BW3} is not understood, whence it is not suitable for the purpose of describing wave propagation. Note that the modulational instability proof in the previous section makes no use of Hamiltonian structure. 

If $i\omega(n+\xi,\pm)$ for some $n\in\mathbb{Z}$ and $\xi\in[0,1/2]$ is a nonzero and purely imaginary eigenvalue of  
\begin{align*}
\L(\xi,0)=&\begin{pmatrix}0 & -\partial_z+i\xi\\ -\partial_z+i\xi & 0 \end{pmatrix}
e^{-i\xi z}\begin{pmatrix} -\c(\k) & 1 \\ \c^2(\k|\partial_z|) & -\c(\k)\end{pmatrix}e^{i\xi z} \\
=&\begin{pmatrix}0 & -\partial_z+i\xi\\ -\partial_z+i\xi & 0 \end{pmatrix}e^{-i\xi z}\delta^2H(\mathbf{0})e^{i\xi z}
\end{align*}
(see \eqref{def:Lx} and \eqref{E:hu0}) then $e^{-i\xi z}\delta^2H(\mathbf{0})e^{i\xi z}$ defines a non-degenerate quadratic form on the associated eigenspace. If the eigenvalue is simple then the eigenspace is spanned by $\e(n+\xi,\pm)$, and
\begin{align*}
\l e^{-i\xi z}\delta^2H(\mathbf{0})e^{i\xi z} \e(n+\xi,\pm),\e(n+\xi,\pm)\r=&\mp 2\c(\k(n+\xi))(\c(\k)\pm\c(\k(n+\xi))) \\
=&\mp2\k\omega(n+\xi,\pm),
\end{align*} 
where $\l\cdot\,,\cdot\r$ means the $L^2(\mathbb{T})\times L^2(\mathbb{T})$ inner product. This is either positive or negative, called the {\em Krein signature}. It is straightforward to verify that the Krein signature is positive for
\begin{equation}\label{E:+sig}
i\omega(n+\xi,-) \quad\text{for $n=-1$ or $n\geq 1$}\quad\text{and}\quad
i\omega(n+\xi,+)\quad \text{for $n\geq 0$},
\end{equation}
and negative for
\begin{equation}\label{E:-sig}
i\omega(n+\xi,-)\quad\text{for $n=0$ or $n\leq -2$}\quad\text{and}\quad
i\omega(n+\xi,+)\quad \text{for $n\geq -1$}
\end{equation}
for any $\xi\in[0,1/2]$. The eigenvalue remains simple and the Krein signature does not change, as parameters vary, so long as it does not collide with another eigenvalue; see \cite{MS1986}, for instance, for details.

It is well known (see \cite{MS1986}, for instance, and reference therein) that 
a necessary condition for spectral instability is that a pair of eigenvalues on the imaginary axis with opposite signature collide, unless they are at the origin. Note from \eqref{E:collision}, and \eqref{E:+sig}, \eqref{E:-sig} that the signatures of all colliding eigenvalues of $\L(\xi,0)$ are opposite, unless they are at the origin. 
For the zero eigenvalue, the Krein signature calculation becomes inconclusive. But in the previous section, we made a spectral perturbation calculation and determined the modulational stability and instability.

\subsection{Spetra of $\L(\xi_0+\xi,0)$ and $\L(\xi_0,a)$}\label{ss:L4}

Let 
\[
i\omega(n_1+\xi_0,-)=i\omega(n_2+\xi_0,+)=:i\omega_0
\] 
for some $n_1, n_2\in\mathbb{Z}$ and $\xi_0\in(0,1/2]$ for some $\k>0$, denote a nonzero and purely imaginary, colliding $L^2(\mathbb{T}) \times L^2(\mathbb{T})$ eigenvalue of $\L(\xi_0,0)$. Recall that $\e(n_1+\x,-)$ and $\e(n_2+\x,+)$ are the associated eigenfunctions, complex valued and orthogonal to each other. For $\xi$, $a\in\mathbb{R}$ and $|\xi|$, $|a|$ sufficiently small, we calculate the spectra of $\L(\x+\xi,0)$ and $\L(\x,a)$ in the vicinity of $i\omega_0$ in $\mathbb{C}$. 

For real valued functions, one must take $\pm\x$ in pair (and, hence, $\pm i\omega_0$) and deal with four functions. But the spectral perturbation calculation in the following subsection involves complex valued operators anyway. Hence this is not worth the effort. 

\subsection*{Notation} In the remaining of the section, $\k>0$ and $\xi_0\in(0,1/2]$ are suppressed for simplicity of notation, unless specified otherwise. We use
\begin{equation}\label{def:cn}
c_0=\c(\k)\quad\text{and}\quad c_{n,\xi}=\c(\k(n+\x+\xi))
\end{equation} 
and 
\begin{equation}\label{def:Mn}
\mathbf{C}_0=\begin{pmatrix} c_0 & 1\\ 0 & c_0 \end{pmatrix}\quad \text{and} \quad
\mathbf{C}_{n,\xi}=\begin{pmatrix}c_0 & -1\\ -c^2_{n,\xi} & c_0 \end{pmatrix}
\end{equation}
for simplicity of notation.

\

For $\xi\in\mathbb{R}$ and $|\xi|$ sufficiently small, note from \eqref{E:Lx0=0} that $i\omega(n_1+\xi_0+\xi,-)$ and $i\omega(n_2+\xi_0+\xi,+)$ are $L^2(\mathbb{T})\times L^2(\mathbb{T})$ eigenvalues of $\L(\x+\xi,0)$ in the vicinity of $i\omega_0$ in $\mathbb{C}$, and
\begin{equation}\label{def:qx0}
\begin{aligned}
&\q_1(z):=\e(n_1+\x+\xi,-)=\begin{pmatrix} 1\\ c_{n_1,\xi} \end{pmatrix} e^{in_1 z}, \\
&\q_2(z):=\e(n_2+\x+\xi,+)=\begin{pmatrix} 1\\ -c_{n_2,\xi} \end{pmatrix} e^{in_2z}
\end{aligned}
\end{equation}
are the associated eigenfunctions, complex valued and orthogonal to each other. 

As $|a|$ increases, two eigenvalues at $i\omega_0$ of $\L(\xi_0,0)$ may move around and the associated eigenfunctions vary, analytically, from $\e(n_1+\x,-)$ and $\e(n_2+\x,+)$. For $a\in\mathbb{R}$ and $|a|$ sufficiently small, we calculate the small amplitude expansion of the eigenvalues and eigenfunctions of $\L(\x,a)$ up to terms of order $a$. To compare, Lemma~\ref{lem:L} implies that as $a$ varies, zero persists to be an eigenvalue of $\L(0,a)$ and one may exploit the variations of \eqref{E:periodic} to find the associated eigenfunctions to any order in $a$. Let (by abuse of notation)
\begin{equation}\label{E:HFa}
\L(\x,a)\q_k(a)=\lambda(a)\q_k(a) \qquad \text{for $k=1,2$},
\end{equation}
where
\begin{equation}\label{def:HFeigen}
\lambda(a)=i\omega_0+a \lambda_1+O(a^2)\quad\text{and}\quad 
\q_k(a)(z)=\begin{pmatrix} 1\\ \pm c_{n_k,0}\end{pmatrix}e^{in_k z}+a\q_{k,1}(z)+O(a^2)
\end{equation}
as $a\to 0$, $\lambda_1\in\mathbb{C}$ and $\q_{1,1}$, $\q_{2,1}$ be $2\pi$ periodic. 
Note from \eqref{def:Lx} and \eqref{E:hu-small} that
\begin{align}\label{E:HFL}
\L(\x,a)= &e^{-i\x z}\partial_z \begin{pmatrix} c_0 & -1\\ -\c^2(\k|\partial_z|) & c_0\end{pmatrix} e^{i\x z}
-a e^{-i\x z}\partial_z(\mathbf{C}_0\cos z)e^{i\x z}+O(a^2)\notag \\
=:&\L_0+a\L_1+O(a^2)
\end{align}
as $a \to0$. 

Substituting \eqref{def:HFeigen} and \eqref{E:HFL} into \eqref{E:HFa}, we make an explicit calculation to arrive, at the order of $1$, at 
\begin{equation}\label{E:q0}
\L_0\begin{pmatrix} 1\\ \pm c_{n_k,0}\end{pmatrix}e^{in_kz}
=i\omega_0 \begin{pmatrix} 1\\ \pm c_{n_k,0}\end{pmatrix}e^{in_k z} \qquad \text{for $k=1$, $2$},
\end{equation}
which holds true by hypothesis. 

To proceed, at the order or $a$, we gather
\[
\L_0\q_{k,1}+\L_1\begin{pmatrix} 1\\ \pm c_{n_k,0}\end{pmatrix}e^{in_kz}
=i\omega_0 \q_{k,1}+\lambda_1\begin{pmatrix} 1\\ \pm c_{n_k,0}\end{pmatrix}e^{in_kz} \qquad \text{for $k=1$, $2$}.
\]
If $\q_{k,1}(z)=\sum_{m\in \mathbb{Z}}\mathbf{q}_{n_k,m}e^{i(n_k+m)z}$ in the Fourier series then
\begin{equation}\label{E:q1}
\begin{split}
\sum_{m\in\mathbb{Z}} i(n_k+m+\x)\mathbf{C}_{n_k+m,0}\mathbf{q}_{n_k,m}&e^{i(n_k+m)z} \\
-\frac12 \begin{pmatrix}\c(\k)\pm c_{n_k,0}\\ \pm \c(\k)c_{n_k,0} \end{pmatrix}
(i(n_k+&1+\x)e^{i(n_k+1)z}+i(n_k-1+\x)e^{i(n_k-1)z})\\
=&i\omega_0 \sum_{m\in\mathbb{Z}}\mathbf{q}_{n_k,m}e^{i(n_k+m)z}
+\lambda_1\begin{pmatrix}1\\ \pm c_{n_k,0} \end{pmatrix}e^{in_kz}
\end{split}
\end{equation}
for $k=1$, $2$. Upon inspection, it follows that $\mathbf{q}_{n_k,m}=\mathbf{0}$ unless $m=0$, $\pm1$. We then take the $L^2(\mathbb{T})\times L^2(\mathbb{T})$ inner products of \eqref{E:q1} and $\begin{pmatrix} 1 \\ 0 \end{pmatrix}e^{in_kz}$,  $\begin{pmatrix} 0 \\ 1 \end{pmatrix}e^{in_kz}$, $k=1$, $2$, to arrive at
\[
i(n_k+\x)\mathbf{C}_{n_k,0}\mathbf{q}_{n_k,0}=
i\omega_0\mathbf{q}_{n_k,0}+\lambda_1\begin{pmatrix}1 \\ \pm c_{n_k,0} \end{pmatrix} \qquad \text{for $k=1$, $2$}.
\]
Note from \eqref{E:q0} that $\lambda_1=0$. This agrees with the result in \cite{AN2014}, for instance, for the water wave problem. Note that $\mathbf{q}_{n_k,0}=\begin{pmatrix}1 \\ \pm c_{n_k,0} \end{pmatrix}$ up to the multiplication by a constant.

Continuing, we take the $L^2(\mathbb{T})\times L^2(\mathbb{T})$ inner products of \eqref{E:q1} and $\begin{pmatrix} 1 \\ 0 \end{pmatrix}e^{i(n_1\pm 1)z}$, $\begin{pmatrix} 0 \\ 1 \end{pmatrix}e^{i(n_1\pm 1)z}$ to arrive at
\[
(n_1\pm 1+\x)\Big(2\mathbf{C}_{n_1\pm 1,0}\mathbf{q}_{n_1\pm 1}
-\mathbf{C}_0\begin{pmatrix}1 \\ c_{n_1,0} \end{pmatrix}\Big)=2\omega_0 \mathbf{q}_{n_1,\pm 1}.
\]
A straightforward calculation then reveals that 
\begin{align}\label{def:Q1}
\mathbf{q}_{n_1,\pm 1}=
&\frac12\frac{n_1\pm 1+\x}{(\omega_0-c_0(n_1\pm 1+\x))^2-c^2_{n_1\pm 1,0}(n_1\pm 1+\x)^2} \\
&\times\begin{pmatrix} c_0(n_1\pm 1+\x)(c_0+2c_{n_1,0})-\omega_0 (c_0+c_{n_1,0}) \\ 
(n_1\pm 1+\x)(c^2_0c_{n_1,0}+c^2_{n_1\pm 1,0}(c_0+c_{n_1,0}))-\omega_0 c_0c_{n_1,0} \end{pmatrix}. \notag
\end{align}
We take the $L^2(\mathbb{T})\times L^2(\mathbb{T})$ inner products of \eqref{E:q1} and $\begin{pmatrix} 1 \\ 0 \end{pmatrix}e^{i(n_2\pm 1)z}$, $\begin{pmatrix} 0 \\ 1 \end{pmatrix}e^{i(n_2\pm 1)z}$, likewise, and we make an explicit calculation to find
\begin{align}\label{def:Q2}
\mathbf{q}_{n_2,\pm 1} =&
\frac12\frac{n_2\pm 1+\x}{(\omega_0-c_0(n_2\pm 1+\x))^2-c^2_{n_2\pm 1,0}(n_2\pm 1+\x)^2} \\
&\times\begin{pmatrix} c_0(n_2\pm 1+\x)(c_0-2c_{n_2,0})-\omega_0 (c_0-c_{n_2,0}) \\ 
(n_2\pm 1+\x)(-c^2_0c_{n_2,0}+c^2_{n_2\pm 1,0}(c_0-c_{n_2,0}))+\omega_0 c_0c_{n_2,0} \end{pmatrix}. \notag
\end{align}
We are able to calculate higher order terms in like manner. But the formulae become lengthy and complicated. We will investigate the details in a future publication.

To recapitulate, for $\xi\in\mathbb{R}$ and $|\xi|$ sufficiently small for $a=0$, the $L^2(\mathbb{T})\times L^2(\mathbb{T})$ spectrum of $\L(\x+\xi,0)$ contains two purely imaginary eigenvalues $i\omega(n_1+\x+\xi,-)$ and $i\omega(n_2+\x+\xi,+)$ in the vicinity of $i\omega_0$ in $\mathbb{C}$, and \eqref{def:qx0} makes the associated eigenfunctions, which depend analytically on $\xi$. For $\xi=0$ for $a\in\mathbb{R}$ and $|a|$ sufficiently small, the spectrum of $\L(\x,a)$ contains two eigenvalues at $i\omega_0$ up to the order of $a$, and 
\begin{equation}\label{def:q0a}
\q_{k}(a)=\begin{pmatrix} 1\\ \pm c_{n_k,0}\end{pmatrix}e^{in_k z}
+a\mathbf{q}_{n_k,1}e^{i(n_k+1)z}+a\mathbf{q}_{n_k,-1}e^{i(n_k-1)z}+O(a^2)\qquad\text{for $k=1$, $2$}
\end{equation}
makes the associated eigenfunctions, which depend analytically on $a$, where $\mathbf{q}_{n_k,\pm 1}$ for $k=1$, $2$ are in \eqref{def:Q1} and \eqref{def:Q2}.

\subsection{Spectra of $\L(\x+\xi,a)$}\label{ss:perturbation4}
For $\xi$, $a\in\mathbb{R}$ and $|\xi|$, $|a|$ sufficiently small, it follows from perturbation theory (see \cite[Section~4.3.5]{K}, for instance, for details) that the $L^2(\mathbb{T})\times L^2(\mathbb{T})$ spectrum of $\L(\x+\xi,a)$ contains two eigenvalues in the vicinity of $i\omega_0$ in $\mathbb{C}$, and the associated eigenfunctions vary analytically from \eqref{def:qx0} and \eqref{def:q0a}. Let (by abuse of notation)
\begin{equation}\label{def:q}
\begin{aligned}
\q_1(\xi,a)(z)=&\begin{pmatrix} 1\\ c_{n_1,\xi}\end{pmatrix}e^{in_1 z}
+a\mathbf{q}_{n_1,1}e^{i(n_1+1)z}+a\mathbf{q}_{n_1,-1}e^{i(n_1-1)z}+O(\xi^2a+a^2),  \\
\q_2(\xi,a)(z)=&\begin{pmatrix} 1\\ -c_{n_2,\xi}\end{pmatrix}e^{in_2 z}
+a\mathbf{q}_{n_2,1}e^{i(n_2+1)z}+a\mathbf{q}_{n_2,-1}e^{i(n_2-1)z}+O(\xi^2a+a^2)
\end{aligned}
\end{equation}
as $\xi$, $a\to0$, where $\mathbf{q}_{n_1,\pm1}$ and $\mathbf{q}_{n_2,\pm1}$ are in \eqref{def:Q1} and \eqref{def:Q2}. For $a=0$, note that $\q_1$ and $\q_2$ become \eqref{def:qx0}. For $\xi=0$ they become \eqref{def:q0a}. Hence $\q_1$ and $\q_2$ are the eigenfunctions associated with the eigenvalues of $\L(\x+\xi,a)$ near $i\omega_0$ up to terms of order $a$ as $\xi$, $a\to 0$.  It seems impossible to uniquely determine terms of order $\xi a$ in the eigenfunction expansion without an ad hoc orthogonality condition. Fortuitously, it turns out that they do not contribute to the spectral instability up to the order of $a$ as $\xi$, $a\to 0$. Hence we may neglect them in \eqref{def:q}.

We proceed as in Section~\ref{ss:perturbation} and calculate (by abuse of notation)
\begin{align}
\mathbf{L}(\xi,a)=&\begin{pmatrix} 
\dfrac{\l \L(\x+\xi,a)\q_k(\xi,a),\q_\ell(\xi,a)\r}{\l \q_k(\xi,a), \q_k(\xi,a)\r} \end{pmatrix}_{k,\ell=1,2}\label{def:L4} 
\intertext{and}
\mathbf{I}(\xi,a)=&\begin{pmatrix} 
\dfrac{\l \q_k(\xi,a),\q_\ell(\xi,a)\r}{\l \q_k(\xi,a), \q_k(\xi,a)\r} \end{pmatrix}_{k, \ell=1,2}\label{def:I4}
\end{align}
up to the order of $a$ as $\xi$, $a\to 0$. Throughout the subsection, $\l\cdot\,,\cdot\r$ means the $L^2(\mathbb{T})\times L^2(\mathbb{T})$ inner product. For $\xi$, $a\in\mathbb{R}$ and $|\xi|$, $|a|$ sufficiently small, it follows from perturbation theory that the roots of $\det(\mathbf{L}-\lambda\mathbf{I})$ coincide with the eigenvalues of $\L(\x+\xi,a)$ up to terms of order $a$. 

We begin by calculating 
\begin{align*}
\l\q_1,\q_1\r=&1+c^2_{n_1,\xi}+O(\xi^2a+a^2), \\
\l\q_2,\q_2\r=&1+c^2_{n_2,\xi}+O(\xi^2a+a^2),
\end{align*}
and $\l\q_1,\q_2\r=\l\q_2,\q_1\r=0+O(\xi^2a+a^2)$ as $\xi$, $a\to 0$, where $c_{n,\xi}$ is in \eqref{def:cn}.  We then write
\begin{align*}
\L(\x+\xi,a)=\L(\x+\xi,0)-&ae^{-i\xi_0z}\partial_z(\mathbf{C}_0\cos z)e^{i\xi_0z} \\
-&i\xi ae^{-i\xi_0z}(\mathbf{C}_0\cos z)e^{i\xi_0z}+O(\xi^2a+a^2)
\end{align*}
as $\xi$, $a\to 0$ (see \eqref{def:Lx} and \eqref{E:hu-small}), where $\mathbf{C}_0$ is in \eqref{def:Mn}. We use \eqref{def:q}, and we make a lengthy but explicit calculation to show that 
\begin{align*}
\L\q_{k}=&i\omega(n_k+\x+\xi,\mp)\begin{pmatrix} 1\\ \pm c_{n_k,\xi}\end{pmatrix}e^{in_k z} \\
&+ia(n_k+1+\x+\xi)\mathbf{C}_{n_k+1,\xi}\mathbf{q}_{n_k,1}e^{i(n_k+1)z} \\
&+ia(n_k-1+\x+\xi)\mathbf{C}_{n_k-1,\xi}\mathbf{q}_{n_k,-1} e^{i(n_k-1)z}\\
&-\frac12 ia\begin{pmatrix}c_0+c_{n_k,\xi}\\ c_0c_{n_k,\xi}\end{pmatrix}
((n_k+1+\x+\xi)e^{i(n_k+1)z}+(n_k-1+\x+\xi)e^{i(n_k-1)z})\\
&+O(\xi^2a+a^2)
\end{align*}
for $k=1$, $2$, as $\xi$, $a \to 0$, where $\mathbf{C}_{n,\xi}$ is in \eqref{def:Mn}. Exact formulae of $\mathbf{C}_{n_k\pm1,\xi}\mathbf{q}_{n_k,\pm1}$ are lengthy and tedious; see Section~\ref{sec:2,0} for instance. But they do not influence the result. Hence we omit the details. Continuing, we take the $L^2(\mathbb{T}) \times L^2(\mathbb{T})$ inner products of the above and \eqref{def:q}, and we make a lengthy but explicit calculation to show that 
\[
\frac{\l \L\q_{k},\q_{k}\r}{\l \q_{k},\q_{k}\r}=i\omega(n_k+\x+\xi,\mp)+O(\xi^2a+a^2)
\]
for $k=1$, $2$, and 
\[
\frac{\l \L\q_1,\q_2\r}{\l \q_1,\q_1\r}=\frac{\l \L\q_2,\q_1\r}{\l \q_2,\q_2\r}=0+O(\xi^2a+a^2)
\]
as $\xi, a\to 0$.

Together, \eqref{def:L4} and \eqref{def:I4} become
\[
\mathbf{L}(\xi,a)=\begin{pmatrix} i\omega(n_1+\x+\xi,-) & 0 \\ 0 & i\omega(n_2+\x+\xi,+) \end{pmatrix}+O(\xi^2a+a^2)
\]
and $\mathbf{I}(\xi,a)=\mathbf{I}+O(\xi^2a+a^2)$
as $\xi$, $a\to 0$, where $\mathbf{I}$ means the $2\times 2$ identity matrix. Clearly, for $\xi$, $a\in\mathbb{R}$ and $|\xi|$, $|a|$ sufficiently small, the roots of $\det(\mathbf{L}-\lambda\mathbf{I})(\xi,a)$ are purely imaginary up to terms of order $a$. Therefore, a sufficiently small, periodic wave train of \eqref{E:main}-\eqref{def:c} is spectrally stable to square integrable perturbations away from the origin in $\mathbb{C}$ to the linear order in the amplitude parameter. To compare, it is spectrally unstable in the vicinity of the origin in $\mathbb{C}$ at the linear order in the amplitude parameter if the modulational instability takes place. Hence, the modulational instability dominates the spectral instability away from the origin for \eqref{E:main}-\eqref{def:c}, if the latter takes place.

Numerical computations in \cite{McLean1981, McLean1982, MS1986, DO, AN2014}, for instance, report that nonzero colliding eigenvalues of the linear operator for the water wave problem contribute to spectral instability as the amplitude increases. The results are implicit, but the growth rate of an unstable eigenvalue seems the steepest at the origin. For instance, for $(n_1,-)=(2,-)$, $(n_2,+)=(6,+)$ and the colliding eigenvalue at $i3.353\dots$ but for $a=0.245\dots$, the unstable eigenvalue grows like $a^4$; see \cite{MS1986}, for instance, for details. But it is difficult to analytically find colliding eigenvalues away from the origin in $\mathbb{C}$ for $a\neq0$. 

In Section~\ref{sec:2,0}, we calculate some higher order terms near the colliding eigenvalues for $(n_1,-)=(2,-)$ and $(n_2,+)=(0,+)$. Unfortunately, we do not detect spectral instability up to the orders of $\xi a$ and $a^2$. But the result seems to agree with that in \cite{AN2014}, for instance, from a numerical computation for the physical problem. 

\section{Effects of surface tension}\label{sec:ST}

The results in the previous sections may be adapted to other related equations. We illustrate this for the full-dispersion shallow water equations in the presence of the effects of surface tension. That is, 
\begin{equation}\label{def:cT}
\c(\k;T):=\sqrt{(1+T\k^2)\frac{\tanh\k}{\k}}
\end{equation}
replaces \eqref{def:c}, where $T$ is the coefficient of surface tension. Throughout the section, we employ the notation in Section~\ref{sec:existence} and Section~\ref{sec:MI}. 

\subsection*{Properties of $\c(\cdot\,;T)$}

\begin{figure}[h]
(a)~~\includegraphics[scale=0.4]{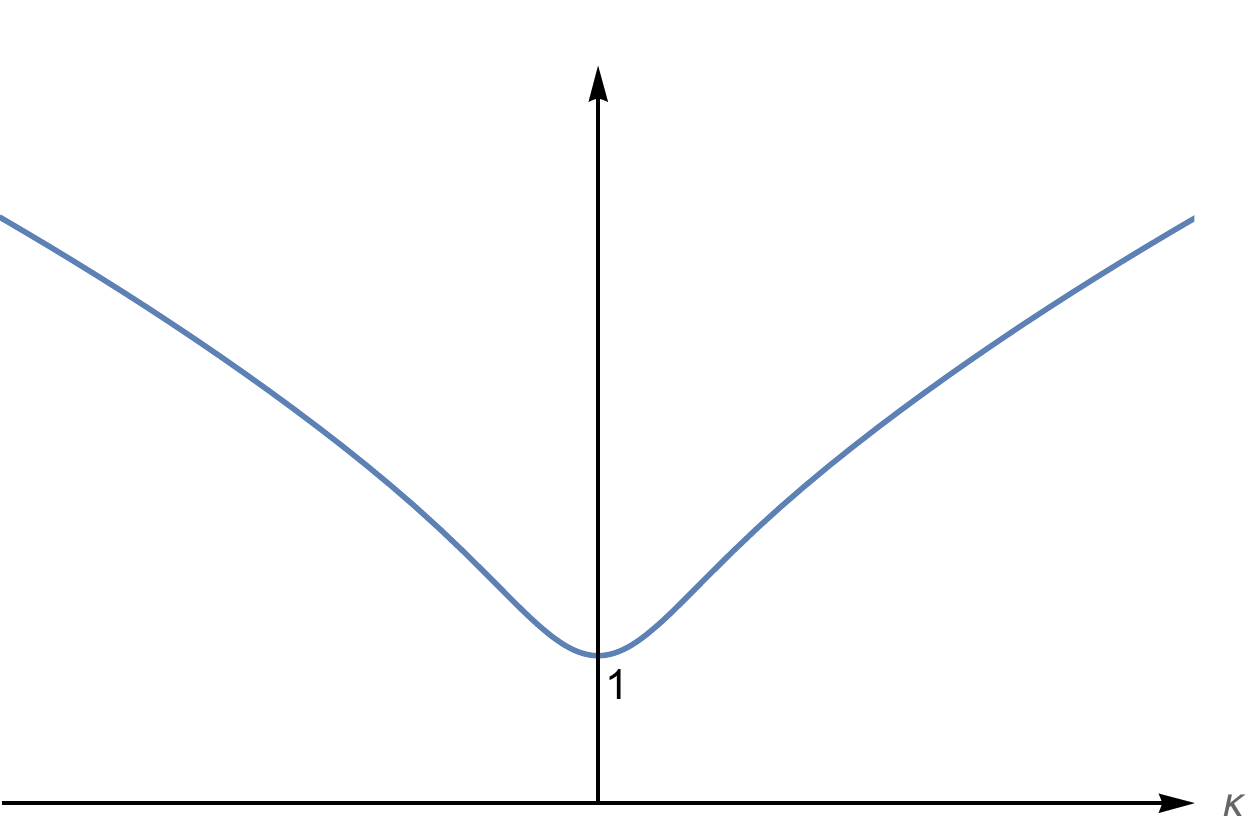}\qquad(b)~~\includegraphics[scale=0.4]{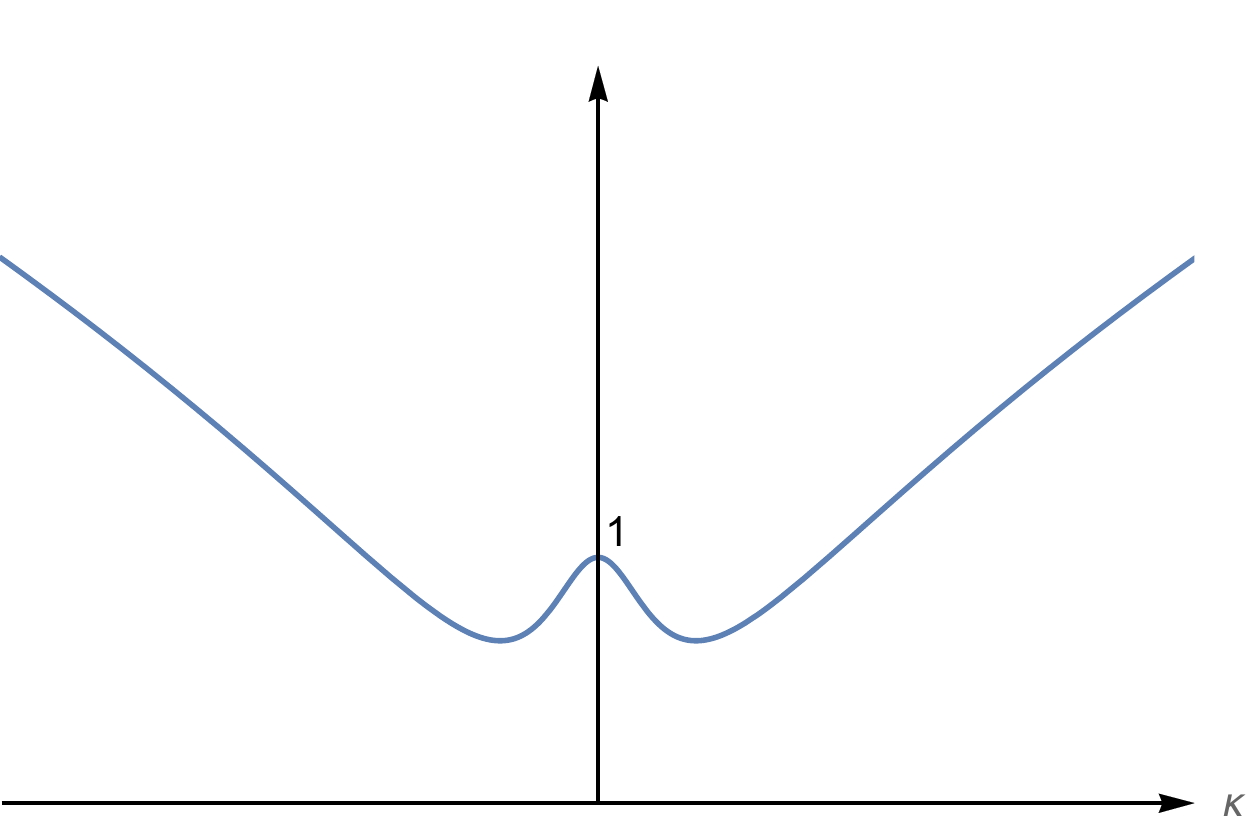}
\caption{Schematic plots of $\c(\cdot\,;T)$ when (a)~$T\geq1/3$ and (b)~$0<T<1/3$.}\label{fig:cT}
\end{figure}

For any $T>0$, since 
\[
\c^2(\k;T)=(1+T\k^2)\c^2(\k),
\] 
note from Section~\ref{ss:c} that $\c^2(\cdot\,;T)$ is even and real analytic, and $\c^2(0;T)=1$. Moreover, $\c^2(|\partial_x|;T)$ may be regarded equivalent to $1+|\partial_x|$ in the $L^2$-Sobolev space setting. In particular, $\c^2(|\partial_x|;T):H^{s+1}(\mathbb{R})\to H^{s}(\mathbb{R})$ for any $s\in\mathbb{R}$.

When $T\geq1/3$, note that $\c(\cdot\,;T)$ increases monotonically and unboundedly away from the origin. When $0<T<1/3$, on the other hand, $\c'(0;T)=0$, $\c''(0;T)<0$ and $\c(\k;T)\to\infty$ as $\k\to\infty$. Hence $\c(\cdot\,;T)$ possesses a unique minimum over the interval $(0,\infty)$; see Figure~\ref{fig:cT}. 

\subsection*{Well-posedness}

For any $T>0$, it follows from harmonic analysis techniques that the solution of the linear part of \eqref{E:main} and \eqref{def:cT} acquires a $1/4$ derivative of ``smoothness," compared to the initial datum. By the way, for $T=0$, the solution does not possess smoothing effects. Nevertheless, it seems difficult to work out the well-posedness in spaces of low regularities. But, for the present purpose, it suffices to solve the Cauchy problem in some functional analytic setting. In Appendix~\ref{sec:LWP}, we comment how to establish the local-in-time well-posedness for \eqref{E:main} and \eqref{def:cT} in $H^s(\mathbb{R})\times H^{s+1/2}(\mathbb{R})$ for any $s>2$. 

\subsection{Existence of sufficiently small, periodic wave trains}

Let $T>0$. We begin by discussing periodic wave trains of \eqref{E:main} and \eqref{def:cT}. That is, $\eta$ and $u$ are $2\pi$ periodic functions of $z:=\k(x-ct)$ for some $\k>0$, the wave number, for some $c>0$, the wave speed, and they solve
\begin{equation}\label{E:periodic5}
\begin{aligned}
&-c\eta+u+u\eta=(1-c^2)b_1,\\
&-cu+\c^2(\k|\partial_z|;T)\eta+\frac12u^2=(1-c^2)b_2
\end{aligned}
\end{equation}
for some $b_1$, $b_2\in\mathbb{R}$; compare \eqref{E:periodic}. For any $T>0$, note that
\begin{equation}\label{E:c-bound5}
\c^2(\k|\partial_z|;T):H^{k+1}(\mathbb{T})\to H^{k}(\mathbb{T})
\qquad\text{for any $\k>0$\quad for any integer $k\geq 0$}.
\end{equation}
Note that
\begin{equation}\label{def:ck5}
\c^2(\k|\partial_z|;T)e^{inz}=\c^2(n\k;T)e^{inz}\qquad\text{for $n\in\mathbb{Z}$};
\end{equation}
compare \eqref{E:c-bound} and \eqref{def:ck}.

Here the existence proof follows along the same line as that in Section~\ref{sec:existence}. Hence we merely hit the main points. We use \eqref{def:uv} whenever it is convenient to do so.

\begin{lemma}[Regularity]\label{lem:regularity5}
For any $T>0$, if $\eta$, $u\in H^1(\mathbb{T})$ solve \eqref{E:periodic5} for some $c>0$, and $\k>0$, $b_1$, $b_2\in\mathbb{R}$ and if $1-\|\eta\|_{L^\infty(\mathbb{T})}\geq\epsilon>0$ for some $\epsilon$ then $\eta$, $u\in H^\infty(\mathbb{T})$.
\end{lemma}

\begin{proof} 
We rearrange \eqref{E:periodic5} as 
\begin{equation}\label{E:hu'5}
u=\frac{1}{1+\eta}(c\eta+(1-c^2)b_1)\quad\text{and}\quad \c^2(\k|\partial_z|)\eta=cu-\frac12u^2+(1-c^2)b_2.
\end{equation}
Since $u\in H^1(\mathbb{T})$ by hypothesis, it follows from the latter equation of \eqref{E:hu'5}, a Sobolev inequality, and \eqref{E:c-bound5} that $\eta\in H^2(\mathbb{T})$. Since $\frac{1}{1+\eta}:H^2(\mathbb{T})\to H^2(\mathbb{T})$ by hypothesis, it follows from the former equation of \eqref{E:hu'5} that $u\in H^2(\mathbb{T})$. A bootstrap argument then completes the proof. Compare Lemma~\ref{lem:regularity} and the proof.
\end{proof}

For any $T>0$, let (by abuse of notation) $\f: H^1(\mathbb{T}) \times H^1(\mathbb{T}) \times \mathbb{R}^+ \times \mathbb{R}^+\times \mathbb{R}\times \mathbb{R} \to L^2(\mathbb{T}) \times L^2(\mathbb{T})$ such that
\[
\f(\u,c;T,\k,b_1,b_2)=
\begin{pmatrix}-c\eta+u+u\eta-(1-c^2)b_1 \\ -cu+\c^2(\k|\partial_z|;T)\eta+\frac12u^2-(1-c^2)b_2\end{pmatrix};
\]
compare \eqref{def:f}. It is well defined by \eqref{E:c-bound5} and a Sobolev inequality. We seek a solution $\u\in H^1(\mathbb{T})\times H^1(\mathbb{T})$, $c>0$, and $\k>0$, $b_1$, $b_2\in\mathbb{R}$ of 
\[
\f(\u,c;T,\k,b_1,b_2)=\mathbf{0}
\] 
satisfying $1-\|\eta\|_{L^\infty(\mathbb{T})}\geq\epsilon>0$ for some $\epsilon$ and, by virtue of Lemma~\ref{lem:regularity5}, a solution $\u\in H^\infty(\mathbb{T})\times H^\infty(\mathbb{T})$ of \eqref{E:periodic5}. We may repeat the argument in Section~\ref{ss:operator} to verity that $\f$ is a real analytic operator.

For any $T>0$, for any $c>0$, $\k>0$, $b_1$, $b_2\in\mathbb{R}$ and $|b_1|$, $|b_2|$ sufficiently small, note that $\u_0:=\begin{pmatrix} \eta_0 \\ u_0 \end{pmatrix}(c;T,\k,b_1,b_2)$ makes a constant solution of $\f(\u,c;T,\k,b_1,b_2)=\mathbf{0}$ and, hence, \eqref{E:periodic5}, where $\eta_0$ and $u_0$ are in \eqref{E:hu0'}. It follows from the implicit function theorem that if non-constant solutions bifurcate from $\u=\u_0$ for some $c=c_0$ then, necessarily, (by abuse of notation)
\[
\mathbf{L}_0:=\partial_\u\f(\u_0,c_0;T,\k,b_1,b_2):
H^1(\mathbb{T})\times H^1(\mathbb{T})\to L^2(\mathbb{T})\times L^2(\mathbb{T})
\]
is not an isomorphism. This is not in general a sufficient condition, but note from Section~\ref{sec:existence} that bifurcation does take place, provided that the kernel of $\mathbf{L}_0$ is two dimensional. Note that 
\[
\mathbf{L}_0\u_1e^{inz}=
\begin{pmatrix} u_0-c_0 & 1+\eta_0 \\ \c^2(\k|\partial_z|;T) & u_0-c_0 \end{pmatrix}\u_1e^{inz}
\qquad\text{for $n\in\mathbb{Z}$}
\]
for some nonzero $\u_1$ if and only if 
\[
(c_0-u_0)^2=\c^2(n\k;T)(1+\eta_0);
\] 
compare \eqref{E:bifurcation}. For $b_1=b_2=0$ and, hence, $\eta_0=u_0=0$ by \eqref{E:hu0'}, it simplifies to $c_0=\pm\c(n\k;T)$. Without loss of generality, we restrict the attention to $n=1$ and we assume the $+$ sign. For $|b_1|$ and $|b_2|$ sufficiently small, we then make an explicit calculation to find \eqref{E:hu0}, where $\c(\cdot\,;T)$ replaces $\c$. 

When $T\geq1/3$, since $\c(\k;T)<\c(n\k;T)$ for any $n=2,3,\dots$ pointwise in $\mathbb{R}$ (see Figure~\ref{fig:cT}a), a straightforward calculation reveals that for any $\k>0$, $b_1$, $b_2\in\mathbb{R}$ and $|b_1|$, $|b_2|$ sufficiently small, the $H^1(\mathbb{T})\times H^1(\mathbb{T})$ kernel of $\mathbf{L}_0=\partial_\u\f(\u_0,c_0;T,\k,b_1,b_2)$ is two dimensional and spanned by $\u_1e^{\pm iz}$, where $\u_1$ is in \eqref{E:hu1} and $\c(\cdot\,;T)$ replaces $\c$. Hence, non-constant solutions bifurcate from $\u=\u_0$ and $c=c_0$. 

When $0<T<1/3$, on the other hand, for any integer $n\geq 2$, it is possible to find some $\k$ such that $\c(\k;T)=\c(n\k;T)$ (see Figure~\ref{fig:cT}b). If $\c(\k;T)\neq\c(n\k;T)$ for any $n=2,3,\dots$ then the kernel of $\mathbf{L}_0$ is likewise two dimensional. Hence, non-constant solutions bifurcate from $\u=\u_0$ and $c=c_0$. But if $\c(\k;T)=\c(n\k;T)$ for some integer $n\geq 2$, resulting in the resonance of the fundamental mode and the $n$-th harmonic, then the kernel is four dimensional.  

To compare, for $T=0$, recall that $\c(\k;0)<\c(n\k;0)$ for any $n=2,3,\dots$ pointwise in $\mathbb{R}$ (see Figure~\ref{fig:c}). Hence, for any $\k>0$, $b_1$, $b_2\in\mathbb{R}$ and $|b_1|$, $|b_2|$ sufficiently small, the kernel is two dimensional.

To proceed, for any $T>0$, for any $\k>0$ satisfying 
\begin{equation}\label{E:k-condition}
\c(\k;T)\neq \c(n\k;T), \qquad n=2,3,\dots,
\end{equation}
$b_1$, $b_2\in\mathbb{R}$ and $|b_1|$, $|b_2|$ sufficiently small, we may repeat the Lyapunov-Schmidt procedure in Section~\ref{ss:LS} to establish that a one parameter family of solutions of \eqref{E:periodic5} exists, denoted (by abuse of notation) $\eta(a;T,\k,b_1,b_2)(z)$, $u(a;T,\k,b_1,b_2)(z)$, and $c(a;T,\k,b_1,b_2)$, near $\eta_0(T,\k,b_1,b_2)$, $u_0(T,\k,b_1,b_2)$, and $c_0(T,\k,b_1,b_2)$, for $a\in\mathbb{R}$ and $|a|$ sufficiently small. Note that $\eta$ and $u$ are $2\pi$ periodic and even in $z$, and they belong to $H^\infty(\mathbb{T})$. Note that $\eta$, $u$, and $c$ depend analytically on $a$, and $\k$, $b_1$, $b_2$. Moreover, we may repeat the small amplitude expansion in Section~\ref{ss:small} to verify \eqref{E:hu-small} and \eqref{E:hu0} as $a$, $b_1$, $b_2\to 0$, where $\c(\cdot\,;T)$ replaces $\c$. We omit the details. 

If $\c(\k;T)=\c(n\k;T)$ for some integer $n\geq 2$ for some $\k>0$ then the proof in Section~\ref{sec:existence} breaks down. Nevertheless, one may employ the Lyapunov-Schmidt procedure in \cite{Jones}, for instance, to prove the existence of sufficiently small, periodic wave trains of \eqref{E:main} and \eqref{def:cT}. But the modulational instability calculation becomes tedious, involving $6\times 6$ matrices. We do not discuss the details. 

\subsection{Modulational stability and instability}

Let $T>0$. Let $\eta=\eta(a;T,\k,0,0)$, $u=u(a;T,\k,0,0)$, and $c=c(a;T,\k,0,0)$, for some $a\in\mathbb{R}$ and $|a|$ sufficiently small for some $\k>0$ satisfying \eqref{E:k-condition}, denote a $2\pi/\k$-periodic wave train of \eqref{E:main} and \eqref{def:cT} near the rest state, whose existence follows from the previous subsection. We turn the attention to its modulational stability and instability. Recall from Section~\ref{ss:MI} that the modulational instability means that the $L^2(\mathbb{T}) \times L^2(\mathbb{T})$ spectra of (by abuse of notation)
\[
\L(\xi)(a;T,\k,0,0):=
e^{-i\xi z}\partial_z \begin{pmatrix}c-u &-1-\eta \\ -\c^2(\k|\partial_z|;T) & c-u \end{pmatrix}(a;T,\k,0,0)e^{i\xi z}
\]
are not contained in the imaginary axis in the vicinity of the origin for $\xi>0$ and small.

Here the modulational stability and instability proof follows along the same line as that in Section~\ref{sec:MI}. Hence we merely hit the main points. In the sequel, $\k>0$ satisfying \eqref{E:k-condition} is suppressed for simplicity of notation, unless specified otherwise. We use the notation of \eqref{def:Lxa}. 

For any $T>0$ for $a=0$, a straightforward calculation reveals that (by abuse of notation)
\[
\mathcal{L}(\xi,0)\e(n+\xi,\pm;T)=i\omega(n+\xi,\pm;T)\e(n+\xi,\pm;T)
\qquad\text{for $n\in\mathbb{Z}$ and $\xi\in[0,1/2]$},
\]
where 
\[
\omega(n+\xi,\pm;T)=(n+\xi)(\c(\k;T)\pm\c(\k(n+\xi);T))
\]
and $\e(n+\xi,\pm;T)(z)=\begin{pmatrix}1\\ \mp\c(\k(n+\xi);T)\end{pmatrix}e^{inz}$; compare \eqref{E:Lx0=0} and \eqref{def:eigen}. Note that 
\[
\omega(0,+;T)=\omega(0,-;T)=\omega(1,-;T)=\omega(-1,-;T)=0.
\]
Since $\c(\k;T)\neq\c(n\k;T)$ for any $n=2,3,\dots$ by hypothesis, a straightforward calculation reveals that zero is an $L^2(\mathbb{T})\times L^2(\mathbb{T})$ eigenvalue of $\L(0,0)$ with algebraic and geometric multiplicity four. Moreover, \eqref{def:p00} makes the associated eigenfunctions, where $\c(\cdot\,;T)$ replaces $\c$. 

For $\xi=0$ for $a\in\mathbb{R}$ and $|a|$ sufficiently small, one may repeat the proof of Lemma~\ref{lem:L} to establish that zero is an $L^2(\mathbb{T})\times L^2(\mathbb{T})$ eigenvalue of $\L(0,a)$ with algebraic multiplicity four and geometric multiplicity three. Moreover, \eqref{def:p0a} makes the associated eigenfunctions, where $\c(\cdot\,;T)$ replaces $\c$. We omit the details.

For any $T>0$, for $\xi>0$, $a\in\mathbb{R}$ and $\xi$, $|a|$ sufficiently small, one may then proceed as in Section~\ref{ss:perturbation} and calculate \eqref{def:L3} and \eqref{def:I3} up to terms of orders of $\xi^2$, $\xi a$, and $a$, where $\p_1$, $\p_2$, $\p_3$, $\p_4$ are in \eqref{def:p} but $\c(\cdot\,;T)$ replaces $\c$. It follows from perturbation theory (see \cite[Section~4.3.5]{K}, for instance, for details) that the roots of $\det(\mathbf{L}-\lambda\mathbf{I})(\xi,a)$ coincide with the $L^2(\mathbb{T})\times L^2(\mathbb{T})$ spectrum of $\L(\xi, a)$ up to terms of orders $\xi^2$ and $a$ as $\xi$, $a\to 0$. We then repeat the argument of Section~\ref{ss:index} and derive a modulational instability index for \eqref{E:main} and \eqref{def:cT}. 

\begin{theorem}[Modulational instability index]\label{thm:indexT}
For any $T>0$ and $T\neq 1/3$ for any $\k>0$ satisfying $\cc(\k;T)\neq \cc(n\k;T)$ for $n=2,3,\dots$, a sufficiently small, $2\pi/\k$-periodic wave train of \eqref{E:main} and \eqref{def:cT} is modulationally unstable, provided that
\[
\Delta(\k;T):=\Big(\frac{i_1i_2}{i_3}i_4\Big)(\k;T)<0,
\]
where
\begin{align*}
i_1(\k;T)=&(\k\cc(\k;T))'', \\
i_2(\k;T)=&((\k\cc(\k;T))')^2-1, \\
i_3(\k;T)=&\cc^2(\k;T)-\cc^2(2\k;T), 
\intertext{and}
i_4(\k;T)=&3\cc^2(\k;T)+5\cc^4(\k;T)-2\cc^2(2\k;T)(\cc^2(\k;T)+2) \\
&+18\k(\cc^3\cc')(\k;T)+\k^2(\cc')^2(\k;T)(5\cc^2(\k;T)+4\cc^2(2\k;T));\hspace*{-20pt}
\end{align*}
$\cc(\k,T)$ is in \eqref{def:cT}. It is spectrally stable to square integrable perturbations in the vicinity of the origin otherwise. 
\end{theorem}

The proof is nearly identical to that of Theorem~\ref{thm:index}. Hence we omit the details. If $T=1/3$ then the result becomes inconclusive.

Theorem~\ref{thm:indexT} elucidates four resonance mechanisms which contribute to the sign change in $\Delta$ and, ultimately, the change in the modulational stability and instability for \eqref{E:main} and \eqref{def:cT}. When $T=0$, note that $\Delta(\k;0)$ becomes \eqref{def:ind}. But for $T>0$, several differences are present. For instance, $i_2(\k;0)<0$ for any $\k>0$, but the effects of surface tension are to increase the group velocity, and they may do so to the extent that $i_2$ changes the sign. 


When $T>1/3$, since $\c(\k;T)$ and $(\k\c(\k;T))'$ increase monotonically over the interval $(0,\infty)$ and since $(\k\c(\k;T))'$ does not possess an extremum by brutal force, $i_1$, $i_2$, $i_3$ in Theorem~\ref{thm:indexT} do not vanish over the interval $(0,\infty)$. Moreover, a numerical evaluation reveals that $i_4$ changes the sign once over the interval $(0,\infty)$. Together, a sufficiently small, periodic wave train of \eqref{E:main} and \eqref{def:cT} is modulationally unstable, provided that the wave number is greater than a critical value, and it is modulationally stable otherwise; compare Corollary~\ref{cor:kc}. Furthermore, a numerical evaluation reveals that the critical wave number $\k_c(T)$, say, satisfies 
\[
\lim_{T\to \infty}\sqrt{T}\k_c(T)\approx 1.054.
\]

When $0<T<1/3$, on the other hand, a straightforward calculation reveals that $(\k\c(\k;T))'$ achieves a unique minimum over the interval $(0,\infty)$. Moreover, $(\k\c(\k;T))'=1$ and $\c(\k;T)=\c(2\k;T)$ each takes one transvere root over the interval $(0,\infty)$. Hence, $i_1$ through $i_4$ each contributes to the change in the modulational stability and instability.

\begin{figure}[h]
\includegraphics[scale=0.75]{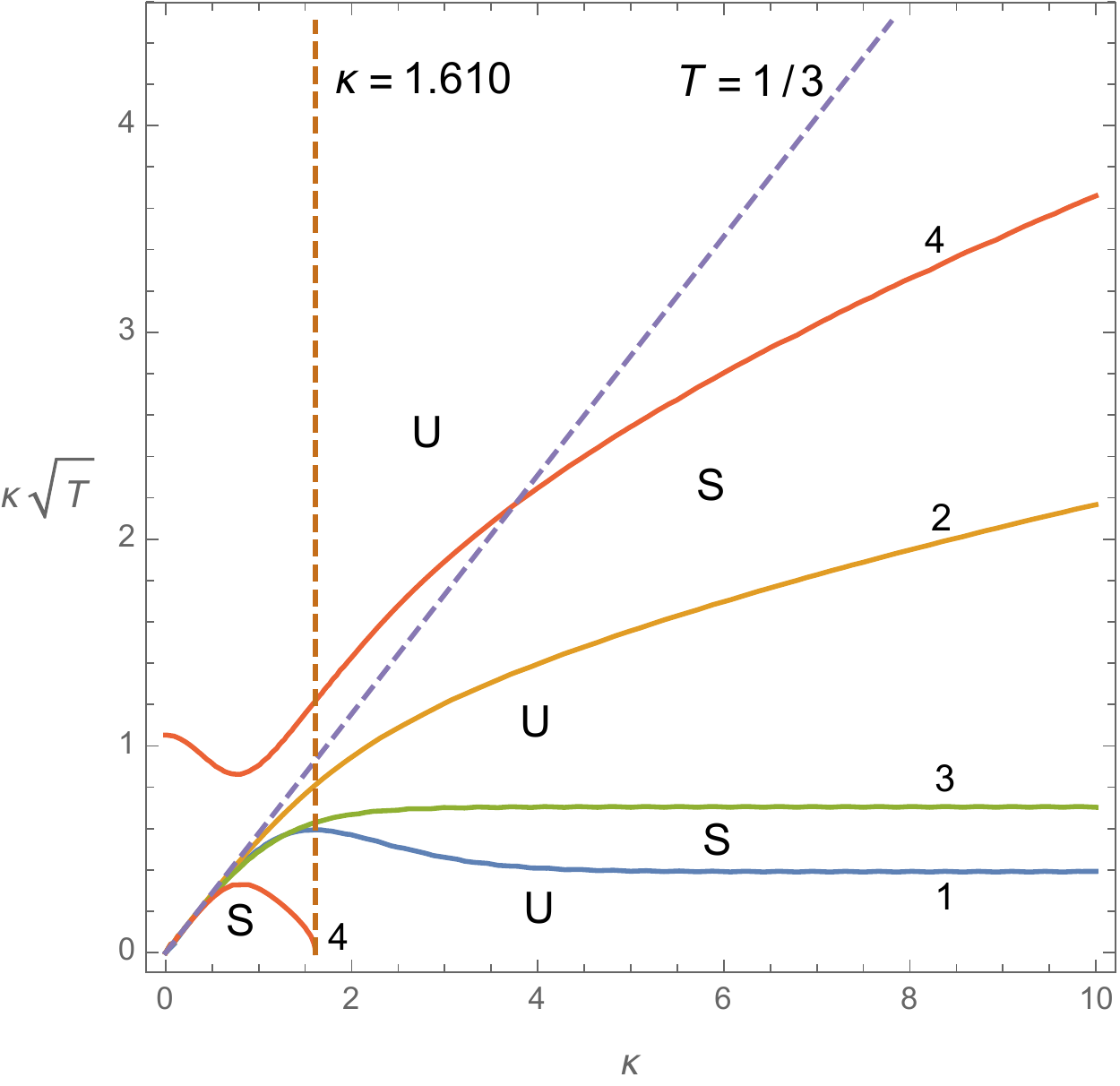}
\caption{Stability diagram for sufficiently small, periodic wave trains of \eqref{E:main} and \eqref{def:cT}. To interpret, for any $T> 0$, one must envision a line through the origin with slope $T$.  ``S" and ``U" denote stable and unstable regions. Solid curves represent roots of the modulational instability index and are labeled according to their mechanism.}\label{fig:ST}
\end{figure}

Figure~\ref{fig:ST} illustrates in the $\k$ versus $\k\sqrt{T}$ plane the regions where a sufficiently small, periodic wave train of \eqref{E:main} and \eqref{def:cT} is modulationally stable and unstable. Along Curve~1, $i_1(\k;T)=0$ and the group speed achieves an extremum at the wave number $\k$. Curve~2 is associated with $i_2(\k;T)=0$, along which the group speed coincides with the phase speed in the long wave limit as $\k\to0$, resulting in the resonance of short and long waves. In the deep water limit, as $\k\to\infty$ while $\k\sqrt{T}$ is fixed, it is asymptotic to $\k=\frac94\k^2T-\frac34$. Curve~3 is associated with $i_3(\k;T)=0$, along which the phase speeds of the fundamental mode and the second harmonic coincide, resulting in the second harmonic resonance. In the deep water limit, it is asymptotic to $k^2T=\frac12$. Moreover along Curve~4, $i_4$ vanishes as a result of a rather complicated balance of the dispersion and nonlinear effects. The ``lower" branch of Curve~4 passes through $\k=1.160\dots$, the critical wave number when $T=0$; see Corollary~\ref{cor:kc}. The ``upper" branch passes through $\k\sqrt{T}=1.054\dots$, the large surface tension limit in \cite{Kawahara}, for instance.

The result qualitatively agrees with those in \cite{Kawahara} and \cite{DR}, for instance, from formal asymptotic expansions for the water wave problem. To compare, the Whitham equation (see \eqref{E:Whitham}) in the presence of the effects of surface tension fails to predict the critical wave number in the large surface tension limit; see \cite{HJ3}, for instance, for details. Perhaps, this is because the Whitham equation neglects many ``higher order" nonlinearities of the physical problem. It is fortuitous that the full-dispersion shallow water equations proposed herein includes more physically realistic nonlinearities to predict all resonances in gravity capillary waves. 

Note that along Curve~3, the modulational instability index becomes singular. 
It is interesting to justify the resonant interactions in \cite{MG1970a, MG1970b}, for instance, from formal asymptotic expansions for the physical problem.

\subsection*{Acknowledgements}

The authors thank Bernard Deconinck, Mariana Haragus, Mathew Johnson, Hanrik Kalisch, and Olga Trichenko for helpful discussions. 

VMH is supported by the National Science Foundation grant CAREER DMS-1352597, an Alfred P. Sloan research fellowship, the Arnold O. Beckman research award RB14100 of the Office of the Vice Chancellor for Research and a Beckman fellowship of the Center for Advanced Study at the University of Illinois at Urbana-Champaign. AKP is supported by CAREER DMS-1352597 and RB14100.

\begin{appendix}

\section{Well-posedness}\label{sec:LWP}

We discuss the solvability of the Cauchy problem associated with \eqref{E:main}-\eqref{def:c} or, equivalently,
\begin{equation}\label{E:main'}
\begin{aligned}
&\partial_t\eta+\partial_xu+\partial_x(u\eta)=0,\\
&\partial_tu-\mathcal{H}\eta+\gamma^2(|\partial_x|)\eta+u\partial_xu=0.
\end{aligned}
\end{equation}
For $v\in L^2(\mathbb{R})$, the Hilbert transform of $v$ is written $\mathcal{H}v$ and defined in the Fourier space as 
\[
\widehat{\mathcal{H}v}(\k)=-i(\text{sgn}\,\k)\widehat{v}(\k).
\] 
Since 
\[
|\text{sgn}\,\k-\tanh\k|\leq e^{-|\k|}\qquad\text{pointwise in $\mathbb{R}$}
\]
by brutal force (see also \cite{Yosihara}, for instance), 
\begin{equation}\label{E:gamma}
\|\gamma^2(|\partial_x|)v\|_{H^s(\mathbb{R})}\leq C\|v\|_{L^2(\mathbb{R})}\qquad\text{for any $s\geq0$}
\end{equation}
for some constant $C>0$ independent of $v$. 

\begin{theorem}[Local-in-time well-posedness]\label{thm:LWP}
For any $s>2$ for any $\eta_0\in H^s(\mathbb{R})$ and $u_0\in H^{s+1/2}(\mathbb{R})$, a unique solution, denoted $\eta(t)=\eta(\cdot,t)$ and $u(t)=u(\cdot,t)$, of  \eqref{E:main'}-\eqref{E:gamma}, 
\[
\eta(\cdot,0)=\eta_0\quad\text{and}\quad u(\cdot,0)=u_0,
\] 
exists in $H^s(\mathbb{R})\times H^{s+1/2}(\mathbb{R})$ over the interval $[0,\t)$ for some $\t>0$. Moreover, $(\eta_0,u_0)\mapsto(\eta(t),u(t))$ is continuous for any $t\in[0,\t)$.
\end{theorem}

The proof follows along the same line as the argument in \cite{Kato-LWP}, for instance. The main difference is how one establishes an a priori bound. The same proof works in the presence of the effects of surface tension, for which
\[
\partial_tu+(1-T\partial_x^2)(-\mathcal{H}+\gamma^2(|\partial_x|))\eta+u\partial_xu=0
\]
replaces the latter equation of \eqref{E:main'}. The same proof works in the periodic setting. 

\subsection*{Preliminaries}

Note that $\|\mathcal{H}v\|_{L^2(\mathbb{R})}=\|v\|_{L^2(\mathbb{R})}$ and $\mathcal{H}^2=-1$. Note that $|\partial_x|:=\mathcal{H}\partial_x$ is self adjoint, and
\[
\int^\infty_{-\infty}(v^2+u|\partial_x|v)~dx
\]
is equivalent to $\|v\|_{H^{1/2}(\mathbb{R})}^2$. Moreover, commutators of $\mathcal{H}$ and $|\partial_x|$ are ``smoothing."

\begin{lemma}[Smoothing]\label{L:smoothing}
\begin{equation}\label{E:comm}
\int^\infty_{-\infty}fv|\partial_x|v~dx\leq C\|f\|_{H^{3/2+}(\mathbb{R})}\|v\|_{H^{1/2}(\mathbb{R})}^2
\quad\text{and}\quad
\int^\infty_{-\infty}f(\partial_xv)\mathcal{H}\partial_xv~dx\leq C\|f\|_{H^{5/2+}(\mathbb{R})}\|v\|_{L^2(\mathbb{R})}^2
\end{equation}
for some constant $C>0$ independent of $f$ and $v$.
\end{lemma}

\begin{proof}
Note that $|\partial_x|^{1/2}$ is self adjoint, and we calculate that 
\[
\int fv|\partial_x|u~dx=\int f(|\partial_x|^{1/2}v)^2~dx+\int(|\partial_x|^{1/2}[|\partial_x|^{1/2},f]v)v~dx.
\]
Clearly, the first term of the right side is bounded by $\|f\|_{L^\infty(\mathbb{R})}\||\partial_x|^{1/2}v\|_{L^2(\mathbb{R})}^2$. We claim that the second term of the right side is bounded by $\||\k|\widehat{f}\|_{L^1(\mathbb{R})}\|v\|_{H^{1/2}(\mathbb{R})}^2$ up to the multiplication by a constant. Indeed, since 
\[
(|\partial_x|^{1/2}[|\partial_x|^{1/2},f]v)\,\widehat{}\,(\k)
=\frac{1}{\sqrt{2\pi}}\int^\infty_{-\infty}|\k|^{1/2}(|\k|^{1/2}-|\k_1|^{1/2})\widehat{f}(\k-\k_1)\widehat{v}(\k_1)~d\k_1
\]
and since $|\k|^{1/2}||\k|^{1/2}-|\k_1|^{1/2}|\leq C|\k-\k_1|$ for any $\k,\k_1\in\mathbb{R}$ for some constant $C>0$ by brutal force (see also \cite{Yosihara}, for instance), it follows from Young's inequality and the Parseval theorem that 
\[
\||\partial_x|^{1/2}[|\partial_x|^{1/2},f]v\|_{L^2(\mathbb{R})}\leq 
C\||\k|\widehat{f}\|_{L^1(\mathbb{R})}\|v\|_{L^2(\mathbb{R})}
\]
for some constant $C>0$ independent of $f$ and $v$. H\"older's inequality then proves the claim. The first inequality of \eqref{E:comm} follows by a Sobolev inequality. 

Note that $\mathcal{H}$ is skew adjoint, and we calculate that 
\begin{align*}
\int f(\partial_xv)|\partial_x|v~dx=
&-\int f(\mathcal{H}\partial_xv)\partial_xv~dx-\int([\mathcal{H},f]\partial_xv)\partial_xv~dx\\
=&-\frac12\int([\mathcal{H},f]\partial_xv)\partial_xv~dx.
\end{align*}
Since
\[
(\partial_x[\mathcal{H},f]\partial_xv)\,\widehat{}\,(\k)
=-\frac{1}{\sqrt{2\pi}}\int^\infty_{-\infty}\k(\text{sgn}(\k)-\text{sgn}(\k_1))\widehat{f}(\k-\k_1)\k_1\widehat{v}(\k_1)~d\k_1
\]
and since $|\k|+|\k_1|=\k\,\text{sgn}\k+\k_1\text{sgn}\k_1\leq |\k-\k_1|$ whenever $\text{sgn}\k\neq\text{sgn}\k_1$ by brutal force (see also \cite{Yosihara}, for instance), it follows from Young's inequality and the Parseval theorem that 
\[
\|\partial_x[\mathcal{H},f]\partial_xv\|_{L^2(\mathbb{R})}\leq 
C\||\k|^2\widehat{f}\|_{L^1(\mathbb{R})}\|v\|_{L^2(\mathbb{R})}
\]
for some constant $C>0$ independent of $f$ and $v$. H\"older's inequality and a Sobolev inequality then prove the second inequality of \eqref{E:comm}. This completes the proof.
\end{proof}

\subsection*{A priori bound}

For $k\geq 1$ an integer, let
\begin{equation}\label{def:Ek}
E_k^2(t)=\frac12\|\eta\|_{L^2(\mathbb{R})}^2(t)+\frac12\|u\|_{L^2(\mathbb{R})}^2(t)+\sum_{\ell=1}^k e_\ell^2(t),
\end{equation}
where
\begin{equation}\label{def:el}
e_\ell^2(t)=\frac12\int^\infty_{-\infty}((\partial_x^\ell\eta(t))^2+(\partial_x^\ell u(t))|\partial_x|(\partial_x^\ell u(t)))~dx.
\end{equation}
Clearly, $E_k(t)$ is equivalent to $\|\eta\|_{H^k(\mathbb{R})}(t)+\|u\|_{H^{k+1/2}(\mathbb{R})}(t)$.

\begin{lemma}[A priori bound]\label{lem:energy}
For any integer $k\geq2$ if $\eta\in H^k(\mathbb{R})$ and $u\in H^{k+1/2}(\mathbb{R})$ solve \eqref{E:main'}-\eqref{E:gamma} over the interval $[0,t_0)$ for some $t_0>0$ then
\begin{equation}\label{I:energy}
E_k(t)\leq \frac{E_k(0)}{1-CE_k(0)t}\qquad\text{for any $t\in[0,t_1]$}
\end{equation}
for some constant $C>0$ independent of $\eta$ and $u$ for some $t_1\in (0,t_0)$ depending on $E_k(0)$. Moreover,
\begin{equation}\label{I:norm}
\|\eta\|_{H^k(\mathbb{R})}(t)+\|u\|_{H^{k+1/2}(\mathbb{R})}(t)
\leq C(t, \|\eta\|_{H^k(\mathbb{R})}(0), \|u\|_{H^{k+1/2}(\mathbb{R})}(0))
\end{equation}
for any $t\in[0,t_1]$.
\end{lemma}

\begin{proof}
For $\ell\geq 1$ an integer, we differentiate \eqref{def:el} with respect to $t$ and use \eqref{E:main'} to arrive at that
\begin{align*}
\frac{de_\ell^2}{dt}=&\int((\partial_x^\ell\partial_t\eta)(\partial_x^\ell\eta)
+(\partial_x^\ell\partial_tu)|\partial_x|(\partial_x^\ell u))~dx \notag \\
=&-\int\partial_x^\ell(\partial_xu+\partial_x(u\eta))(\partial_x^\ell\eta)~dx
-\int\partial_x^\ell(-\mathcal{H}\eta+\gamma^2(|\partial_x|)\eta+u\partial_xu)|\partial_x|(\partial_x^\ell u)~dx \notag \\
=:&(I)+(II)
\end{align*}
over the interval $(0,t_0)$. An integration by parts leads to that 
\begin{align}\label{e:I}
(I)=&-\int(\partial_x^{\ell+1}u)(\partial_x^\ell\eta)~dx+\Big(\ell+\frac12\Big)\int(\partial_xu)(\partial_x^\ell\eta)^2~dx \\
&-\int(\partial_x^{\ell+1}(u\eta)-u(\partial_x^{\ell+1}\eta)
-(\ell+1)(\partial_xu)(\partial_x^\ell\eta))(\partial_x^\ell\eta)~dx. \notag
\end{align}
Recall $|\partial_x|=\mathcal{H}\partial_x$. Since $\mathcal{H}$ is skew adjoint and $\mathcal{H}^2=-1$, moreover,
\begin{align}\label{e:II}
(II)=&-\int(\partial_x^{\ell+1}\eta)(\partial_x^\ell u)~dx
-\int(\mathcal{H}\partial_x^{\ell+1}\gamma^2(|\partial_x|)\eta)(\partial_x^\ell u)~dx\\
&-\int u(\partial_x^{\ell+1}u)\mathcal{H}(\partial_x^{\ell+1}u)~dx
-\ell\int(\partial_xu)(\partial_x^\ell u)(\mathcal{H}\partial_x^{\ell+1}u)~dx \notag \\
&-\int(\partial_x^\ell(u\partial_xu)-u(\partial_x^{\ell+1}u)
-\ell(\partial_xu)(\partial_x^\ell u))|\partial_x|(\partial_x^\ell u)~dx. \notag
\end{align}

Note that the first term of the right side of \eqref{e:I} and the first term of the right side of \eqref{e:II} cancel each other when added together after an integration by parts. Note that the second term of the right side of \eqref{e:I} is bounded by $(\ell+\frac12)\|\partial_xu\|_{L^\infty(\mathbb{R})}\|\partial_x^\ell\eta\|_{L^2(\mathbb{R})}^2$ and the last term of the right side of \eqref{e:I} is bounded by $\|u\|_{H^\ell(\mathbb{R})}\|\partial_x^\ell\eta\|_{L^2(\mathbb{R})}^2$ up to the multiplication by a constant by the Leibniz rule. Moreover, note that the second term of the right side of \eqref{e:II} is bounded by $\|\eta\|_{L^2(\mathbb{R})}\|\partial_x^\ell u\|_{L^2(\mathbb{R})}$ by \eqref{E:gamma}, the third and the fourth terms of the right side of \eqref{e:II} are bounded by $\|u\|_{H^{5/2+}(\mathbb{R})}\|\partial_x^\ell u\|_{H^{1/2}(\mathbb{R})}^2$ by \eqref{E:comm}. Note that for $\ell\geq 2$, the last term of the right side of \eqref{e:II} is bounded by $\|u\|_{H^{\ell+1/2}(\mathbb{R})}^2\||\partial_x|^{1/2}\partial_x^\ell u\|_{L^2(\mathbb{R})}$ up to the multiplication by a constant by the Leibniz rule and a Sobolev inequality. Together,
\begin{equation}\label{e:ej}
\frac{de_\ell^2}{dt}\leq C(1+\|u\|_{H^{5/2+}(\mathbb{R})}+\|u\|_{H^{\ell+1/2}(\mathbb{R})})(\|\eta\|_{H^\ell(\mathbb{R})}^2+\|u\|_{H^{\ell+1/2}(\mathbb{R})}^2)
\end{equation}
for any integer $\ell\geq 2$ over the interval $(0,t_0)$, for some constant $C>0$ independent of $\eta$ and $u$. 

To proceed, we make an explicit calculation to show that
\begin{align}
\frac12\frac{d}{dt}\|\eta\|_{L^2(\mathbb{R})}^2=&-\int(\partial_xu+\partial_x(u\eta))\eta~dx\label{e:hL2} \\
\leq&\|\partial_xu\|_{L^2(\mathbb{R})}\|\eta\|_{L^2(\mathbb{R})}
+\frac12\|\partial_xu\|_{L^\infty(\mathbb{R})}\|\eta\|_{L^2(\mathbb{R})}^2, \notag \\
\frac12\frac{d}{dt}\|u\|_{L^2}^2=&\int(\mathcal{H}\eta-\gamma^2(|\partial_x|)\eta-u\partial_xu)u~dx\label{e:uL2} \\
\leq&2\|\eta\|_{L^2(\mathbb{R})}\|u\|_{L^2(\mathbb{R})}
+\|\partial_xu\|_{L^\infty(\mathbb{R})}\|u\|_{L^2(\mathbb{R})}^2 \notag
\end{align}
over the interval $(0,t_0)$. Here the first equalities of \eqref{e:hL2} and \eqref{e:uL2} use \eqref{E:main'}. Adding \eqref{e:ej} through \eqref{e:uL2}, we deduce that 
\[
\frac{dE_k}{dt}\leq C E_k^2
\]
for any integer $k\geq 2$ over the interval $(0,t_0)$, for some constant $C>0$ independent of $\eta$ and $u$. We then deduce \eqref{I:energy} because it invites a solution until the time $t_1=(CE_k(0))^{-1}$. Moreover, we deduce \eqref{I:norm} because $E_k(t)$ is equivalent to $\|\eta\|_{H^k(\mathbb{R})}(t)+\|u\|_{H^{k+1/2}(\mathbb{R})}(t)$. This completes the proof. 
\end{proof}


\section{Collision of $(2,-)$ and $(0,+)$}\label{sec:2,0}

Throughout the section, we employ the notation in Section~\ref{sec:MI} and Section~\ref{sec:HF}. In particular, we use the notation of \eqref{def:Lxa}, and \eqref{def:cn}, \eqref{def:Mn} for simplicity of notation. Let 
\[
i\omega(2+\xi_0,-)=i\omega(0+\xi_0,+)=:i\omega_0
\] 
for some $\xi_0\in(0,1/2]$ denote a nonzero and purely imaginary, colliding $L^2(\mathbb{T})\times L^2(\mathbb{T})$ eigenvalue of $\L(\xi_0,0)$. For $\xi$, $a\in\mathbb{R}$ and $|\xi|$, $|a|$ sufficiently small, recall that the spectrum of $\L(\xi_0+\xi,a)$ contains two eigenvalues in the vicinity of $i\omega_0$ in $\mathbb{C}$, and 
\begin{equation}\label{def:q20}
\begin{aligned}
\q_{1}(z)=&\begin{pmatrix} 1 \\ c_{2,\xi} \end{pmatrix}e^{2iz}+a\mathbf{q}_{2,1}e^{3iz}+a\mathbf{q}_{2,-1}e^{iz} +O(\xi^2a+a^2), \\
\q_{2}(z)=&\begin{pmatrix} 1 \\ -c_{0,\xi} \end{pmatrix}+a\mathbf{q}_{0,1}e^{iz}+a\mathbf{q}_{0,-1}e^{-iz}+O(\xi^2a+a^2)  
\end{aligned}
\end{equation}
are the associated eigenfunctions as $\xi$, $a\to 0$, where 
\begin{equation}\label{def:a20}
\begin{aligned}
\mathbf{q}_{2,\pm1}=&\frac12\frac{2\pm1+\x}{(\omega_0-c_0(2\pm1+\x))^2-c_{2\pm 1,0}^2(2\pm1+\x)^2}  \\
&\times\begin{pmatrix} c_0(2\pm1+\x)(c_0+2c_{2,0})-\omega_0(c_0+c_{2,0}) \\ 
(2\pm1+\x)(c_0^2c_{2,0}+c_{2\pm1,0}^2(c_0+c_{2,0}))-\omega_0 c_0c_{2,0}\end{pmatrix}, \\
\mathbf{q}_{0,\pm1}=&\frac12\frac{\pm1+\x}{(\omega_0-c_0(\pm1+\x))^2-c_{\pm 1,0}^2(\pm1+\x)^2} \\
&\times\begin{pmatrix} c_0(\pm1+\x)(c_0-2c_{0,0})-\omega_0(c_0-c_{0,0})\\ 
(\pm1+\x)(-c_0^2c_{0,0}+c_{\pm1,0}^2(c_0-c_{0,0}))+\omega_0c_0c_{0,0}\end{pmatrix};
\end{aligned}
\end{equation}
see \eqref{def:Q1} and \eqref{def:Q2}. In Section~\ref{ss:perturbation4}, we calculated \eqref{def:L4} and \eqref{def:I4} up to the order of $a$ as $\xi$, $a\to 0$. Here we take matters further and calculate terms of orders $\xi^2a$ and $a^2$. One may explore other collisions in like manner.

For $\xi=0$ for $a\in\mathbb{R}$ and $|a|$ sufficiently small, it turns out that terms of order $a^2$ in the eigenfunction expansion do not contribute to the spectral instability. Hence we may neglect them in \eqref{def:q20}. 

We begin by calculating 
\begin{align*}
\l \q_{1},\q_{1}\r=&1+c_{2,\xi}^2
+a^2 \mathbf{q}_{2,1}\cdot\mathbf{q}_{2,1}+a^2\mathbf{q}_{2,-1}\cdot\mathbf{q}_{2,-1}+O(\xi^2 a+\xi a^2+a^3), \\
\l\q_{2},\q_{2}\r=&1+c_{0,\xi}^2
+a^2\mathbf{q}_{0,1}\cdot\mathbf{q}_{0,1}+a^2\mathbf{q}_{0,-1}\cdot\mathbf{q}_{0,-1}+O(\xi^2 a+\xi a^2+a^3)
\intertext{and}
\l \q_{1},\q_{2}\r=&\l \q_{2},\q_{1}\r=a^2\mathbf{q}_{2,-1}\cdot\mathbf{q}_{0,+1}+O(\xi^2 a+\xi a^2+a^3)
\end{align*}
as $\xi,a\to 0$, where $c_{n,\xi}$ is in \eqref{def:cn}, 
\begin{align*}
\mathbf{q}_{2,\pm1}\cdot \mathbf{q}_{2,\pm1}=&
\frac14\Big(\frac{\x+2\pm1}{(\omega_0-c_0(\x+2\pm1))^2-c_{2\pm1,0}^2(\x+2\pm1)^2}\Big)^2 \\
&\times((d_{2,1}\omega_0-c_0d_{2,2}(\x+2\pm1))^2 \\
&\qquad+(c_0c_{2,0}\omega_0-(c_0^2c_{2,0}+d_{2,1}c_{2\pm1,0}^2)(\x+2\pm1))^2),
\intertext{and $d_{2,1}=c_0+c_{2,0}$, $d_{2,2}=c_0+2c_{2,0}$. Moreover,}
\mathbf{q}_{0,\pm1}\cdot \mathbf{q}_{0,\pm1}=&
\frac14\Big(\frac{\x\pm1}{(\omega_0 - c_0(\x\pm1))^2-c_{\pm1,0}^2 (\x\pm1)^2}\Big)^2 \\
&\times ((d_{0,1} \omega_0-c_0d_{0,2} (\x\pm1))^2  \\
&\qquad+ (c_0 c_{0,0} \omega_0-(c_0^2 c_{0,0}-d_{0,1} c_{\pm1,0}^2) (\x\pm1))^2),
\end{align*}
and $d_{0,1}=c_0-c_{0,0}$, $d_{0,2}=c_0-2c_{0,0}$. We merely pause to remark that $\mathbf{q}_{2,-1}\cdot\mathbf{q}_{0,1}$ is real valued. The exact formula is lengthy and tedious. It does not influence the result. Hence we omit the detail. A straightforward calculation reveals that
\begin{align*}
\frac{\l\q_{1},\q_{1}\r}{\l \q_{1},\q_{1}\r}&=\frac{\l \q_{2},\q_{2}\r}{\l \q_{2},\q_{2}\r}=1
\intertext{and}
\frac{\l \q_{1},\q_{2}\r}{\l \q_{1},\q_{1}\r}&=a^2\frac{\mathbf{q}_{2,-1}\cdot\mathbf{q}_{0,1}}{1+c^2_{2,\xi}}
+O(\xi^2 a+\xi a^2+a^3), \\
\frac{\l\q_{2},\q_{1}\r}{\l \q_{2},\q_{2}\r}&=a^2\frac{\mathbf{q}_{2,-1}\cdot\mathbf{q}_{0,1}}{1+c^2_{0,\xi}}
+O(\xi^2 a+\xi a^2+a^3)
\end{align*}
as $\xi$, $a\to 0$. Note that terms of order $a^2$ are real valued.

To proceed, we use \eqref{def:Lx} and \eqref{E:hu-small} to write 
\begin{align*}
\L(\x+\xi,a)=&\L(\x+\xi,0) \\
&-ae^{-i\x z}\partial_z(\mathbf{C}_0\cos z)e^{i\x z}-i\xi ae^{-i\x z}(\mathbf{C}_0\cos z)e^{i\x z}\\
&-a^2e^{-i\x z}\partial_z\begin{pmatrix}h_0-\frac12 & h_0\\ 0 & h_0-\frac12 \end{pmatrix}e^{i\x z} \\
&-a^2e^{-i\x z}\partial_z\begin{pmatrix}h_2-\frac12 & h_2\\ 0 & h_2-\frac12 \end{pmatrix}\cos 2z e^{i\x z}
+O(\xi^2 a+\xi a^2+a^3)
\end{align*}
as $\xi$, $a\to 0$, where $\mathbf{C}_0$ is in \eqref{def:Mn}, $h_0$ and $h_2$ are in \eqref{def:h02}. It is then straightforward to verify that 
\begin{align*}
\L(\x+\xi,a)\begin{pmatrix} \zeta \\v \end{pmatrix}e^{inz} &
= i(n+\x+\xi)\begin{pmatrix}c_0\zeta-v \\ c_0v-c_{n,\xi}^2\zeta \end{pmatrix}e^{inz} \\
&-\frac{1}{2}ia\begin{pmatrix} c_0\zeta+v\\ c_0v\end{pmatrix}
((n+1+\x+\xi)e^{i(n+1)z}+(n-1+\x+\xi)e^{i(n-1)z})\\
&-ia^2(n+\xi)\begin{pmatrix} h_0-\frac12 h +h_0v\\ h_0-\frac12v\end{pmatrix}e^{inz}\\
&-\frac{1}{2}ia^2\begin{pmatrix} h_2-\frac12h+h_2v\\ h_2-\frac12v\end{pmatrix}
((n+2+\x)e^{i(n+2)z}+(n-2+\x)e^{i(n-2)z})\\
&+O(\xi^2 a+\xi a^2+a^3)
\end{align*}
as $\xi$, $a\to0$ for any constants $\zeta$, $v$ and $n\in\mathbb{Z}$.

We use the above formula for $\L(\x+\xi,a)$ and \eqref{def:q20}, and we make a lengthy and complicated, but explicit, calculation to show that
\begin{align*}
\L\q_{1}=&i\omega(2+\x+\xi,-)\begin{pmatrix} 1 \\ c_{2,\xi} \end{pmatrix}e^{2iz} \\
&+ia(3+\x+\xi)\mathbf{C}_{3,\xi}\mathbf{q}_{2,1}e^{3iz}
+ia(1+\x+\xi)\mathbf{C}_{1,\xi}\mathbf{q}_{2,-1} e^{iz}\\
&-\frac12 ia\begin{pmatrix}c_0+c_{2,\xi}\\ c_0c_{2,\xi}\end{pmatrix} ((3+\x+\xi)e^{3iz}+(1+\x+\xi)e^{iz})\\
&-\frac12ia^2\begin{pmatrix} h_2-\frac12+h_2c_{2,\xi}\\ h_2-\frac12 c_{2,\xi}\end{pmatrix}((4+\x)e^{4iz}+\x)\\
&-ia^2\begin{pmatrix} h_0-\frac12+h_0c_{2,\xi}\\ h_0-\frac12 c_{2,\xi}\end{pmatrix}(2+\x)e^{2iz}\\ 
&-\frac 12 ia^2\mathbf{C}_0\mathbf{q}_{2,1}((4+\x+\xi)e^{4iz}+(2+\x+\xi)e^{2iz})\\
&-\frac 12 ia^2\mathbf{C}_0\mathbf{q}_{2,-1}((2+\x+\xi)e^{2iz}+\x+\xi)+O(\xi^2a+\xi a^2+a^3)\\ 
\intertext{as $\xi$, $a\to 0$, where $c_{n,\xi}$ is in \eqref{def:cn}, $h_0$ and $h_2$ is in \eqref{def:h02}. Moreover,}
\L\q_{2}=& i\omega(\x+\xi,+)\begin{pmatrix} 1 \\ -c_{0,\xi} \end{pmatrix}\\
&+ia(1+\x+\xi)\mathbf{C}_{1,\xi}\mathbf{q}_{0,1} e^{iz}
+ia(-1+\x+\xi)\mathbf{C}_{-1,\xi}\mathbf{q}_{0,-1} e^{-iz}\\
&-\frac12 ia\begin{pmatrix}c_0-c_{0,\xi}\\ -c_0c_{0,\xi} \end{pmatrix}((1+\x+\xi)e^{iz}+(-1+\x+\xi)e^{-iz})\\
&-\frac12 ia^2\begin{pmatrix} h_2-\frac12-h_2c_{0,\xi}\\- h_2-\frac12 c_{0,\xi}\end{pmatrix}
((2+\x)e^{2iz}+(-2+\x)e^{-2iz})\\
&- ia^2\begin{pmatrix} h_0-\frac12-h_0c_{0,\xi}\\ -h_0-\frac12 c_{0,\xi}\end{pmatrix}\x\\
&-\frac 12 ia^2\mathbf{C}_0\mathbf{q}_{0,1}((2+\x+\xi)e^{2iz}+\x+\xi)\\
&-\frac 12 ia^2\mathbf{C}_0\mathbf{q}_{0,-1}(\x+\xi+(-2+\x+\xi)e^{-2iz})+O(\xi^2a+\xi a^2+a^3)
\end{align*}
as $\xi$, $a \to 0$, where $c_{n,\xi}$ is in \eqref{def:cn}, $h_0$ and $h_2$ are in \eqref{def:h02}. The exact formulae of $\mathbf{C}_{2\pm1,\xi}\mathbf{q}_{2,\pm1}$, $\mathbf{C}_{0\pm1,\xi}\mathbf{q}_{0,\pm1}$, and $\mathbf{C}_0\mathbf{q}_{2,\pm1}$, $\mathbf{C}_0\mathbf{q}_{0,\pm1}$ are lengthy and tedious. They do not influence the result. But we include them for completeness:
\begin{align*}
\mathbf{C}_{2\pm1,\xi}\mathbf{q}_{2,\pm1}=&
\frac12\frac{\x+2\pm 1}{(\omega_0-c_0(\x+2\pm 1))^2-c_{2\pm 1,0}^2(\x+2\pm1)^2} \\
&\times\begin{pmatrix} (\x+2\pm1)(c_0^2d_{2,2}-c_0^2c_{2,0}-c_{2\pm1,0}^2d_{2,1})-\omega_0 c_0^2  \\ 
c_0(\x+2\pm1)(c_0^2c_{2,0}+c_{2\pm1,0}^2d_{2,1}-c_{n,\xi}^2d_{2,2})+\omega_0 (c_{n,\xi}^2d_{2,1}-c_0^2 c_{2,0})\end{pmatrix},
\end{align*}
where $d_{2,1}=c_0+c_{2,0}$ and $d_{2,2}=c_0+2c_{2,0}$, and
\begin{align*}
\mathbf{C}_{\pm1,\xi}\mathbf{q}_{0,\pm1}=&
\frac12\frac{\x\pm 1}{(\omega_0-c_0(\x\pm 1))^2-c_{\pm 1,0}^2(\x\pm1)^2} \\
&\times\begin{pmatrix} (\x\pm1)(c_0^2d_{0,2}+c_0^2c_{0,0}-c_{\pm1,0}^2d_{0,1})-\omega_0 c_0^2  \\ 
c_0(\x\pm1)(-c_0^2c_{0,0}+c_{\pm1,0}^2d_{0,1}-c_{n,\xi}^2d_{0,2})+\omega_0(c_{n,\xi}^2d_{0,1}+c_0^2c_{0,0})\end{pmatrix},
\end{align*}
where $d_{0,1}=c_0-c_{0,0}$ and $d_{0,2}=c_0-2c_{0,0}$. Moreover, 
\begin{align*}
\mathbf{C}_0\mathbf{q}_{2,\pm1}=&
\frac12\frac{\x+2\pm 1}{(\omega_0-c_0(\x+2\pm 1))^2-c_{2\pm 1,0}^2(\x+2\pm1)^2} \\
&\times\begin{pmatrix} (\x+2\pm1)(c_0^2d_{2,2}+c_0^2c_{2,0}+c_{2\pm1,0}^2d_{2,1})-\omega_0 c_0d_{2,2} \\ 
c_0(\x+2\pm1)(c_0^2c_{2,0}+c_{2\pm1,0}^2d_{2,1})-\omega_0c_0^2c_{2,0})\end{pmatrix}
\end{align*}
and 
\begin{align*}
\mathbf{C}_0\mathbf{q}_{0,\pm1}=& 
\frac12\frac{\x\pm 1}{(\omega_0-c(\k)(\x\pm 1))^2-c_{\pm 1,0}^2(\x\pm1)^2} \\
&\times\begin{pmatrix} (\x\pm1)(c_0^2d_{0,2}-c_0^2c_{0,0}+c_{\pm1,0}^2d_{0,1})-\omega_0 c_0d_{0,2}
\\ c_0(\x\pm1)(-c_0^2c_{0,0}+c_{2\pm1,0}^2d_{0,1})+\omega_0 c_0^2 c_{0,0}\end{pmatrix}.
\end{align*}

Continuing, we take the $L^2(\mathbb{T})\times L^2(\mathbb{T})$ inner products of the above and \eqref{def:q20}, and we make a lengthy and complicated, but explicit, calculations to show that
\begin{align*}
\frac{\l \L\q_{1},\q_{1}\r}{\l \q_{1},\q_{1}\r}=i\omega(2+\x+&\xi,-) \\
+\frac{ia^2}{1+c_{2,\xi}^2} \Big(&-(2+\x)\Big(h_0-\frac12(1+c_{2,\xi}^2)+h_0 c_{2,\xi}\Big)\\
&+(3+\x+\xi)\Big(\mathbf{C}_{3,\xi}\mathbf{q}_{2,1}\cdot \mathbf{q}_{2,1}
-\frac12 \begin{pmatrix}c_0+c_{2,\xi}\\ c_0c_{2,\xi} \end{pmatrix}\cdot \mathbf{q}_{2,1}\Big) \\
&+(1+\x+\xi)\Big(\mathbf{C}_{1,\xi}\mathbf{q}_{2,-1}\cdot \mathbf{q}_{2,-1}
-\frac12\begin{pmatrix}c_0+c_{2,\xi}\\ c_0c_{2,\xi} \end{pmatrix}\cdot \mathbf{q}_{2,-1} \Big) \\
&-\frac12(2+\x+\xi)\mathbf{C}_0(\mathbf{q}_{2,1}+\mathbf{q}_{2,-1})
\cdot\begin{pmatrix}1\\ c_{2,\xi} \end{pmatrix}\\
&-\omega(2+\x+\xi,-)(\mathbf{q}_{2,1}\cdot \mathbf{q}_{2,1}+\mathbf{q}_{2,-1}\cdot \mathbf{q}_{2,-1})\Big)\\
+O(\xi^2a+\xi&a^2+a^3) \\
=:i\omega(2+\x+&\xi,-)+ia^2L_{1,1}+O(\xi^2a+\xi a^2+a^3)
\end{align*}
as $\xi$, $a\to 0$, where $c_{n,\xi}$ is in \eqref{def:cn} and $h_0$ is in \eqref{def:h02}. Moreover,
\begin{align*}
\frac{\l \L\q_1,\q_2\r}{\l \q_1,\q_1\r}
=&\frac{ia^2}{1+c_{2,\xi}^2}\Big(\mathbf{C}_{1,\xi}\mathbf{q}_{2,-1}\cdot\mathbf{q}_{0,1}
-\frac12\x\Big(h_2-\frac12(1-c_{0,\xi}c_{2,\xi})+h_2 c_{2,\xi}\Big)\Big)\\
&+O(\xi^2a+\xi a^2+a^3)\\
=:&ia^2L_{1,2}+O(\xi^2a+\xi a^2+a^3), \\
\frac{\l \L\q_2,\q_1\r}{\l \q_2,\q_2\r} 
=&\frac{ia^2}{1+c_{0,\xi}^2}\Big((1+\x+\xi)\Big(\mathbf{C}_{1,\xi}\mathbf{q}_{0,1}\cdot \mathbf{q}_{2,-1}
-\frac12 \begin{pmatrix}c_0-c_{0,\xi}\\ -c_0 c_{0,\xi} \end{pmatrix}\cdot\mathbf{q}_{2,-1} \Big)\\
&\hspace*{35pt}-\frac12 (2+\x)(h_2-\frac12(1-c_{0,\xi}c_{2,\xi})-h_2c_{0,\xi})\Big)\\
&+O(a^3+\xi a^2) \\
=:&ia^2L_{2,1}+O(\xi^2a+\xi a^2+a^3),
\end{align*}
where $c_{n,\xi}$ is in \eqref{def:cn} and $h_2$ is in \eqref{def:h02}, and
\begin{align*}
\frac{\l \L\q_2,\q_2\r}{\l \q_2,\q_2\r} =i\omega(\xi+\x,+&)\\
+\frac{ia^2}{1+c_{0,\xi}^2}\Big(&-\x(h_0-\frac12(1+c_{0,\xi}^2)-h_0 c_{0,\xi})\\
&+(1+\x+\xi)\Big(\mathbf{C}_{1,\xi}\mathbf{q}_{0,1}\cdot \mathbf{q}_{0,1}
-\frac12 \begin{pmatrix}c_0-c_{0,\xi}\\ -c_0 c_{0,\xi} \end{pmatrix}\cdot \mathbf{q}_{0,1}\Big) \\
&+(-1+\x+\xi)\Big(\mathbf{C}_{-1,\xi}\mathbf{q}_{0,-1}\cdot \mathbf{q}_{0,-1}
-\frac12 \begin{pmatrix}c_0-c_{0,\xi}\\ -c_0 c_{0,\xi} \end{pmatrix}\cdot \mathbf{q}_{0,-1}\Big) \\
&-\frac12(\x+\xi) \mathbf{C}_0(\mathbf{q}_{0,1}+\mathbf{q}_{0,-1})\cdot \begin{pmatrix}1\\ -c_{0,\xi} \end{pmatrix}\\
&-\omega(\x+\xi,+)(\mathbf{q}_{0,1}\cdot \mathbf{q}_{0,1}+\mathbf{q}_{0,-1}\cdot \mathbf{q}_{0,-1})\Big) \\
+O(\xi^2a+\xi&a^2+a^3) \\
=:i\omega(\x+\xi,+&)+ia^2L_{2,2}+O(\xi^2a+\xi a^2+a^3)
\end{align*}
as $\xi, a\to 0$, where $c_{n,\xi}$ is in \eqref{def:cn} and $h_0$ is in \eqref{def:h02}.

Together, \eqref{def:L4} and \eqref{def:I4} become
\begin{align*}
\mathbf{L}(\xi,a)=&\begin{pmatrix} i\omega(2+\x+\xi,-) & 0 \\ 0 & i\omega(\x+\xi,+) \end{pmatrix} \\
&+ia^2\begin{pmatrix} L_{1,1} & L_{1,2} \\ L_{2,1} & L_{2,2} \end{pmatrix}+O(\xi^2a+\xi a^2+a^3)
\intertext{and} 
\mathbf{I}(\xi,a)=&\mathbf{I}
+a^2\begin{pmatrix} 0 & \dfrac{\mathbf{q}_{2,-1}\mathbf{q}_{0,1}}{1+c^2_{2,\xi}} \\ 
\dfrac{\mathbf{q}_{2,-1}\mathbf{q}_{0,1}}{1+c^2_{0,\xi}} & 0 \end{pmatrix}+O(\xi^2a+\xi a^2+a^3)
\end{align*}
as $\xi$, $a\to 0$, where $L_{k,\ell}$ for $k$, $\ell=1$, $2$ are found above, $c_{n,\xi}$ is in \eqref{def:cn}, $\mathbf{q}_{n,\pm1}$ is in \eqref{def:Q1} and \eqref{def:Q2}, and $\mathbf{I}$ means the $2\times 2$ identity matrix. Note that the coefficient matrix of $\mathbf{I}(\xi,a)$ is real valued at the order of $a^2$. Hence for $\xi$, $a\in\mathbb{R}$ and $|\xi|$, $|a|$ sufficiently small, the roots of $\det(\mathbf{L}-\lambda\mathbf{I})(\xi,a)$ are purely imaginary up to terms of orders $\xi a$ and $a^2$, implying the spectral stability in the vicinity of $i\omega_0$ in $\mathbb{C}$. The result seems to agree with that in \cite{AN2014}, for instance, from a numerical computation for the water wave problem.

\section{Ill-posedness for \eqref{E:BW1}}\label{sec:BW1}

For any $b\in\mathbb{R}$, note that $\eta=b$ makes a constant traveling wave of a Boussinesq-Whitham equation, after normalization of parameters,
\begin{equation}\label{E:BW1'}
\partial^2_t \eta=\c^2(|\partial_x|)\partial_x^2\eta+\partial_x^2(\eta^2),
\end{equation}
where $\c$ is in \eqref{def:c}. Linearizing \eqref{E:BW1'} about $\eta=b$ in the coordinate frame moving at the speed $c$, and seeking a solution of the form $e^{\lambda t}\zeta(x)$, $\lambda\in\mathbb{C}$, we arrive at 
\[
\lambda^2\zeta-2c\lambda\partial_x\zeta+c^2\partial_x^2\zeta=(\c^2(|\partial_x|)+2b)\partial_x^2\zeta.
\] 
A straightforward calculation reveals that it possesses infinitely many eigenvalues and eigenfunctions:
\[
\lambda(n+\xi,\pm)=in(c\pm\sqrt{\c^2(n+\xi)+2b})\quad \text{and}\quad \zeta(n+\xi)(z)=e^{i(n+\xi)z}
\]
for $n\in\mathbb{Z}$ and $\xi\in[0,1/2]$. If $b>0$ then $\lambda(n+\xi,\pm)$ are purely imaginary for any $n\in\mathbb{Z}$ and $\xi\in[0,1/2]$, implying spectral stability. If $b<0$, on the other hand, then since $\c^2$ decreases to zero monotonically away from the origin, it is possible to find $n+\xi\in\mathbb{R}$ sufficiently large such that $\lambda(n+\xi,-)$ is real and positive, implying spectral instability. In other words, a negative constant solution of \eqref{E:BW1'} is spectrally unstable however small it is. This is physically unrealistic. Nevertheless, in \cite{DT}, the spectral instability in \eqref{E:BW1'} away from the origin in $\mathbb{C}$ was argued by a numerical approximation of the spectrum of the linearization about a periodic wave train but for $b>0$.


To compare, for any $\k>0$, for any $b_1$, $b_2\in\mathbb{R}$ and $|b_1|$, $|b_2|$ sufficiently small, the linearization of \eqref{E:main}-\eqref{def:c} about $\eta=\eta_0(\k,b_1,b_2)$, $u=u_0(\k,b_1,b_2)$ and $c=c_0(\k,b_1,b_2)$ possesses infinitely many eigenvalues and eigenfunctions
\[
\lambda(n+\xi,\pm;b_1,b_2)=i(n+\xi)(c_0-u_0\pm \sqrt{1+h_0}\c(\k(n+\xi)))
\]
and
\[
\e(n+\xi,\pm;b_1,b_2)(z)=\begin{pmatrix}\sqrt{1+h_0}\\ \mp \c(n+\xi)\end{pmatrix}e^{i(n+\xi)z}
\]
for $n\in\mathbb{Z}$ and $\xi\in[0,1/2]$, where $\eta_0$, $u_0$, and $c_0$ are in \eqref{E:hu0}. For $b_1=b_2=0$, note that $\lambda(n+\xi,\pm;0,0)$ and $\e(n+\xi,\pm;0,0)$ agree with \eqref{def:eigen}. Since $1+h_0>0$ for any $b_1$, $b_2\in\mathbb{R}$ and $|b_1|$, $|b_2|$ sufficiently small by \eqref{E:h0}, it follows that $\lambda(n+\xi,\pm;b_1,b_2)$ lies on the imaginary axis for any $n\in \mathbb{Z}$ and $\xi \in [0,1/2]$ for any $b_1$, $b_2\in\mathbb{R}$ and $|b_1|$, $|b_2|$ sufficiently small. In other words, a sufficiently small, constant solution of \eqref{E:main}-\eqref{def:c} is spectrally stable.

\end{appendix}

\bibliographystyle{amsalpha}
\bibliography{stability.bib}

\end{document}